\documentclass[11pt]{amsart}
\pdfoutput=1
\usepackage{amsmath,amssymb,amsthm,amsfonts}
\usepackage{mathtools}
\usepackage{pgf}

\usepackage[english]{babel}
\usepackage{comment}
\usepackage{hyperref}
\usepackage{bbm}
\usepackage{bm}
\usepackage{mathrsfs}
\usepackage{verbatim}

\usepackage{color}
\usepackage{graphicx}
\usepackage{tikz}

\usepackage{xspace}

\usepackage{hhline}
\usepackage{scalerel}
\usepackage{tikz-cd}
\usepackage{tikz-3dplot}
\usetikzlibrary{calc}
\usetikzlibrary{arrows}
\usetikzlibrary{shapes}
\usetikzlibrary{patterns}
\usetikzlibrary{positioning}
\usetikzlibrary{arrows.meta}
\usetikzlibrary{decorations.markings}

\usepackage{enumitem}
\usepackage[arrow]{xy}
\usepackage{diagbox}
\usepackage{subfig}
\usepackage{arcs}
\usepackage{xcolor}
\usepackage{tabu}
\usepackage{booktabs}

\usepackage{manfnt}  

\usepackage{upgreek} 

\usepackage[margin=1.0in]{geometry}

\usepackage{letltxmacro}
\usepackage{thmtools,etoolbox}

\usepackage{slashbox}

\def\myarabic#1{\normalfont(\roman{#1})}
\newlist{theoremlist}{enumerate}{1}
\setlist[theoremlist]{label=\myarabic{theoremlisti},ref={\myarabic{theoremlisti}},itemindent=0pt,labelindent=0pt,
leftmargin=*,noitemsep}

\makeatletter
\renewcommand{\p@theoremlisti}{\perh@ps{\thetheorem}}
\protected\def\perh@ps#1#2{\textup{#1#2}}
\newcommand{\itemrefperh@ps}[2]{\textup{#2}}
\newcommand{\itemref}[1]{\begingroup\let\perh@ps\itemrefperh@ps\ref{#1}\endgroup}
\makeatother

\usepackage{nameref,hyperref}
\usepackage[capitalize]{cleveref}

\theoremstyle{plain}
\newtheorem{theorem}{Theorem}[section]
\newtheorem{proposition}[theorem]{Proposition}
\newtheorem{lemma}[theorem]{Lemma}
\newtheorem{corollary}[theorem]{Corollary}
\newtheorem{conjecture}[theorem]{Conjecture}

\theoremstyle{definition}
\newtheorem{definition}[theorem]{Definition}
\newtheorem{remark}[theorem]{Remark}
\newtheorem{example}[theorem]{Example}
\newtheorem{notation}[theorem]{Notation}

\crefname{figure}{Figure}{Figures}

\addtotheorempostheadhook[theorem]{\crefalias{theoremlisti}{theorem}}
\addtotheorempostheadhook[lemma]{\crefalias{theoremlisti}{lemma}}
\addtotheorempostheadhook[proposition]{\crefalias{theoremlisti}{proposition}}
\addtotheorempostheadhook[corollary]{\crefalias{theoremlisti}{corollary}}

\def\Z{\mathbb{Z}}
\def\R{\mathbb{R}}

\def\v{{\bm{v}}}
\def\t{{\bm{t}}}
\def\x{{\bm{x}}}
\def\y{{\bm{y}}}
\def\p{{\bm{p}}}
\def\q{{\bm{q}}}
\def\u{{\bm{u}}}

\def\Pl{\Delta}
\def\A{\mathcal{A}}
\def\Acal{\mathcal{A}}
\def\V{\mathcal{V}}

\def\h{{\bm{h}}}
\def\ddelta{\bm{\delta}}

\def\ggamma{\bm{\gamma}}
\def\e{{\bm{e}}}
\def\a{{\bm{a}}}
\def\conv{\mathrm{conv}}
\def\Conv{\mathrm{conv}}

\def\Tr{\tau}

\def\Zon{\mathcal{Z}}

\def\ZV{\Zon_\V}

\def\cube{\mbox{\,\mancube\,}}
\def\Vol{\mathrm{Vol}}
\def\level{\mathrm{level}}
\def\Level{\operatorname{Level}}
\newcommand*{\parallelogram}{%
  \rlap{\rotatebox{-30}{\rule[.05ex]{.4pt}{.77em}}}%
  \kern.04em%
  \rlap{\kern.36em\raisebox{0.649519052835em}{\rule{.6em}{.4pt}}}%
  \rule{.6em}{.4pt}\kern-.04em%
  \rotatebox{-30}{\rule[.05ex]{.4pt}{.77em}}}

\def\GKZ{{\operatorname{GKZ}}}

\def\vert{\mathrm{vert}}
\def\vertHSP{\wh\vert}
\def\vertHSP{\widehat{\vert}}
\def\vertGKZ{\vert^\GKZ}
\def\vertFib{\vert^\fiber}

\def\HSP{\widehat\Upsigma}
\def\Sigma{\HSP}
\def\PT{\operatorname{PT}}
\def\fiber{{\operatorname{fib}}}
\def\FibPoly{\Upsigma^\fiber}
\def\SigmaGKZ{\Upsigma_\A^{\GKZ}}
\def\pito{\overset \pi \to}

\def\VC{\V}
\def\Bcal{\mathcal{B}}
\def\Ccal{\mathcal{C}}
\def\Mcal{\mathcal{M}}
\def\Tcal{\mathcal{T}}

\def\Delabs_#1{\Vol(#1)}
\def\suppX{{\underline{X}}}
\def\suppC{{\underline{C}}}

\def\Tiling{\Tcal}

\def\tile{\Pi}

\newcommand\shifteq{\stackrel{\mathclap{\normalfont\mbox{\scalebox{0.6}{shift}}}}{=\joinrel=}}

\newcommand\summa[2]{\sum_{\substack{#1 \\ #2}}}
\def\sumlevel#1{\summa{\tile_{A,B}\in\Tiling}{|A|=#1}}
\def\sumcond#1{\summa{\tile_{A,B}\in\Tiling}{#1}}

\def\Voldd{\Vol^{d-1}}
\def\Vold{\Vol^d}

\def\VolX(#1){\Vol^d(\tile_{#1})}
\def\VolB{\VolX(B)}
\def\VolDB{\Vol^{d-1}(\Delta_B)}

\def\WS{{\mathcal{D}}}
\def\eltminus{\setminus}

\let\subset\subseteq
\let\supset\supseteq

\def\boundnoderadius{1pt}
\def\boundnode(#1){
\draw[blue, line width=\plabiclw,fill=white] (#1) circle (\boundnoderadius); 
}

\def\boundLabel_#1{#1}

\newcommand{\Euli}{\genfrac{\langle}{\rangle}{0pt}{}}

\def\Eul(#1,#2){\Euli{#1}{#2}}

\title{Higher Secondary Polytopes and Regular Plabic Graphs}

\author{Pavel Galashin}
\address{Department of Mathematics, University of California, Los Angeles, CA 90095, USA}
\email{{\href{mailto:galashin@math.ucla.edu}{galashin@math.ucla.edu}}}

\author{Alexander Postnikov}
\address{Department of Mathematics, Massachusetts Institute of Technology, Cambridge, MA 02139, USA}
\email{{\href{mailto:apost@mit.edu}{apost@mit.edu}}}

\author{Lauren Williams}
\address{Department of Mathematics, Harvard University, Cambridge, MA 02138, USA}
\email{{\href{mailto:williams@math.harvard.edu}{williams@math.harvard.edu}}}

\date{\today}



\begin{document}

\begin{abstract}
Given a configuration $\A$ of $n$ points in $\R^{d-1}$, we introduce the \emph{higher secondary polytopes} $\Sigma_{\A,1},\dots, \Sigma_{\A,n-d}$, which have the property that 
	$\Sigma_{\A,1}$ agrees with  the secondary polytope of Gelfand--Kapranov--Zelevinsky, while the Minkowski sum of these polytopes agrees with Billera--Sturmfels' fiber zonotope associated with (a lift of) $\A$. 
	In a special case when $d=3$, we refer to our polytopes as \emph{higher associahedra}. They turn out to be related to the theory of total positivity, specifically, to certain combinatorial objects called \emph{plabic graphs}, introduced by the second author in his study of the totally positive Grassmannian.  We define a subclass of \emph{regular} plabic graphs and show that they correspond to the vertices of the higher associahedron $\Sigma_{\A,k}$, while \emph{square moves} connecting them correspond to the edges of $\Sigma_{\A,k}$. Finally we connect our polytopes to \emph{soliton graphs}, the contour plots of soliton solutions to the KP equation, which were recently studied by Kodama and the third author.  In particular, we confirm their conjecture that when the higher times evolve, soliton graphs change according to the moves for plabic graphs. 
\end{abstract}

\maketitle
\setcounter{tocdepth}{1}
\tableofcontents

\numberwithin{equation}{section}
\def\Gr{\operatorname{Gr}}
\def\Grtp{\Gr_{>0}}
\def\ngon{\text{$n$-gon}}

\def\pslope{0.6}
\def\parallelogramNew{
\begin{tikzpicture}[yscale=0.25,xscale=0.2]
\draw[line width=0.5pt] (0,0)--(\pslope,1)--(1+\pslope,1)--(1,0)--cycle;
\end{tikzpicture}\!
}

\section{Introduction}
\label{sec:intro}

Motivated by the study of discriminants, Gelfand, Kapranov, and Zelevinsky
\cite{GKZ} introduced the \emph{secondary polytope} $\SigmaGKZ$
for a configuration $\A$ of $n$ points in $\R^{d-1}$.  Vertices of this
remarkable polytope correspond to  regular triangulations of the convex
hull of $\A$, and its faces correspond to  regular polyhedral
subdivisions.  Billera and Sturmfels \cite{BS} defined a more general
notion of a \emph{fiber polytope} $\FibPoly(P\pito Q)$ 
for any linear projection $\pi:P\to Q$ of polytopes.
Secondary polytopes are exactly the fiber polytopes in the case when $P$ is a
simplex.

In this paper, we extend the notion of a secondary polytope and define the
\emph{higher secondary polytopes} $\Sigma_{\A,1}, \dots,\Sigma_{\A,n-d}$ so
that $\Sigma_{\A,1}$ coincides with the secondary polytope $\SigmaGKZ$ up to affine translation and dilation. An example of a higher secondary polytope is shown in \cref{fig:Gr_3_6}.

\begin{figure}

\rotatebox{90}{
\begin{tikzpicture}%
	[xscale=1.3,x={(0.895821cm, 0.211046cm)},
	y={(-0.170607cm, 0.975929cm)},
	z={(0.410364cm, -0.054974cm)},
	scale=2.000000,
	back/.style={dashed, line width=0.6pt},
	wrong/.style={dotted, line width=0.5pt,red},
	wrongnode/.style={scale=0.8,inner sep=-1pt,circle,draw=red,fill=red!50,thick,anchor=base,text=black},
	goodnode/.style={scale=0.8,inner sep=-1pt,circle,draw=green!25!black,fill=green!75!black,thick,anchor=base,text=black},
	edge/.style={color=blue!95!black, line width=0.9pt},
	facet/.style={fill=green,fill opacity=0.300000},
	vertex/.style={inner sep=1.5pt,circle,draw=green!25!black,fill=green!75!black,thick,anchor=base}]
\def\bot{-0.11339}
\def\bat{0.08661}
\def\top{1.86184}
\def\tap{1.66184}
\def\batsymbol{\rotatebox{-90}{\textcolor{black}{$b\strut$}}}
\def\tapsymbol{\rotatebox{-90}{\textcolor{black}{$c\strut$}}}
\def\topsymbol{\rotatebox{-90}{\textcolor{black}{$d\strut$}}}
\def\botsymbol{\rotatebox{-90}{\textcolor{black}{$a\strut$}}}
%
%
\coordinate (-0.04450, -4.15573, 0.02242) at (-0.04450, -4.15573, 0.02242);
\coordinate (-0.47100, -2.39801, 0.02242) at (-0.47100, -2.39801, 0.02242);
\coordinate (-0.34716, -3.02138, -3.40209) at (-0.34716, -3.02138, -3.40209);
\coordinate (-0.06517, -4.15257, -2.46364) at (-0.06517, -4.15257, -2.46364);
\coordinate (-0.60396, -1.90506, -1.64437) at (-0.60396, -1.90506, -1.64437);
\coordinate (2.35240, -2.00559, -1.67273) at (2.35240, -2.00559, -1.67273);
\coordinate (\bot, -1.33316, -3.40209) at (\bot, -1.33316, -3.40209);
\coordinate (\bat, -1.33316, -3.40209) at (\bat, -1.33316, -3.40209);
\coordinate (-0.27019, -0.21683, -1.64437) at (-0.27019, -0.21683, -1.64437);
\coordinate (-0.02564, -1.32583, 1.71064) at (-0.02564, -1.32583, 1.71064);
\coordinate (0.00000, 0.00000, 0.00000) at (0.00000, 0.00000, 0.00000);
\coordinate (0.91859, 0.78396, 0.00000) at (0.91859, 0.78396, 0.00000);
\coordinate (0.30684, 0.28923, -3.27002) at (0.30684, 0.28923, -3.27002);
\coordinate (0.89792, 0.78712, -2.48606) at (0.89792, 0.78712, -2.48606);
\coordinate (0.85550, -3.36878, -5.02774) at (0.85550, -3.36878, -5.02774);
\coordinate (0.22987, -2.51532, -5.02774) at (0.22987, -2.51532, -5.02774);
\coordinate (0.56363, -0.82710, -5.02774) at (0.56363, -0.82710, -5.02774);
\coordinate (1.79295, 0.24509, -3.33951) at (1.79295, 0.24509, -3.33951);
\coordinate (\top, -2.57749, 0.08499) at (\top, -2.57749, 0.08499);
\coordinate (\tap, -2.57749, 0.08499) at (\tap, -2.57749, 0.08499);
\coordinate (0.89294, -0.54187, 1.71064) at (0.89294, -0.54187, 1.71064);
\coordinate (2.01863, -3.69381, -1.67273) at (2.01863, -3.69381, -1.67273);
\coordinate (1.34759, -0.82710, -5.02774) at (1.34759, -0.82710, -5.02774);
\coordinate (1.51858, -1.39532, 1.71064) at (1.51858, -1.39532, 1.71064);
\coordinate (1.77409, -2.58482, -5.02774) at (1.77409, -2.58482, -5.02774);
\coordinate (0.82986, -4.69460, -3.31710) at (0.82986, -4.69460, -3.31710);
\coordinate (2.09561, -0.88926, 0.08499) at (2.09561, -0.88926, 0.08499);
\coordinate (1.18481, -3.08354, 1.71064) at (1.18481, -3.08354, 1.71064);
\coordinate (0.40085, -3.08354, 1.71064) at (0.40085, -3.08354, 1.71064);
\coordinate (1.44161, -4.19987, -0.04708) at (1.44161, -4.19987, -0.04708);
\coordinate (1.81362, 0.24193, -0.85345) at (1.81362, 0.24193, -0.85345);
\coordinate (0.85053, -4.69776, -0.83104) at (0.85053, -4.69776, -0.83104);
\coordinate (2.21945, -1.51263, -3.33951) at (2.21945, -1.51263, -3.33951);
\coordinate (1.74845, -3.91064, -3.31710) at (1.74845, -3.91064, -3.31710);
\draw[edge,back] (2.35240, -2.00559, -1.67273) -- (2.01863, -3.69381, -1.67273);
\draw[edge,back] (2.35240, -2.00559, -1.67273) -- (2.09561, -0.88926, 0.08499);
\draw[edge,back] (2.35240, -2.00559, -1.67273) -- (2.21945, -1.51263, -3.33951);
\draw[edge,back] (0.89792, 0.78712, -2.48606) -- (1.79295, 0.24509, -3.33951);
\draw[edge,back] (0.85550, -3.36878, -5.02774) -- (0.22987, -2.51532, -5.02774);
\draw[edge,back] (0.85550, -3.36878, -5.02774) -- (1.77409, -2.58482, -5.02774);
\draw[edge,back] (0.85550, -3.36878, -5.02774) -- (0.82986, -4.69460, -3.31710);
\draw[edge,back] (0.56363, -0.82710, -5.02774) -- (1.34759, -0.82710, -5.02774);
\draw[edge,back] (1.79295, 0.24509, -3.33951) -- (1.34759, -0.82710, -5.02774);
\draw[edge,back] (1.79295, 0.24509, -3.33951) -- (1.81362, 0.24193, -0.85345);
\draw[edge,back] (1.79295, 0.24509, -3.33951) -- (2.21945, -1.51263, -3.33951);
\draw[edge,back] (\top, -2.57749, 0.08499) -- (2.01863, -3.69381, -1.67273);
\draw[edge,back] (\top, -2.57749, 0.08499) -- (2.09561, -0.88926, 0.08499);
\draw[edge,back] (\top, -2.57749, 0.08499) -- (1.18481, -3.08354, 1.71064);
\draw[edge,back] (2.01863, -3.69381, -1.67273) -- (1.44161, -4.19987, -0.04708);
\draw[edge,back] (2.01863, -3.69381, -1.67273) -- (1.74845, -3.91064, -3.31710);
\draw[edge,back] (1.34759, -0.82710, -5.02774) -- (1.77409, -2.58482, -5.02774);
\draw[edge,back] (1.77409, -2.58482, -5.02774) -- (2.21945, -1.51263, -3.33951);
\draw[edge,back] (1.77409, -2.58482, -5.02774) -- (1.74845, -3.91064, -3.31710);
\draw[edge,back] (0.82986, -4.69460, -3.31710) -- (1.74845, -3.91064, -3.31710);

\draw[edge,back,wrong] (-0.60396, -1.90506, -1.64437) -- (\bat, -1.33316, -3.40209);
\draw[edge,back,wrong] (\bat, -1.33316, -3.40209) -- (0.22987, -2.51532, -5.02774);
\draw[edge,back,wrong] (\bat, -1.33316, -3.40209) -- (0.30684, 0.28923, -3.27002);

\draw[edge,back,wrong] (\tap, -2.57749, 0.08499) -- (1.51858, -1.39532, 1.71064);
\draw[edge,back,wrong] (\tap, -2.57749, 0.08499) -- (1.44161, -4.19987, -0.04708) ;
\draw[edge,back,wrong] (\tap, -2.57749, 0.08499) -- (2.35240, -2.00559, -1.67273);

\node[vertex] at (2.35240, -2.00559, -1.67273)     {};
\node[vertex] at (2.01863, -3.69381, -1.67273)     {};
\node[vertex] at (1.77409, -2.58482, -5.02774)     {};
\node[vertex] at (2.21945, -1.51263, -3.33951)     {};
\node[vertex] at (1.74845, -3.91064, -3.31710)     {};
\node[goodnode] at (\top, -2.57749, 0.08499)     {\topsymbol};
\node[wrongnode] at (\tap, -2.57749, 0.08499)     {\tapsymbol};
\node[vertex] at (1.79295, 0.24509, -3.33951)     {};
\node[vertex] at (0.85550, -3.36878, -5.02774)     {};
\node[vertex] at (1.34759, -0.82710, -5.02774)     {};


\node[wrongnode] at (\bat, -1.33316, -3.40209)     {\batsymbol};
\fill[facet] (-0.34716, -3.02138, -3.40209) -- (-0.60396, -1.90506, -1.64437) -- (-0.47100, -2.39801, 0.02242) -- (-0.04450, -4.15573, 0.02242) -- (-0.06517, -4.15257, -2.46364) -- cycle {};
\fill[facet] (-0.27019, -0.21683, -1.64437) -- (-0.60396, -1.90506, -1.64437) -- (-0.34716, -3.02138, -3.40209) -- (\bot, -1.33316, -3.40209) -- cycle {};
\fill[facet] (0.00000, 0.00000, 0.00000) -- (-0.27019, -0.21683, -1.64437) -- (-0.60396, -1.90506, -1.64437) -- (-0.47100, -2.39801, 0.02242) -- (-0.02564, -1.32583, 1.71064) -- cycle {};
\fill[facet] (0.89294, -0.54187, 1.71064) -- (-0.02564, -1.32583, 1.71064) -- (0.00000, 0.00000, 0.00000) -- (0.91859, 0.78396, 0.00000) -- cycle {};
\fill[facet] (0.56363, -0.82710, -5.02774) -- (\bot, -1.33316, -3.40209) -- (-0.27019, -0.21683, -1.64437) -- (0.30684, 0.28923, -3.27002) -- cycle {};
\fill[facet] (0.85053, -4.69776, -0.83104) -- (-0.04450, -4.15573, 0.02242) -- (0.40085, -3.08354, 1.71064) -- (1.18481, -3.08354, 1.71064) -- (1.44161, -4.19987, -0.04708) -- cycle {};
\fill[facet] (0.89792, 0.78712, -2.48606) -- (0.91859, 0.78396, 0.00000) -- (0.00000, 0.00000, 0.00000) -- (-0.27019, -0.21683, -1.64437) -- (0.30684, 0.28923, -3.27002) -- cycle {};
\fill[facet] (0.56363, -0.82710, -5.02774) -- (\bot, -1.33316, -3.40209) -- (-0.34716, -3.02138, -3.40209) -- (0.22987, -2.51532, -5.02774) -- cycle {};
\fill[facet] (0.85053, -4.69776, -0.83104) -- (-0.04450, -4.15573, 0.02242) -- (-0.06517, -4.15257, -2.46364) -- (0.82986, -4.69460, -3.31710) -- cycle {};
\fill[facet] (1.81362, 0.24193, -0.85345) -- (0.91859, 0.78396, 0.00000) -- (0.89294, -0.54187, 1.71064) -- (1.51858, -1.39532, 1.71064) -- (2.09561, -0.88926, 0.08499) -- cycle {};
\fill[facet] (0.40085, -3.08354, 1.71064) -- (-0.02564, -1.32583, 1.71064) -- (0.89294, -0.54187, 1.71064) -- (1.51858, -1.39532, 1.71064) -- (1.18481, -3.08354, 1.71064) -- cycle {};
\fill[facet] (0.40085, -3.08354, 1.71064) -- (-0.04450, -4.15573, 0.02242) -- (-0.47100, -2.39801, 0.02242) -- (-0.02564, -1.32583, 1.71064) -- cycle {};

\draw[edge] (-0.04450, -4.15573, 0.02242) -- (-0.47100, -2.39801, 0.02242);
\draw[edge] (-0.04450, -4.15573, 0.02242) -- (-0.06517, -4.15257, -2.46364);
\draw[edge] (-0.04450, -4.15573, 0.02242) -- (0.40085, -3.08354, 1.71064);
\draw[edge] (-0.04450, -4.15573, 0.02242) -- (0.85053, -4.69776, -0.83104);
\draw[edge] (-0.47100, -2.39801, 0.02242) -- (-0.60396, -1.90506, -1.64437);
\draw[edge] (-0.47100, -2.39801, 0.02242) -- (-0.02564, -1.32583, 1.71064);
\draw[edge] (-0.34716, -3.02138, -3.40209) -- (-0.06517, -4.15257, -2.46364);
\draw[edge] (-0.34716, -3.02138, -3.40209) -- (-0.60396, -1.90506, -1.64437);
\draw[edge] (-0.34716, -3.02138, -3.40209) -- (\bot, -1.33316, -3.40209);
\draw[edge] (-0.34716, -3.02138, -3.40209) -- (0.22987, -2.51532, -5.02774);
\draw[edge] (-0.06517, -4.15257, -2.46364) -- (0.82986, -4.69460, -3.31710);
\draw[edge] (-0.60396, -1.90506, -1.64437) -- (-0.27019, -0.21683, -1.64437);
\draw[edge] (\bot, -1.33316, -3.40209) -- (-0.27019, -0.21683, -1.64437);
\draw[edge] (\bot, -1.33316, -3.40209) -- (0.56363, -0.82710, -5.02774);
\draw[edge] (-0.27019, -0.21683, -1.64437) -- (0.00000, 0.00000, 0.00000);
\draw[edge] (-0.27019, -0.21683, -1.64437) -- (0.30684, 0.28923, -3.27002);
\draw[edge] (-0.02564, -1.32583, 1.71064) -- (0.00000, 0.00000, 0.00000);
\draw[edge] (-0.02564, -1.32583, 1.71064) -- (0.89294, -0.54187, 1.71064);
\draw[edge] (-0.02564, -1.32583, 1.71064) -- (0.40085, -3.08354, 1.71064);
\draw[edge] (0.00000, 0.00000, 0.00000) -- (0.91859, 0.78396, 0.00000);
\draw[edge] (0.91859, 0.78396, 0.00000) -- (0.89792, 0.78712, -2.48606);
\draw[edge] (0.91859, 0.78396, 0.00000) -- (0.89294, -0.54187, 1.71064);
\draw[edge] (0.91859, 0.78396, 0.00000) -- (1.81362, 0.24193, -0.85345);
\draw[edge] (0.30684, 0.28923, -3.27002) -- (0.89792, 0.78712, -2.48606);
\draw[edge] (0.30684, 0.28923, -3.27002) -- (0.56363, -0.82710, -5.02774);
\draw[edge] (0.22987, -2.51532, -5.02774) -- (0.56363, -0.82710, -5.02774);
\draw[edge] (0.89294, -0.54187, 1.71064) -- (1.51858, -1.39532, 1.71064);
\draw[edge] (1.51858, -1.39532, 1.71064) -- (2.09561, -0.88926, 0.08499);
\draw[edge] (1.51858, -1.39532, 1.71064) -- (1.18481, -3.08354, 1.71064);
\draw[edge] (0.82986, -4.69460, -3.31710) -- (0.85053, -4.69776, -0.83104);
\draw[edge] (2.09561, -0.88926, 0.08499) -- (1.81362, 0.24193, -0.85345);
\draw[edge] (1.18481, -3.08354, 1.71064) -- (0.40085, -3.08354, 1.71064);
\draw[edge] (1.18481, -3.08354, 1.71064) -- (1.44161, -4.19987, -0.04708);
\draw[edge] (1.44161, -4.19987, -0.04708) -- (0.85053, -4.69776, -0.83104);
\node[vertex] at (-0.04450, -4.15573, 0.02242)     {};
\node[vertex] at (-0.47100, -2.39801, 0.02242)     {};
\node[vertex] at (-0.34716, -3.02138, -3.40209)     {};
\node[vertex] at (-0.06517, -4.15257, -2.46364)     {};
\node[vertex] at (-0.60396, -1.90506, -1.64437)     {};
\node[goodnode] at (\bot, -1.33316, -3.40209)     {\botsymbol};
\node[vertex] at (-0.27019, -0.21683, -1.64437)     {};
\node[vertex] at (-0.02564, -1.32583, 1.71064)     {};
\node[vertex] at (0.00000, 0.00000, 0.00000)     {};
\node[vertex] at (0.91859, 0.78396, 0.00000)     {};
\node[vertex] at (0.30684, 0.28923, -3.27002)     {};
\node[vertex] at (0.89792, 0.78712, -2.48606)     {};
\node[vertex] at (0.22987, -2.51532, -5.02774)     {};
\node[vertex] at (0.56363, -0.82710, -5.02774)     {};
\node[vertex] at (0.89294, -0.54187, 1.71064)     {};
\node[vertex] at (1.51858, -1.39532, 1.71064)     {};
\node[vertex] at (0.82986, -4.69460, -3.31710)     {};
\node[vertex] at (2.09561, -0.88926, 0.08499)     {};
\node[vertex] at (1.18481, -3.08354, 1.71064)     {};
\node[vertex] at (0.40085, -3.08354, 1.71064)     {};
\node[vertex] at (1.44161, -4.19987, -0.04708)     {};
\node[vertex] at (1.81362, 0.24193, -0.85345)     {};
\node[vertex] at (0.85053, -4.69776, -0.83104)     {};

\end{tikzpicture}
}
  \caption{ \label{fig:Gr_3_6} The higher secondary polytope $\Sigma_{\A,k}$ for $n=6$, $d=3$, $k=2$, where $\A\subset \R^2$ is the set of vertices of a generic convex hexagon. Thus $\Sigma_{\A,k}$ is a \emph{higher associahedron}. The polytope  $\Sigma_{\A,k}$ has $32$ vertices, and two points in the interior of $\Sigma_{\A,k}$ (labeled by $b$ and $c$), corresponding to non-regular fine zonotopal tilings, are shown in red. The $34$ points shown in this picture correspond to the $34$ bipartite plabic graphs for $\Gr(3,6)$, and the edges connecting them represent square moves of plabic graphs. See \cref{sec:high-assoc-plab} and \cref{ex:hexagons} for more details.}
\end{figure}
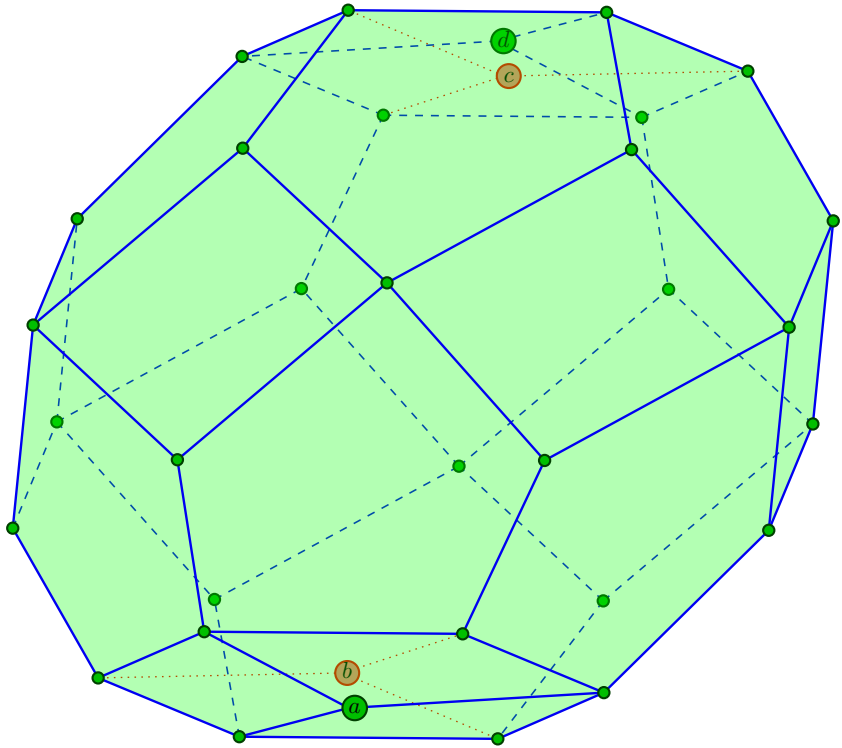

Our main motivation for the introduction of polytopes $\Sigma_{\A, k}$ comes from
total positivity.  \cite{Pos06} constructed a parametrization of 
the totally positive part $\Grtp(k,n)$ of the Grassmannian using 
\emph{plabic graphs}, which are certain graphs drawn in a disk with vertices 
colored in two colors.  These graphs have interesting combinatorial, algebraic,
and geometric features.
Remarkably, plabic graphs play a role in several different areas of mathematics
and physics: cluster algebras \cite{Scott}, quantum minors \cite{Scott2}, soliton solutions
of  Kadomtsev-Petviashvili (KP) equation \cite{KW, KW2}, scattering amplitudes 
in $\mathcal{N}=4$ supersymmetric Yang-Mills (SYM) theory \cite{abcgpt},
electrical networks~\cite{Lam}, the Ising model~\cite{Ising}, and many other areas.

Plabic graphs are also closely related to polyhedral geometry.  There are two
variations of plabic graphs: trivalent plabic graphs and bipartite plabic
graphs.  \cite{Gal} showed that trivalent plabic graphs can be identified with
sections of fine zonotopal tilings of 3-dimensional cyclic zonotopes.   A
related construction \cite{PosICM} identified trivalent plabic graphs with
$\pi$-induced subdivisions for a projection $\pi$ from the hypersimplex
$\Delta_{k,n}$ to an $n$-gon.  From both points of view, it is natural to define
the subclass of \emph{regular} plabic graphs.  Such regular plabic graphs can be
explicitly constructed from a vector $\h\in\R^n$.  Regular trivalent plabic graphs
correspond to (1) sections of regular fine zonotopal tilings 
of a 3-dimensional cyclic zonotope, 
and (2) vertices of the fiber polytope $\FibPoly(\Delta_{k,n}\pito \ngon)$ associated to a projection of a hypersimplex $\Delta_{k,n}$ to a convex $n$-gon.

While regular trivalent plabic graphs correspond to vertices of 
the fiber polytope $\FibPoly(\Delta_{k,n}\pito \ngon)$,
regular bipartite plabic graphs also correspond to vertices of certain 
polytopes, which do not fit into the framework of fiber polytopes.
In general, these polytopes are \emph{deformations} of fiber polytopes, obtained by contracting certain 
edges of fiber polytopes. These polytopes, whose vertices correspond to regular bipartite plabic graphs,
are the higher secondary polytopes $\Sigma_{\A,k}$, where $\A$
is the configuration of vertices of a convex $n$-gon.
We call these polytopes \emph{higher associahedra}, because, 
for $k=1$, they are the usual secondary polytopes of $n$-gons, 
which are exactly the celebrated associahedra of Stasheff~\cite{Tamari,Stasheff}.

The study of soliton solutions of the Kadomtsev-Petviashvili (KP) equation
also leads to regular trivalent plabic graphs \cite{KW, KW2, KK}, which
were called \emph{realizable plabic graphs} in \cite{KK}, in the case that 
$\mathcal{A} = ((\kappa_1,\kappa_1^2),\dots, (\kappa_n, \kappa_n^2)).$ 
 To understand a soliton solution
$u_A(x,y,t)$ of the KP equation coming from a point $A$ in the positive Grassmannian,
one fixes the time $t$ and plots the points
where $u_A(x,y,t)$ has a local maximum.  This gives rise to a tropical curve
in the $xy$-plane; as soliton solutions model shallow water waves, such as 
beach waves, this tropical curve shows the positions in the plane 
where the corresponding wave has a peak.  As was shown in \cite{KW, KW2},
this tropical curve is a reduced plabic graph, and hence the Pl\"ucker coordinates
naturally labeling the regions of the curve form a \emph{cluster} for the 
cluster structure on the Grassmannian; the authors moreover speculated in \cite{KW}
that when the time $t$ varies, one observes the face labels of the 
soliton graph change by \emph{cluster transformations}, see 
\cref{fig:mutation}.
\begin{figure} 
 \begin{center}
   \includegraphics[width=.45\textwidth]{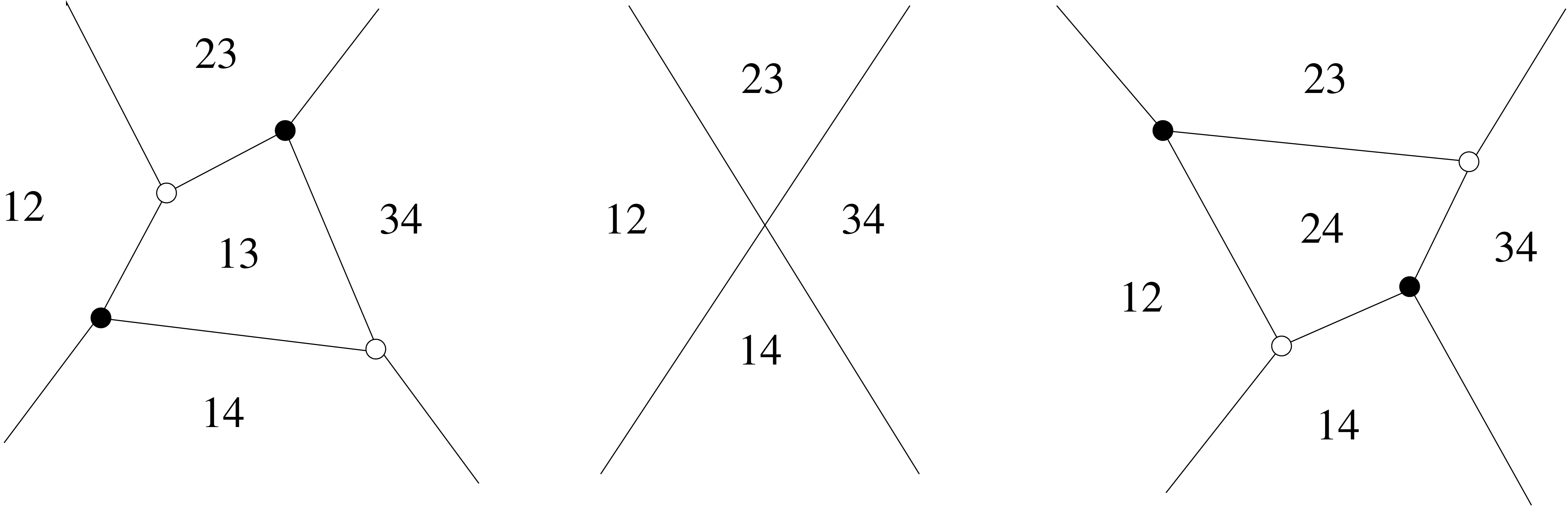}
	\caption{A contour plot coming from a point in $\Grtp(2,4)$ undergoing a cluster mutation as time varies.\label{fig:mutation}}
\end{center}
\end{figure}
We prove this conjecture using the connection between 
soliton graphs and regular plabic graphs.

\subsection*{Acknowledgements}
This project grew out of discussions
during the Fall of 2017, while all authors were in residence at the Mathematical Sciences Research Institute in Berkeley, CA.  They are grateful to MSRI for providing an ideal work environment.  
The first author is grateful to Miriam Farber for discussions regarding \cref{fig:Gr_3_6} during the development of~\cite{FG}. 
The third author would like to thank Yuji Kodama for their joint
work on KP solitons, which provided part of the motivation for this project.
This work was 
 partially supported 
by the National Science
Foundation under Grant
No. DMS-1764370, No. DMS-1440140, 
No. DMS-1600447, and No. DMS-1854512.
Any opinions, findings
and conclusions or recommendations expressed in this material are those of
the authors and do not necessarily reflect the views of the National
Science Foundation.

We now discuss our constructions and results in more detail. 

\section{Main results}\label{sec:main-results}
\subsection{Background on secondary and fiber polytopes}
Let $\A=(\a_1,\dots,\a_n)$ be a configuration of $n$ points in $\R^{d-1}$, and let $Q\subset\R^{d-1}$ be the convex hull of $\A$. We assume that the points in $\A$ affinely span $\R^{d-1}$. An \emph{$\A$-triangulation}
	is a polyhedral subdivision of $Q$ formed by simplices of the form
 $\Delta_{B}:=\conv\{\a_i\mid i \in B\}$ 
for $d$-element subsets $B$ of $[n]:=\{1,\dots,n\}$.
We view such simplices $\Delta_B$ as labeled by subsets $B$, see \cref{rmk:labeled}. To every $\A$-triangulation $\Tr$, Gelfand--Kapranov--Zelevinsky~\cite{GKZ} associated a point $\vertGKZ(\Tr)\in\R^n$ defined by
\begin{equation}	\label{eq:defGKZ}
\vertGKZ(\Tr):=\sum_{\Delta_B\in \Tr} \VolDB\cdot  \e_B,
\end{equation}
where $\Vol^{d-1}$ is the usual Euclidean volume in $\R^{d-1}$, $\e_1,\e_2,\dots,\e_n$ is the standard basis of $\R^n$, and we set  $\e_B:=\sum_{i\in B}\e_i$ for $B\subset [n]$. The \emph{secondary polytope} $\SigmaGKZ$ of $\A$ is defined as the convex hull of vectors $\vertGKZ(\Tr)$ where $\Tr$ ranges over all $\A$-triangulations. It turns out~\cite[Chapter~7, Theorem~1.7]{GKZ} that the vertices of $\SigmaGKZ$ correspond precisely to \emph{regular $\A$-triangulations}, defined in \cref{sec:regular}.

 Billera and Sturmfels~\cite{BS}  introduced a more general notion of a \emph{fiber
polytope} $\FibPoly(P\pito Q)$ 
for any affine projection of polytopes 
$\pi:P\to Q$,
 which we review in 
\cref{sec:fiber-polytopes}.
If $P:=\Delta^{n-1}=\conv(\e_1,\dots,\e_n)$ is the standard
$(n-1)$-dimensional simplex in $\R^n$, $Q:=\conv{\A}$, and $\pi$ is defined by $\pi(\e_i) = \a_i$ for all $i$, 
then the fiber polytope 
$\FibPoly(\Delta^{n-1}\pito Q)$ is a dilation of the secondary 
polytope $\SigmaGKZ$, see~\cite[Theorem~2.5]{BS}. Therefore the vertices of
	$\FibPoly(\Delta^{n-1}\pito Q)$
correspond to regular $\A$-triangulations. 

Another interesting case is when $P=\cube_n=[0,1]^n$ is
the standard $n$-cube. Let us denote by 
$\VC:=(\v_1,\dots,\v_n)$ the \emph{lift of $\A$}, i.e., the vector configuration in $\R^d$ obtained from $\A$ by setting $\v_i:=(\a_i,1)\in\R^d$ for $i=1,\dots, n$, and let 
$\Zon_\VC:=\sum_{i=1}^n[0,\v_i]\subset \R^d$ be the zonotope associated to $\VC$. 
We have a projection 
$\cube_n\pito \Zon_\VC$, defined by $\pi(\e_i) = \v_i$ for all $i$, 
and in this case, the fiber polytope $\FibPoly(\cube_n\pito \Zon_\VC)$ is called
the \emph{fiber zonotope} of $\ZV$.  Its vertices correspond to \emph{regular fine zonotopal tilings} of the zonotope $\Zon_\VC$, discussed below.
Restricting this projection map $\pi$ to the \emph{hypersimplex} $\Delta_{k,n}:=\cube_n\cap \{\x\in\R^n\mid x_1+\dots+x_n=k\}$, and denoting its image  by $Q_k:=\pi(\Delta_{k,n})=\Zon_\VC\cap \{\y\in\R^d\mid y_d=k\}$, we obtain a fiber polytope $\FibPoly(\Delta_{k,n}\pito Q_k)$ which has recently appeared in the theory of total positivity for Grassmannians~\cite{Gal,PosICM} and was studied further in~\cite{OS}.

\subsection{Higher secondary polytopes}
Given a configuration of $n$ points $\A \subset \R^{d-1}$ and its lift 
$\V \subset \R^d$ as above, we introduce a family of polytopes $\Sigma_{\A,1},\dots,
\Sigma_{\A,n-d}$, called \emph{higher secondary polytopes}, defined as follows.
For a $d$-element subset $B$ of  $[n]$, let
$\VolB:=|\det(\v_i)_{i\in B}|$ be the volume of the parallelepiped $\tile_B$
spanned by the vectors $\{\v_i\mid i\in B\}$. For a pair of disjoint subsets
$A,B$ of $[n]$ such that $|B|=d$  and $\VolB>0$ (i.e., such that $B$ is a
\emph{basis} of $\VC$), define the shifted \emph{parallelepiped}
$\tile_{A,B}\subset\Zon_\VC$ by 
\[
\tile_{A,B}:=\sum_{a\in A} \v_a+\sum_{b\in B}
[0,\v_b].
\]  
Clearly $\Vol^d(\tile_{A,B}) = \VolB$ for any $A$.  A \emph{fine zonotopal
tiling} of $\Zon_\VC$ is (roughly speaking) a collection $\Tiling$ of
parallelepipeds $\tile_{A,B}$ that form a polyhedral subdivision of $\Zon_\VC$,
see \cref{def:finetiling},
and we say that $\Tiling$ is \emph{regular} if it can be obtained as a projection of
the upper boundary of a $(d+1)$-dimensional zonotope onto $\Zon_\VC$, see 
\cref{dfn:regular_tiling}.

\begin{definition}\label{def:HSP} 
 For a fine zonotopal tiling $\Tiling$ of $\Zon_\VC$ and $k\in\Z$, we introduce a vector
\begin{equation}\label{eq:HSP}
\vertHSP_k(\Tiling):=\sumlevel{k} \VolB\cdot  \e_A \in\R^n.
\end{equation}
\noindent It is clear that $\vertHSP_k(\Tiling)=0$ if $k\notin [n-d]$. For $k\in[n-d]$, the \emph{higher secondary polytope} $\Sigma_{\A,k}$ is defined by
\[\Sigma_{\A,k}:=\Conv\left\{\vertHSP_k(\Tiling)\;\middle|\;  \text{$\Tiling$ is a fine \emph{regular} zonotopal tiling of $\Zon_\VC$} \right\}.\]
\end{definition}
\noindent We expect that the word \emph{regular} can be omitted from the above definition, see \cref{mainconj}. 
 As we will see in \cref{prop:HSP_dim}, for each $k\in[n-d]$, the polytope $\Sigma_{\A,k}$ has dimension $n-d$. An example of a higher secondary polytope is shown in \cref{fig:Gr_3_6}.

For simplicity, we formulate the following result modulo affine translation.
A more precise formulation will be given in~\eqref{eq:proof}. For polytopes
$P,P'\subset\R^m$, we write $P\shifteq P'$ if $P=P'+{\ggamma}$ for some
$\ggamma\in\R^m$. 

\begin{theorem}\label{thm:sigma} 
Let $\A\subset\R^{d-1}$ be a point configuration. Recall that
$Q = \conv \A$, $\VC\subset\R^d$ is the lift of $\A$, $\Zon_\VC$ is the zonotope of $\VC$, and 
 $Q_k=\Zon_\VC\cap \{\y\in\R^d\mid y_d=k\}$ is the $k$-th section of $\Zon_\VC$. 
Then we have the following.
	\begin{theoremlist}
\item\label{thm:sigma_GKZ} 
	$\SigmaGKZ\shifteq \frac1{(d-1)!}\Sigma_{\A,1}$, equivalently, $\FibPoly(\Delta^{n-1}\pito Q)    \shifteq	\frac1{d! \Voldd(Q)}\Sigma_{\A,1}$.
	\item\label{thm:FibPoly}
  $\FibPoly(\cube_n\pito \Zon_\VC)\shifteq \frac1{\Vold(\Zon_\VC)} \left(\Sigma_{\A,1} + \cdots + \Sigma_{\A,n-d}\right)$.
\item\label{thm:FibPoly_k}
$\FibPoly(\Delta_{k,n}\pito Q_k) \shifteq
\frac1{\Voldd(Q_k)} \left(p_{0,d}\Sigma_{\A,k} + p_{1,d}\Sigma_{\A,k-1} + \dots + p_{d-1,d}\Sigma_{\A,k-d+1}\right)$ for all $k\in[n-1]$, where $p_{r,d}$ is the probability that a random permutation in $S_d$ has $r$ descents.
\item\label{thm:duality} 
\emph{Duality:} $\Sigma_{\A,k} \shifteq - \Sigma_{\A,n-d-k+1}$ for all $k\in[n-d]$.
\end{theoremlist}
\end{theorem}
\noindent Here we assume that $\Sigma_{\A,k}$ is a single point if $k\notin [n-d]$. The volume forms $\Vold$ and $\Voldd$ on $\R^d$ are scaled so that $\Vold([0,1]^d)=\Voldd([0,1]^{d-1}\times\{y_d\})=1$ for any $y_d\in\R$. The numbers $p_{r,d}$ are given by the formula $p_{r,d}=\frac{\Eul(d,r)}{d!}$, where $\Eul(d,r)$ is the \emph{Eulerian number}, i.e., the number of permutations of $1,2,\dots,d$ with exactly $r$ descents.

\begin{remark}\label{rem:notobvious}
	\cref{thm:sigma_GKZ}	is not an obvious consequence of the definitions: 
	it says that $\SigmaGKZ$ (defined by~\eqref{eq:defGKZ}) 
	is the convex hull of points
\begin{equation}\label{eq2:GKZ}
	\frac1{(d-1)!}\sumlevel{1}
	\VolB \cdot \e_A
\end{equation}
for all regular fine zonotopal tilings $\Tiling$ of $\Zon_\VC$. The formulae~\eqref{eq:defGKZ} 
	and~\eqref{eq2:GKZ} are quite different: we have $\e_B$ in~\eqref{eq:defGKZ} as opposed to $\e_A$ in~\eqref{eq2:GKZ}, and we have $|A|=0$ in~\eqref{eq:defGKZ} as 
	opposed to $|A|=1$ in~\eqref{eq2:GKZ}.

On the other hand, it is easy to see from the definitions that 
the \emph{last} higher secondary polytope $\Sigma_{\A,n-d}$ satisfies $\SigmaGKZ\shifteq -\frac1{(d-1)!}\Sigma_{\A,n-d}$. Thus \cref{thm:sigma_GKZ} 
follows from \cref{thm:duality}.
\end{remark}

\begin{remark}
 The polytope $\Sigma_{\A,k}$ in \cref{fig:Gr_3_6} is centrally symmetric, in agreement with \cref{thm:duality}: we have $k=2=n-k-d+1$, thus $\Sigma_{\A,k}\shifteq -\Sigma_{\A,k}$.
\end{remark}

\begin{example}\label{ex:d1} 
Let $d=1$ and let $\A$ be the configuration of $n$ points
$\a_1=\cdots = \a_n = 0\in\R^0$.
Then $\V$ is the configuration of $n$ vectors
$\v_1 = \dots = \v_n = (1)\in \R^1$, and the zonotope
$\ZV$ is the interval $[0,n] \subset \R^1$.
	There are $n!$ fine zonotopal tilings of $\ZV$
	(see \cref{def:finetiling}), 
in bijection with the permutations $w\in S_n$.  More specifically,
for each $w\in S_n$, we have the following fine zonotopal tiling $\Tiling_w$ of $\ZV$:
\[\Tiling_w := \left\{\tile_{\emptyset, \{w_1\}}, 
\tile_{\{w_1\}, \{w_2\}}, \dots,
\tile_{\{w_1,\dots,w_{n-1}\}, \{w_n\}}\right\}.\]
Even though geometrically the tilings $\Tiling_w$ are the same for all $w\in S_n$, we treat them as different tilings because we take into account the labels of the tiles, see \cref{rmk:labeled}. 
\def\Perm{\operatorname{Perm}}
 We have $\vertHSP_k(\Tiling_w) = \e_{\{w_1,\dots,w_k\}}$, thus $\Sigma_{\A,k}$ is the 
	hypersimplex $\Delta_{k,n}$. It is straightforward to see from the definitions (cf.~\cite[Example~5.4]{BS} or~\cite[Example~9.8]{Ziegler}) that $n\cdot \FibPoly(\cube_n\pito \Zon_\VC)$ is the \emph{permutohedron} $\Perm_n:=\Conv\{(w_1,\dots,w_n)\mid w\in S_n\}$. Thus \cref{thm:FibPoly} recovers the following well known decomposition~\cite[Section~16]{PostnikovPAB} (implicit in~\cite{GS}) 
of the permutohedron as a Minkowski sum of hypersimplices:
\[\Perm_n=\Delta_{1,n}+\Delta_{2,n}+\dots+\Delta_{n-1,n}.\]
\end{example}

More generally, one can consider the case\footnote{Even more generally, we could choose a sequence of $n$ vectors such that $\det(\v_i)_{i\in B}>0$ for all $B\subset [n]$ of size $k$.} where $\VC$ is a  \emph{cyclic vector configuration} $C(n,d)$, i.e., is given by $\v_i = (u_i^{d-1}, \dots, u_i, 1)$ for $i\in[n]$ and $0 < u_1 < u_2 < \dots < u_n \in \R$. 
Thus \cref{ex:d1}  corresponds to the case $d=1$. If $d=2$, 
then the zonotope $\ZV$ is a $2n$-gon, and fine zonotopal tilings are 
exactly the rhombus tilings of the $2n$-gon.  They correspond to commutation
classes of reduced decompositions of the longest permutation $w_0\in S_n$ \cite{Elnitsky}.
 It would be interesting to understand the structure of the associated higher secondary polytopes in more detail.

 \begin{remark}\label{rem:higher}
 	For a cyclic vector configuration $C(n,d)$, Ziegler \cite{ZieglerHigher} identified the 
 fine zonotopal tilings of the cyclic zonotope $\ZV$ with elements of 
 Manin-Shekhtman's \emph{higher Bruhat order} $B(n,d)$ \cite{MS}, also studied by 
 	Voevodsky and Kapranov \cite{VK}.  Note that $B(n,1)$ coincides with the weak Bruhat 
 	order on permutations, corresponding to the case $d=1$ 
 in \cref{ex:d1}. 
 \end{remark}

We next proceed to the case $d=3$.

\def\Grtnn{\Gr_{\geq0}}
\def\BIP{{\operatorname{bip}}}
\def\bip#1{#1^{\BIP}}
\def\Gbip{\bip{G}}

\subsection{Higher associahedra and plabic graphs}\label{sec:high-assoc-plab}
Our main motivating example is the case when $\Sigma_{\A,k}$ is a \emph{higher associahedron}, that is, when $d=3$ and 
$\A$ is the configuration of vertices of a convex $n$-gon in $\R^2$. For example, one could take the points in $\A$ lying on a parabola, in which case the lift $\VC$ of $\A$ is a cyclic vector configuration $C(n,3)$. It turns out that the combinatorics of higher associahedra is directly related to \emph{bipartite plabic graphs} that were introduced in~\cite{Pos06} in the study of the \emph{totally nonnegative Grassmannian} $\Grtnn(k,n)$. 

A \emph{plabic graph} is a planar graph embedded in a disk 
such that every boundary vertex has degree $1$ and every interior vertex is colored either black or white. A plabic graph is called \emph{trivalent} if every interior vertex has degree $3$, and it is called \emph{bipartite} if no two interior vertices of the same color are connected by an edge. Note that  taking a trivalent plabic graph $G$ and contracting all edges between interior vertices of the same color produces a bipartite plabic graph denoted $\Gbip$.

There is a special class of \emph{$(k,n)$-plabic graphs} (cf.~\cref{def:kn}), that were used in~\cite{Pos06} to parametrize the top-dimensional cell of $\Grtnn(k,n)$. Each $(k,n)$-plabic graph has $n$ boundary vertices and $k(n-k)+1$ faces, and its \emph{face labels} (cf.~\cref{def:faces}) form a cluster in the cluster algebra structure on the coordinate ring of the Grassmannian~\cite{Scott}.

\def\plabiclw{1pt}\def\gridlw{1pt}
\def\gridop{0.25}
\def\gscl{1.5}

\begin{figure}
\def\boundnoderadius{1.5pt}

\scalebox{0.7}{
\setlength{\tabcolsep}{12pt}
\begin{tabular}{cc}
\scalebox{0.8}{
\begin{tikzpicture}[scale=1.0] %
\coordinate (bnode1236n1) at (7.67,5.67);
\coordinate (wnode23n1) at (6.00,7.00);
\coordinate (wnode16n1) at (7.33,3.67);
\coordinate (bnode1236n2) at (7.33,4.67);
\coordinate (wnode16n2) at (7.00,2.67);
\coordinate (bnode1346n1) at (5.33,4.00);
\coordinate (wnode36n1) at (5.33,5.00);
\coordinate (bnode1456n1) at (4.67,2.00);
\coordinate (wnode46n1) at (3.67,3.33);
\coordinate (bnode2346n1) at (4.33,6.33);
\coordinate (wnode34n1) at (2.67,6.00);
\coordinate (bnode3456n1) at (2.00,4.00);
\draw[dashed, line width=1pt,black!70] (5.00,4.50) circle (4.60);
\coordinate (bound_6_1 2_3) at (9.40,5.84);
\coordinate (bound_1_2 3_4) at (6.55,8.83);
\coordinate (bound_2_3 4_5) at (1.62,7.63);
\coordinate (bound_3_4 5_6) at (0.49,3.59);
\coordinate (bound_4_5 6_1) at (3.96,0.02);
\coordinate (bound_5_1 6_2) at (8.76,1.84);
\coordinate (node12) at (9.40,5.84);
\coordinate (node16) at (7.00,3.00);
\coordinate (node23) at (6.00,7.00);
\coordinate (node34) at (2.67,6.00);
\coordinate (node36) at (5.33,5.00);
\coordinate (node45) at (0.49,3.59);
\coordinate (node46) at (3.67,3.33);
\coordinate (node56) at (3.96,0.02);
\coordinate (node1234) at (6.55,8.83);
\coordinate (node1236) at (7.50,5.25);
\coordinate (node1256) at (8.76,1.84);
\coordinate (node1346) at (5.33,4.00);
\coordinate (node1456) at (4.67,2.00);
\coordinate (node2345) at (1.62,7.63);
\coordinate (node2346) at (4.33,6.33);
\coordinate (node3456) at (2.00,4.00);
\draw[blue, line width=\plabiclw] (node12).. controls (8.50,5.50) .. (bnode1236n1);
\draw[blue, line width=\plabiclw] (wnode16n1).. controls (8.00,4.00) .. (bnode1236n2);
\draw[blue, line width=\plabiclw] (node23).. controls (7.00,6.50) .. (bnode1236n1);
\draw[blue, line width=\plabiclw] (node36).. controls (6.50,5.00) .. (bnode1236n2);

\draw[blue, line width=\plabiclw] (wnode16n1).. controls (6.50,3.50) .. (node1346);
\draw[blue, line width=\plabiclw] (wnode16n2).. controls (6.00,2.00) .. (node1456);
\draw[blue, line width=\plabiclw] (wnode16n2).. controls (7.50,2.50) .. (node1256);

\draw[blue, line width=\plabiclw] (wnode16n2) to[bend right=30] (wnode16n1);
\draw[blue, line width=\plabiclw] (bnode1236n1) to[bend right=-30] (bnode1236n2);

\draw[blue, line width=\plabiclw] (node23).. controls (6.00,7.50) .. (node1234);
\draw[blue, line width=\plabiclw] (node23).. controls (5.00,7.00) .. (node2346);
\draw[blue, line width=\plabiclw] (node34).. controls (2.50,6.50) .. (node2345);
\draw[blue, line width=\plabiclw] (node34).. controls (2.00,5.00) .. (node3456);
\draw[blue, line width=\plabiclw] (node34).. controls (3.50,6.50) .. (node2346);
\draw[blue, line width=\plabiclw] (node36).. controls (4.50,5.50) .. (node2346);
\draw[blue, line width=\plabiclw] (node36).. controls (5.00,4.50) .. (node1346);
\draw[blue, line width=\plabiclw] (node45).. controls (1.50,3.50) .. (node3456);
\draw[blue, line width=\plabiclw] (node46).. controls (4.50,4.00) .. (node1346);
\draw[blue, line width=\plabiclw] (node46).. controls (2.50,3.50) .. (node3456);
\draw[blue, line width=\plabiclw] (node46).. controls (4.00,2.50) .. (node1456);
\draw[blue, line width=\plabiclw] (node56).. controls (4.00,1.50) .. (node1456);
\draw[blue, line width=\plabiclw] (node1234).. controls (6.00,7.50) .. (node23);
\draw[blue, line width=\plabiclw] (node1346).. controls (4.50,4.00) .. (node46);
\draw[blue, line width=\plabiclw] (node1346).. controls (5.00,4.50) .. (node36);
\draw[blue, line width=\plabiclw] (node1456).. controls (4.00,1.50) .. (node56);
\draw[blue, line width=\plabiclw] (node1456).. controls (4.00,2.50) .. (node46);
\draw[blue, line width=\plabiclw] (node2345).. controls (2.50,6.50) .. (node34);
\draw[blue, line width=\plabiclw] (node2346).. controls (5.00,7.00) .. (node23);
\draw[blue, line width=\plabiclw] (node2346).. controls (4.50,5.50) .. (node36);
\draw[blue, line width=\plabiclw] (node2346).. controls (3.50,6.50) .. (node34);
\draw[blue, line width=\plabiclw] (node3456).. controls (2.00,5.00) .. (node34);
\draw[blue, line width=\plabiclw] (node3456).. controls (2.50,3.50) .. (node46);
\draw[blue, line width=\plabiclw] (node3456).. controls (1.50,3.50) .. (node45);

\draw[blue, line width=\plabiclw,fill=white] (wnode16n1) circle (3pt);
\draw[blue, line width=\plabiclw,fill=white] (wnode16n2) circle (3pt);
\draw[blue, line width=\plabiclw,fill=blue] (bnode1236n1) circle (3pt);
\draw[blue, line width=\plabiclw,fill=blue] (bnode1236n2) circle (3pt);


\draw[blue, line width=\plabiclw,fill=white] (node23) circle (3.0pt);
\draw[blue, line width=\plabiclw,fill=white] (node34) circle (3.0pt);
\draw[blue, line width=\plabiclw,fill=white] (node36) circle (3.0pt);
\draw[blue, line width=\plabiclw,fill=white] (node46) circle (3.0pt);
\draw[blue, line width=\plabiclw,fill=blue] (node1346) circle (3.0pt);
\draw[blue, line width=\plabiclw,fill=blue] (node1456) circle (3.0pt);
\draw[blue, line width=\plabiclw,fill=blue] (node2346) circle (3.0pt);
\draw[blue, line width=\plabiclw,fill=blue] (node3456) circle (3.0pt);

\boundnode(node12)
\boundnode(node45)
\boundnode(node56)
\boundnode(node1234)
\boundnode(node2345)
\boundnode(node1256)
\end{tikzpicture}}
&
\scalebox{0.8}{
\begin{tikzpicture}[scale=1.0] %
\coordinate (bnode1236n1) at (7.67,5.67);
\coordinate (wnode23n1) at (6.00,7.00);
\coordinate (wnode16n1) at (7.33,3.67);
\coordinate (bnode1236n2) at (7.33,4.67);
\coordinate (wnode16n2) at (7.00,2.67);
\coordinate (bnode1346n1) at (5.33,4.00);
\coordinate (wnode36n1) at (5.33,5.00);
\coordinate (bnode1456n1) at (4.67,2.00);
\coordinate (wnode46n1) at (3.67,3.33);
\coordinate (bnode2346n1) at (4.33,6.33);
\coordinate (wnode34n1) at (2.67,6.00);
\coordinate (bnode3456n1) at (2.00,4.00);
\draw[dashed, line width=1pt,black!70] (5.00,4.50) circle (4.60);
\coordinate (bound_6_1 2_3) at (9.40,5.84);
\coordinate (bound_1_2 3_4) at (6.55,8.83);
\coordinate (bound_2_3 4_5) at (1.62,7.63);
\coordinate (bound_3_4 5_6) at (0.49,3.59);
\coordinate (bound_4_5 6_1) at (3.96,0.02);
\coordinate (bound_5_1 6_2) at (8.76,1.84);
\coordinate (node12) at (9.40,5.84);
\coordinate (node16) at (7.00,3.00);
\coordinate (node23) at (6.00,7.00);
\coordinate (node34) at (2.67,6.00);
\coordinate (node36) at (5.33,5.00);
\coordinate (node45) at (0.49,3.59);
\coordinate (node46) at (3.67,3.33);
\coordinate (node56) at (3.96,0.02);
\coordinate (node1234) at (6.55,8.83);
\coordinate (node1236) at (7.50,5.25);
\coordinate (node1256) at (8.76,1.84);
\coordinate (node1346) at (5.33,4.00);
\coordinate (node1456) at (4.67,2.00);
\coordinate (node2345) at (1.62,7.63);
\coordinate (node2346) at (4.33,6.33);
\coordinate (node3456) at (2.00,4.00);
\draw[blue, line width=\plabiclw] (node12).. controls (8.50,5.50) .. (node1236);
\draw[blue, line width=\plabiclw] (node16).. controls (8.00,4.00) .. (node1236);
\draw[blue, line width=\plabiclw] (node16).. controls (6.50,3.50) .. (node1346);
\draw[blue, line width=\plabiclw] (node16).. controls (6.00,2.00) .. (node1456);
\draw[blue, line width=\plabiclw] (node16).. controls (7.50,2.50) .. (node1256);
\draw[blue, line width=\plabiclw] (node23).. controls (6.00,7.50) .. (node1234);
\draw[blue, line width=\plabiclw] (node23).. controls (5.00,7.00) .. (node2346);
\draw[blue, line width=\plabiclw] (node23).. controls (7.00,6.50) .. (node1236);
\draw[blue, line width=\plabiclw] (node34).. controls (2.50,6.50) .. (node2345);
\draw[blue, line width=\plabiclw] (node34).. controls (2.00,5.00) .. (node3456);
\draw[blue, line width=\plabiclw] (node34).. controls (3.50,6.50) .. (node2346);
\draw[blue, line width=\plabiclw] (node36).. controls (6.50,5.00) .. (node1236);
\draw[blue, line width=\plabiclw] (node36).. controls (4.50,5.50) .. (node2346);
\draw[blue, line width=\plabiclw] (node36).. controls (5.00,4.50) .. (node1346);
\draw[blue, line width=\plabiclw] (node45).. controls (1.50,3.50) .. (node3456);
\draw[blue, line width=\plabiclw] (node46).. controls (4.50,4.00) .. (node1346);
\draw[blue, line width=\plabiclw] (node46).. controls (2.50,3.50) .. (node3456);
\draw[blue, line width=\plabiclw] (node46).. controls (4.00,2.50) .. (node1456);
\draw[blue, line width=\plabiclw] (node56).. controls (4.00,1.50) .. (node1456);
\draw[blue, line width=\plabiclw] (node1234).. controls (6.00,7.50) .. (node23);
\draw[blue, line width=\plabiclw] (node1236).. controls (8.50,5.50) .. (node12);
\draw[blue, line width=\plabiclw] (node1236).. controls (8.00,4.00) .. (node16);
\draw[blue, line width=\plabiclw] (node1236).. controls (6.50,5.00) .. (node36);
\draw[blue, line width=\plabiclw] (node1236).. controls (7.00,6.50) .. (node23);
\draw[blue, line width=\plabiclw] (node1256).. controls (7.50,2.50) .. (node16);
\draw[blue, line width=\plabiclw] (node1346).. controls (6.50,3.50) .. (node16);
\draw[blue, line width=\plabiclw] (node1346).. controls (4.50,4.00) .. (node46);
\draw[blue, line width=\plabiclw] (node1346).. controls (5.00,4.50) .. (node36);
\draw[blue, line width=\plabiclw] (node1456).. controls (6.00,2.00) .. (node16);
\draw[blue, line width=\plabiclw] (node1456).. controls (4.00,1.50) .. (node56);
\draw[blue, line width=\plabiclw] (node1456).. controls (4.00,2.50) .. (node46);
\draw[blue, line width=\plabiclw] (node2345).. controls (2.50,6.50) .. (node34);
\draw[blue, line width=\plabiclw] (node2346).. controls (5.00,7.00) .. (node23);
\draw[blue, line width=\plabiclw] (node2346).. controls (4.50,5.50) .. (node36);
\draw[blue, line width=\plabiclw] (node2346).. controls (3.50,6.50) .. (node34);
\draw[blue, line width=\plabiclw] (node3456).. controls (2.00,5.00) .. (node34);
\draw[blue, line width=\plabiclw] (node3456).. controls (2.50,3.50) .. (node46);
\draw[blue, line width=\plabiclw] (node3456).. controls (1.50,3.50) .. (node45);
\draw[blue, line width=\plabiclw,fill=white] (node16) circle (3.0pt);
\draw[blue, line width=\plabiclw,fill=white] (node23) circle (3.0pt);
\draw[blue, line width=\plabiclw,fill=white] (node34) circle (3.0pt);
\draw[blue, line width=\plabiclw,fill=white] (node36) circle (3.0pt);
\draw[blue, line width=\plabiclw,fill=white] (node46) circle (3.0pt);
\draw[blue, line width=\plabiclw,fill=blue] (node1236) circle (3.0pt);
\draw[blue, line width=\plabiclw,fill=blue] (node1346) circle (3.0pt);
\draw[blue, line width=\plabiclw,fill=blue] (node1456) circle (3.0pt);
\draw[blue, line width=\plabiclw,fill=blue] (node2346) circle (3.0pt);
\draw[blue, line width=\plabiclw,fill=blue] (node3456) circle (3.0pt);

\boundnode(node12)
\boundnode(node45)
\boundnode(node56)
\boundnode(node1234)
\boundnode(node2345)
\boundnode(node1256)
\end{tikzpicture}}
\\

\end{tabular}

}
\def\boundnoderadius{1pt}

\caption{\label{fig:plabic_graph} A plabic graph and its bipartite version.}
\end{figure}
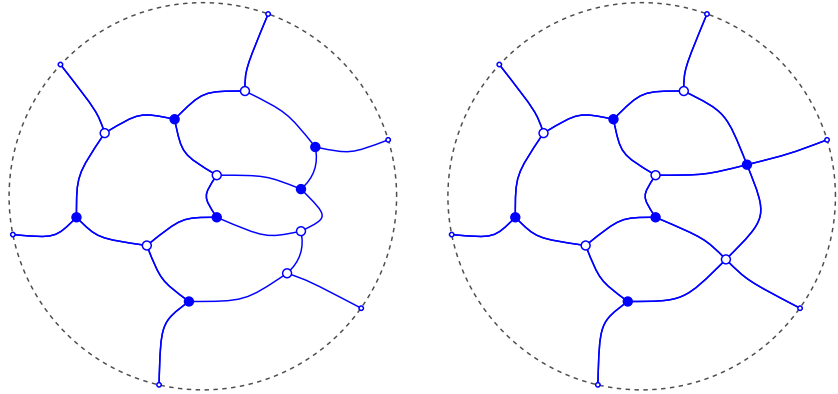

\def\movescl{0.4}
\def\bluecircscl{1}

\def\plabiclww{1.4pt}

\begin{figure}

\begin{tabular}{c|c|c} 

\scalebox{\movescl}{
 \begin{tikzpicture}
  \coordinate(a) at (-2,1);
  \coordinate(b) at (-2,-1);
  \coordinate(c) at (2,0);
  \coordinate(d) at (4,0);
  \draw[blue,line width=\plabiclww] (a) -- (b);
  \draw[blue,line width=\plabiclww] (c) -- (d);
  \draw[blue,line width=\plabiclww] (a) -- (-3,2);
  \draw[blue,line width=\plabiclww] (a) -- (-1,2);
  \draw[blue,line width=\plabiclww] (b) -- (-3,-2);
  \draw[blue,line width=\plabiclww] (b) -- (-1,-2);
  \draw[<->,line width=2pt] (-1,0) -- (1,0);

  \draw[blue,line width=\plabiclww] (c) -- (1,1);
  \draw[blue,line width=\plabiclww] (c) -- (1,-1);
  \draw[blue,line width=\plabiclww] (d) -- (5,1);
  \draw[blue,line width=\plabiclww] (d) -- (5,-1);
  
 \node[fill=white,draw=blue,circle,scale=\bluecircscl,line width=\plabiclww] (AAA) at (a) {};
 \node[fill=white,draw=blue,circle,scale=\bluecircscl,line width=\plabiclww] (AAA) at (b) {};
 \node[fill=white,draw=blue,circle,scale=\bluecircscl,line width=\plabiclww] (AAA) at (c) {};
 \node[fill=white,draw=blue,circle,scale=\bluecircscl,line width=\plabiclww] (AAA) at (d) {};
 \end{tikzpicture}} & 

\scalebox{\movescl}{
 \begin{tikzpicture}

 \draw[blue,line width=\plabiclww] (-4,1) -- (-5,2);
 \draw[blue,line width=\plabiclww] (-4,-1) -- (-5,-2);
 \draw[blue,line width=\plabiclww] (-2,1) -- (-1,2);
 \draw[blue,line width=\plabiclww] (-2,-1) -- (-1,-2);
 \draw[blue,line width=\plabiclww] (4,1) -- (5,2);
 \draw[blue,line width=\plabiclww] (4,-1) -- (5,-2);
 \draw[blue,line width=\plabiclww] (2,1) -- (1,2);
 \draw[blue,line width=\plabiclww] (2,-1) -- (1,-2);
 \draw[blue,line width=\plabiclww] (-4,1) -- (-2,1) -- (-2,-1) -- (-4,-1) -- cycle;
 \draw[blue,line width=\plabiclww] (4,1) -- (2,1) -- (2,-1) -- (4,-1) -- cycle;
 
 \draw[<->,line width=2pt] (-1,0) -- (1,0);
 
 \node[fill=blue,draw=blue,draw=blue,circle,scale=\bluecircscl,line width=\plabiclww] (AAA) at (-4,-1) {};
 \node[fill=white,draw=blue,circle,scale=\bluecircscl,line width=\plabiclww] (AAA) at (-4,1) {};
 \node[fill=blue,draw=blue,circle,scale=\bluecircscl,line width=\plabiclww] (AAA) at (-2,1) {};
 \node[fill=white,draw=blue,circle,scale=\bluecircscl,line width=\plabiclww] (AAA) at (-2,-1) {};
 \node[fill=blue,draw=blue,circle,scale=\bluecircscl,line width=\plabiclww] (AAA) at (4,-1) {};
 \node[fill=white,draw=blue,circle,scale=\bluecircscl,line width=\plabiclww] (AAA) at (4,1) {};
 \node[fill=blue,draw=blue,circle,scale=\bluecircscl,line width=\plabiclww] (AAA) at (2,1) {};
 \node[fill=white,draw=blue,circle,scale=\bluecircscl,line width=\plabiclww] (AAA) at (2,-1) {}; 
 \end{tikzpicture}}
&
\scalebox{\movescl}{
 \begin{tikzpicture}
  \coordinate(a) at (-2,1);
  \coordinate(b) at (-2,-1);
  \coordinate(c) at (2,0);
  \coordinate(d) at (4,0);
  \draw[blue,line width=\plabiclww] (a) -- (b);
  \draw[blue,line width=\plabiclww] (c) -- (d);
  \draw[blue,line width=\plabiclww] (a) -- (-3,2);
  \draw[blue,line width=\plabiclww] (a) -- (-1,2);
  \draw[blue,line width=\plabiclww] (b) -- (-3,-2);
  \draw[blue,line width=\plabiclww] (b) -- (-1,-2);
  \draw[<->,line width=2pt] (-1,0) -- (1,0);

  \draw[blue,line width=\plabiclww] (c) -- (1,1);
  \draw[blue,line width=\plabiclww] (c) -- (1,-1);
  \draw[blue,line width=\plabiclww] (d) -- (5,1);
  \draw[blue,line width=\plabiclww] (d) -- (5,-1);

 \node[fill=blue,draw=blue,circle,scale=\bluecircscl,line width=\plabiclww] (AAA) at (a) {};
 \node[fill=blue,draw=blue,circle,scale=\bluecircscl,line width=\plabiclww] (AAA) at (b) {};
 \node[fill=blue,draw=blue,circle,scale=\bluecircscl,line width=\plabiclww] (AAA) at (c) {};
 \node[fill=blue,draw=blue,circle,scale=\bluecircscl,line width=\plabiclww] (AAA) at (d) {};
 \end{tikzpicture}}\\
(M1) & (M2) & (M3)
\end{tabular}
  \caption{\label{fig:3moves} Moves on plabic graphs.}
\end{figure}
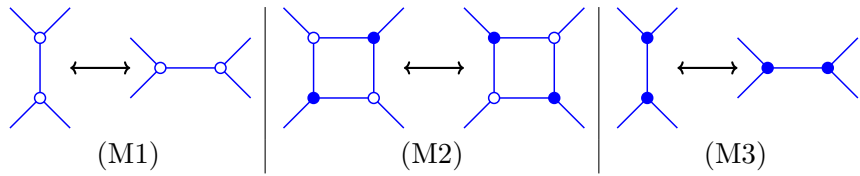

Given a plabic graph, one can apply certain moves to it,
as shown in \cref{fig:3moves}. Any two trivalent $(k,n)$-plabic graphs can be connected using moves (M1)--(M3), see~\cite[Theorem~13.4]{Pos06}. Since applying the moves (M1) and (M3) to $G$ does not change its bipartite version $\Gbip$, it follows that any two bipartite $(k,n)$-plabic graphs can be connected using only the \emph{square move} (M2).\footnote{We make the convention that applying a square move (M2) to a bipartite graph $\Gbip$ means first uncontracting some vertices of $\Gbip$ so that the vertices of the square become trivalent, then performing the square move, and then taking the bipartite version of the resulting graph.}
For example, there are $34$ bipartite $(3,6)$-plabic graphs corresponding to the $34$ points in \cref{fig:Gr_3_6} (including the two points labeled by $b$ and $c$), and square moves between them correspond to the edges in \cref{fig:Gr_3_6}.

Building on the work of Oh--Postnikov--Speyer~\cite{OPS}, it was shown in~\cite{Gal} that trivalent $(k,n)$-plabic graphs are exactly the 
planar duals of the  horizontal sections of fine zonotopal tilings of 
the zonotope $\ZV$ (where $\VC\subset \R^3$ is the lift of $\A$ as above), see \cref{thm:Gal}.
It was later observed in~\cite{PosICM} that trivalent $(k,n)$-plabic graphs correspond to $\pi$-induced subdivisions for the map $\pi:\Delta_{k,n}\to Q_k$.

 We say that a trivalent $(k,n)$-plabic graph $G$ is \emph{$\A$-regular} if it is the planar dual of a horizontal section of some regular fine zonotopal tiling of $\ZV$, or equivalently, if it corresponds to a regular $\pi$-induced subdivision of $Q_k$. We say that a bipartite $(k,n)$-plabic graph $G'$ is \emph{$\A$-regular} if it equals to $\Gbip$ for some $\A$-regular trivalent $(k,n)$-plabic graph $G$. For example, if $\A$ is the set of vertices of a generic hexagon, then there are $32$ $\A$-regular bipartite $(3,6)$-plabic graphs, corresponding to the $32$ vertices of the polytope shown in \cref{fig:Gr_3_6}. See \cref{ex:hexagons} for more details.

\begin{theorem}\label{thm:regular}
Let $d=3$ and $\A$ be the configuration of vertices of
a convex $n$-gon.  Then:
\begin{theoremlist}
\item\label{thm:regular_bip} For each $k\in[n-3]$, the vertices of $\Sigma_{\A, k}$  correspond
to $\A$-regular \emph{bipartite} $(k+1,n)$-plabic graphs, and the square moves connecting them correspond to the edges of $\Sigma_{\A,k}$.
\item\label{thm:regular_triv}  For each $k\in[n-1]$, the vertices of $\Sigma_{\A, k}+\Sigma_{\A, k-1} + \Sigma_{\A, k-2}$ (equivalently, of $\FibPoly(\Delta_{k,n}\pito Q_k)$) correspond to $\A$-regular \emph{trivalent} $(k,n)$-plabic graphs, and the moves (M1)--(M3) connecting them correspond to the edges of 
	$\Sigma_{\A, k}+\Sigma_{\A, k-1} + \Sigma_{\A, k-2}$.
\end{theoremlist}
\end{theorem}

\begin{example}\label{ex:Catalan}
The number of bipartite $(2,n)$-plabic graphs equals to the number of trivalent $(1,n)$-plabic graphs, and is given by the Catalan number $C_{n-2}$, where $C_m:=\frac1{m+1}{2m\choose m}$. In both cases, all such plabic graphs are regular, and the corresponding polytope is $\Sigma_{\A,1}$ which by \cref{thm:sigma_GKZ} is a realization of the associahedron.
\end{example}

\begin{figure}
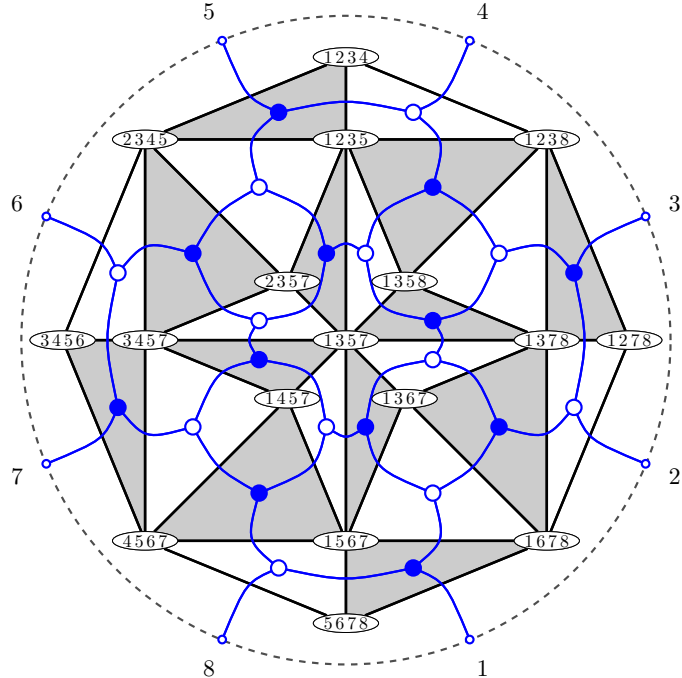


\scalebox{0.82}{
}

\caption{\label{fig:compass} {
A trivalent, bipartite $(4,8)$-plabic graph which admits only $4$ square moves,
superimposed onto its dual plabic tiling. This plabic graph is not $\A$-regular for any $\A$, see \cref{example:non_regular}.}}
\end{figure}

\begin{example}\label{example:non_regular}
	Since $\Sigma_{\A,k}$ has dimension $n-d$ by \cref{prop:HSP_dim}, it follows from \cref{thm:regular_bip} that every $\A$-regular bipartite $(k,n)$-plabic graph admits at least $n-d=n-3$ square moves. \cref{fig:compass} contains a (both trivalent and bipartite) $(4,8)$-plabic graph that admits only $4$ square moves, and therefore is not $\A$-regular for any $\A$. This plabic graph contains as a subgraph another plabic graph known in physics as the ``four-mass box'', see \cite[Section~11.1]{abcgpt}.  

An example of a trivalent $(9,18)$-plabic graph that is not $\A$-regular for any $\A$ was constructed in~\cite[Section~6]{KK}.
\end{example}

\def\opp{{\operatorname{op}}}
Let us say that the \emph{diameter} of a polytope is the maximal graph distance between its vertices in its $1$-skeleton. It would be interesting to find the diameter of a higher associahedron $\Sigma_{\A,k}$, which equals the maximal square move distance between two $\A$-regular plabic graphs. Finding the diameter of the usual associahedron $\Sigma_{\A,1}$ is a well-studied problem: answering a question of Sleator--Tarjan--Thurston~\cite{STT}, Pournin~\cite{Pournin} showed that it equals $2n-10$ for all $n>12$. The following conjecture is due to Miriam Farber.  
\begin{conjecture}[\cite{Miriam}]\label{conj:Miriam}
Let $n=2k$. Then the diameter of the higher associahedron $\Sigma_{\A,k-1}$ equals $\frac12 k(k-1)^2$. More generally, for any bipartite $(k,2k)$-plabic graph $G$, the minimal number of square moves needed to connect $G$ with $G^\opp$ equals $\frac12 k(k-1)^2$, where $G^\opp$ is obtained from $G$ by a $180^\circ$ rotation followed by changing the colors of all vertices.
\end{conjecture}
\noindent An example for $k=3$ is shown in \cref{fig:Gr_3_6}. The diameter of this polytope is equal to $\frac12 k(k-1)^2=6$, and moreover the graph distance between any vertex and its antipodal vertex is also equal to $6$.

It was shown in~\cite[Section~6]{BW1} that for certain bipartite $(k,2k)$-plabic graphs $G$ (coming from \emph{double wiring diagrams} of~\cite{FZDouble}), the square move distance between $G$ and $G^\opp$ is at least $\frac12 k(k-1)^2$, giving a lower bound on the diameter of $\Sigma_{\A,k-1}$ in \cref{conj:Miriam}. See~\cite{BW2} for related results. 

\subsection{Vertices, edges, and deformations}\label{sec:vert-edges-deform}
For simplicity, we assume here that $\A$ is a generic point configuration in $\R^{d-1}$. The extension of the results in this subsection to arbitrary point configurations will be given in 
\cref{sec2:vert-edges-deform}.

It is well known (cf. \cref{lemma:regular_flips}) that any two regular fine zonotopal tilings of $\ZV$ can be related to each other by a sequence of \emph{flips}. A \emph{flip} is an elementary transformation of a zonotopal tiling: if $\VC'$ consists of $d+1$ vectors that span $\R^d$ then $\Zon_{\VC'}$ admits precisely two fine zonotopal tilings. For general vector configurations $\VC$,
applying a flip $F=(\Tiling\to\Tiling')$ to a fine zonotopal tiling $\Tiling$ of $\ZV$ amounts to finding a shifted copy of a fine zonotopal tiling of $\Zon_{\VC'}$ for some $\VC'\subset \VC$ of size $d+1$, and replacing it with the other fine zonotopal tiling of $\Zon_{\VC'}$, which produces another fine zonotopal tiling $\Tiling'$ of $\ZV$,  see  \cref{fig:flips} (left). Flips can occur at different \emph{levels}: if 
the above copy of $\Zon_{\VC'}$ is shifted by $\y\in\R^d$,
then the last coordinate $y_d$ of $\y$ belongs to $\{0,1,\dots,n-d-1\}$, and we define the \emph{level} of the flip $F$ to be $\level(F):=y_d+1$. 
See \cref{def:level} and \cref{ex:flip}.

Since flips of regular fine zonotopal tilings correspond to the edges of the fiber zonotope $\FibPoly(\cube_n\pito \ZV)$, we define the \emph{level} of an edge of $\FibPoly(\cube_n\pito \ZV)$ to be the level of the corresponding flip.

Let us say that a polytope $P$ is a \emph{parallel deformation} of another polytope $P'$ if the normal fan of $P$ is a coarsening of the normal fan of $P'$, see e.g.~\cite[Theorem 15.3]{PRW} and ~\cite[Section 2.2]{ACEP}. Roughly speaking, $P$ is a parallel deformation of $P'$ if $P$ is obtained from $P'$ by moving its faces while preserving their direction. During this process, every edge of $P'$ stays parallel to itself but gets rescaled by some nonnegative real number.

We say that two fine zonotopal tilings $\Tiling$ and $\Tiling'$ of $\ZV$ are \emph{$k$-equivalent} if they can be connected by flips $F$ such that $\level(F)\neq k$. Similarly, we say that two flips $F=(\Tiling_1\to\Tiling_2)$ and $F'=(\Tiling'_1\to\Tiling'_2)$ of level $k$ are \emph{$k$-equivalent} 
if $\Tiling_1$ is $k$-equivalent to $\Tiling'_1$ and $\Tiling_2$ is $k$-equivalent to $\Tiling'_2$.

\begin{proposition}\label{prop:deform}
Let $\A$ be a generic configuration of $n$ points in $\R^{d-1}$, and let $k\in[n-d]$.
\begin{theoremlist}
\item\label{item:deform_vert} The vertices of the higher secondary polytope $\Sigma_{\A,k}$ are in bijection with $k$-equivalence classes of regular fine zonotopal tilings of $\ZV$.
\item\label{item:deform_edge} The edges of $\Sigma_{\A,k}$ correspond to $k$-equivalence classes of flips of level $k$.
\item\label{item:deform_deform} For any nonnegative real numbers $x_1,\dots,x_{n-d}$, the Minkowski sum
  \[\frac1{\Vold(\ZV)} \left(x_1 \Sigma_{\A,1}+\dots + x_{n-d} \Sigma_{\A,n-d}\right)\]
is a parallel deformation of the fiber zonotope $\FibPoly(\cube_n\pito \ZV)$, where edges of level $k$ are rescaled by $x_k$ for all $k=1,\dots,n-d$.
\end{theoremlist}  
\end{proposition}

\subsection{Soliton graphs} 

Finally we give applications of our previous results to the soliton graphs
\cite{KW, KW2, KK}
associated to 
 the Kadomtsev-Petviashvili (KP) equation. 
 To understand a soliton solution
$u_A(x,y,t)$ of the KP equation coming from a point $A$ in the positive Grassmannian,
one fixes the time $t$ and plots the points
where $u_A(x,y,t)$ has a local maximum.  This gives rise to a tropical curve
in the $xy$-plane.
By \cite{KW, KW2},
this tropical curve is a reduced plabic graph, and 
as discussed in \cite[Section 2.3]{KK}, it comes from a regular subdivision of 
a three-dimensional cyclic zonotope; we give a precise statement in 
\cref{cor:plabicsoliton}. 
 We then apply some of our previous results to classify the soliton 
graphs coming from the positive Grassmannian when the
time parameter $t$ tends to $\pm \infty$, and to show that 
generically, when the higher time parameters evolve, the 
face labels of soliton graphs change via
the square moves (cluster transformations) on plabic graphs.

\section{Fiber polytopes and zonotopal tilings}

We give further background on fiber polytopes of~\cite{BS} and discuss several specializations of their construction. More details can be found in~\cite{BS}, \cite[Chapter~7]{GKZ}, and~\cite[Lecture~9]{Ziegler}.

\subsection{Fiber polytopes}\label{sec:fiber-polytopes}

Let $P \subset \R^n$ be a polytope, and let 
$\pi:P\to Q$ be a linear projection of polytopes.  
We denote by $\{\p_i\}_{i\in [m]}$ the vertex set of $P$ (for some $m\geq1$). For $i\in [m]$, let $\q_i:=\pi(\p_i)$, and let $\Acal:=\{\q_i\}_{i\in[m]}$ be the associated point configuration. The \emph{fiber polytope} $\FibPoly(P\pito Q)$ is defined
as the \emph{Minkowski integral}
$$
\FibPoly(P\pito Q):=
\frac{1}{\Vol(Q)}\int_{\x\in Q}  (\pi^{-1}(\x)\cap P)\, d\x.
$$
\noindent Here $\Vol$ denotes the $\dim(Q)$-dimensional volume form on the affine span of $Q$, and the Minkowski integral can be understood in several ways,
for example, as the set of points $\int_{\x\in Q} \gamma(\x)\, d\x\in \R^n,$
where $\gamma:Q\to P$ runs over all sections of $\pi$ \cite{BS, BS2}.

However, instead of working with the Minkowski integral, we will use 
the following description of 
 $\FibPoly(P\pito Q)$ as a convex hull of points. Recall that an \emph{$\Acal$-triangulation} $\Tr=\{\Delta_B\}$ is a triangulation of $Q$ into simplices $\Delta_B:=\conv\{\q_i\mid i\in B\}$, where $B\subset [m]$ is a $(\dim(Q)+1)$-element subset.

\begin{proposition}[{\cite[Corollary~2.6]{BS}}]\label{prop:BS_triang}
  The fiber polytope $\FibPoly(P\pito Q)$ equals the convex hull \[\FibPoly(P\pito Q)=\conv\{\vertFib(\Tr)\mid \Tr\text{ is an $\Acal$-triangulation}\},\quad\text{where}\] 
\begin{equation}\label{eq:vertFib}
  \vertFib(\Tr):=\frac1{(\dim(Q)+1) \Vol(Q)} \sum_{\Delta_B\in \Tr}  \left(\Vol(\Delta_B) \cdot \sum_{i\in B} \p_i\right) \in \R^n.
\end{equation}
\end{proposition}

\begin{definition}[{\cite[Definition~9.1]{Ziegler}}]
	\label{dfn:pi_induced}
Let $\pi:P\to Q$ be a projection of polytopes as above.
	A \emph{$\pi$-induced subdivision} of $Q$  is a collection $\Tiling$ of faces of $P$ such that
\begin{itemize}
\item the images $\{\pi(F)\mid F\in \Tiling\}$ form a polyhedral subdivision\footnote{Recall that a \emph{polyhedral subdivision}
of a polytope $Q$ is a polytopal complex $\mathcal{C}$
(any two elements of $\mathcal{C}$ intersect in a common face)
with underlying set $Q$.} of $Q$;
\item for any $F,F'\in\Tiling$ such that $\pi(F)\subset\pi(F')$, we have $F=F'\cap \pi^{-1}(\pi(F))$. 
\end{itemize}
A $\pi$-induced subdivision $\Tiling$ is called \emph{fine} if all of its faces have dimension at most $\dim(Q)$.
\end{definition}

\begin{definition}\label{dfn:pi_regular}
For a polytope $P\subset \R^n$ and a vector $\h\in\R^n$, let $(P)^\h$ denote the face of $P$ that maximizes the scalar product with $\h$. Every vector $\h\in\R^n$ gives rise to a $\pi$-induced subdivision $\Tiling_\h$ of $Q$ obtained as follows: for each point $\q\in Q$, consider the preimage $P\cap \pi^{-1}(\q)$ of $\q$ under $\pi$, and let $P_{\q,\h}$ be the unique minimal by inclusion face of $P$ that contains $(P\cap \pi^{-1}(\q))^\h$. The subdivision $\Tiling_\h$ consists of the faces $P_{\q,\h}$ for all $\q\in Q$. 
	A $\pi$-induced subdivision $\Tiling$ of $Q$ is called \emph{regular} if it equals  $\Tiling_\h$ for some $\h\in\R^n$.
\end{definition}
\noindent 
Our notion of a regular $\pi$-induced subdivision coincides 
with the notion of a \emph{$\pi$-coherent subdivision} 
from~\cite[Section 1]{BS} and \cite[Definition 9.2]{Ziegler}.

It turns out (see the paragraph before~\cite[Corollary~2.7]{BS}) that if $\Tiling$ is a fine $\pi$-induced subdivision then the vector $\vertFib(\Tr)$ is the same for any triangulation $\Tr$ of $\Tiling$. We denote this vector by $\vertFib(\Tiling)$.
\begin{corollary}[{\cite[Corollary~2.7]{BS}}]\label{cor:BS_tiling}
The fiber polytope $\FibPoly(P\pito Q)$ equals the convex hull
\begin{align}
\label{eq:BS_tiling}
 \FibPoly(P\pito Q)&=\conv\{\vertFib(\Tiling)\mid \Tiling\text{ is a fine $\pi$-induced subdivision of $Q$}\}.\\
\intertext{The vertices of $\FibPoly(P\to Q)$ are the vectors $\vertFib(\Tiling)$, where $\Tiling$ ranges over all regular fine $\pi$-induced subdivisions of $Q$, and in particular,}
\label{eq:BS_tiling_reg}
 \FibPoly(P\pito Q)&=\conv\{\vertFib(\Tiling)\mid \Tiling\text{ is a \emph{regular} fine $\pi$-induced subdivision of $Q$}\}.
\end{align}
\end{corollary}

We now specialize this construction to the case where $P$ is either a cube or a (hyper)simplex. In these cases, regular fine $\pi$-induced subdivisions recover well-studied objects such as regular triangulations and regular fine zonotopal tilings. We discuss them briefly here, and 
in more detail in \cref{sec:regular}.
In what follows, we will repeatedly use  the following notation.
\def\Hyp{H}

\begin{notation}\label{notation}
Let $\A = (\a_1,\dots,\a_n)$ be
a point configuration in $\R^{d-1}$ which affinely spans $\R^{d-1}$. 
Let $\VC:=(\v_1,\dots,\v_n)$ be the lift of $\A$, thus $\v_i:=(\a_i,1)\in\R^d$ 
for $i=1,\dots, n$.  
Then the endpoints of the vectors in $\VC$ belong to $\Hyp_1$,  
where 
the hyperplane $\Hyp_k$ is defined by $\Hyp_k:=\{\y\in\R^d\mid y_d=k\}$ in 
$\R^d$. 
The \emph{zonotope} $\Zon_\V$ is defined as the Minkowski sum of line segments:
$$
\Zon_\V:=\sum_{i=1}^n [0,\v_i] \subset \R^d.
$$
We also let $Q_k:=\ZV\cap \Hyp_k\subset \R^d$.
Let $\pi$ be the projection $\pi: \R^n \to \R^d$ defined by 
 $\pi(\e_i) = \v_i$ for all $i$, where 
 $\e_1,\dots,\e_n$ is the standard basis in $\R^n$. 
\end{notation}

\subsection{Fiber polytopes for projections of a cube: fiber zonotopes}
\label{sec:zonotopal_tilings}

Let $P = \cube_n:=[0,1]^n = \sum_{i=1}^n [0,\e_i]\subset \R^n$ 
be the standard $n$-dimensional cube.
We have a linear projection $\pi:\cube_n\to\Zon_\V$ given by  
 $\pi(\e_i) = \v_i$, for $i\in [n]$.
The \emph{fiber zonotope} of $\ZV$ is the fiber polytope $\FibPoly(\cube_n\pito \Zon_\VC)$.

 Recall that for $A\subset[n]$, we set $\e_A:=\sum_{i\in A} \e_i$.
Faces
$\square_{A,B}$ of the $n$-cube $\cube_n$ are labeled by pairs 
$(A,B)$ of disjoint subsets $A$ and $B$ of $[n]$.  They are given by

\[\square_{A,B} := \e_A + \sum_{b\in B} [0,\e_b] = \{(x_1,\dots,x_n)\in\cube_n \mid x_a = 
1\textrm{ for } a\in A, \textrm{ and } 
x_c = 0 \textrm{ for } c \in [n]\setminus (A\sqcup B)\}.\]

\begin{definition}\label{def:finetiling}
A \emph{fine zonotopal tiling} $\Tiling$ of $\Zon_\V$ is a collection 
of $d$-dimensional faces $\square_{A,B}$ of the $n$-cube such 
that
\begin{enumerate}
\item The images $\tile_{A,B}:=\pi(\square_{A,B})$, for all $\square_{A,B}\in
\Tiling$, 
are $d$-dimensional parallelepipeds that 
form a polyhedral subdivision of the zonotope $\Zon_\V$.  
\item 
For any two faces $\square_{A_1,B_1},\, \square_{A_2,B_2}\in \Tiling$, we have 
$$
\pi(\square_{A_1,B_1}\cap \square_{A_2,B_2}) = \tile_{A_1,B_1} \cap \tile_{A_2,B_2}.
$$
\end{enumerate}
\end{definition}

From our definition, it is clear that each fine zonotopal tiling is a fine $\pi$-induced subdivision. We say that a fine zonotopal tiling is \emph{regular} if the corresponding fine $\pi$-induced subdivision is regular. See \cref{sec:regular_zon_tilings} for several alternative definitions.

\begin{remark}\label{rmk:labeled}
We refer to the $d$-parallelepipeds $\tile_{A,B}=\pi(\square_{A,B})$ 
	as \emph{(labeled) tiles}.
It may happen that 
for two different pairs $(A_1,B_1)$ and $(A_2,B_2)$, 
the two tiles $\tile_{A_1,B_1}$ and $\tile_{A_2, B_2}$ 
coincide as subsets of $\R^d$.   
However, we regard them as different
\emph{labeled tiles}, because they are labeled by different pairs. 
We will identify a fine zonotopal tiling $\Tiling$ with the collection
of such labeled tiles $\tile_{A,B}$. 
\end{remark}

The fiber zonotope of $\ZV$ can be described explicitly as follows.

\begin{proposition}\label{prop:vertFib}
	Let $\V  \subset \R^d$ be as in \cref{notation}.
The fiber zonotope $\FibPoly(\cube_n\pito \Zon_\VC)$ equals the convex hull
\begin{align}
  \FibPoly(\cube_n\pito \Zon_\VC)&=\conv\{\vertFib(\Tiling)\mid \Tiling\text{ is a fine zonotopal tiling of $\Zon_\VC$}\},\quad\text{and} \\
\label{eq:vertFibzon}
  \vertFib(\Tiling)&=\frac1{\Vol(\ZV)} \sum_{\tile_{A,B}\in \Tiling} \VolB\cdot \left(\e_A+\frac12 \e_B \right).
\end{align}
\end{proposition}

\begin{proof}
We use \eqref{eq:vertFib}, 
	and let $\tau$ be a triangulation of a fixed tile 
	$\tile_{A,B}$ of  
	$\Tiling$.  More specifically, we use Stanley's  
	triangulation \cite{Sta77} of $\cube_d$ into $d!$ 
	equal-volume simplices 
	$\nabla_w$ labeled by permutations $w\in S_d$:
	\begin{equation}\label{eq:Stanley}
		\nabla_w := \{(y_1,\dots,y_d) \in [0,1]^d \ \mid \ 
	0 < y_{w_1} < \dots < y_{w_d} < 1\}.
	\end{equation}
	This gives 
	rise to a triangulation $\tau$ of 
	 $\tile_{A,B}$ into $d!$ simplices, 
	 each of volume $\frac{\VolB}{d!}$. 
	 By symmetry, we know that the combined contribution of these simplices to~\eqref{eq:vertFib} has the form $x\cdot  \e_A+y\cdot \e_B$ for some $x,y\in\R$. Each simplex $\nabla_w$ contributes $\frac{\VolB}{d!\Vol(\ZV)} \e_A+\u(w)$ for some $\u(w)\in\R^n$. Let $\bar w\in S_d$ be the permutation given by $\bar w_i=w_{d+1-i}$ for all $i\in[d]$. It is easy to see that $\u(w)+ \u(\bar w)=\frac{\VolB}{d!\Vol(\ZV)} \e_B$, thus $x=\frac{\VolB}{\Vol(\ZV)}$ and $y=\frac{\VolB}{2\Vol(\ZV)}$. 
\end{proof}

\subsection{Fiber polytopes for projections of a simplex:  secondary polytopes.}\label{sec:fiber-secondary}

Let $\A$ and $\V$ be as in \cref{notation}.
Let $P=\Delta^{n-1}=\conv(\e_1,\dots,\e_n)$ be the standard $(n-1)$-dimensional
simplex in $\R^n$, 
and $\pi:P \to Q:=\conv{\A}$ the projection defined
by $\pi(\e_i) = \a_i$ for all $i$.

\begin{definition}[{\cite[Definition~1.6]{GKZ}}]
The \emph{secondary polytope} $\SigmaGKZ$ is defined as the convex hull
\begin{align}
\SigmaGKZ&:=\conv\{\vertGKZ(\Tr)\mid \Tr\text{ is an $\Acal$-triangulation}\},\quad\text{where} \\
\label{eq:vertGKZ}
  \vertGKZ(\Tr)&:= \sum_{\Delta_B\in \Tr} \VolDB\cdot \e_B.
\end{align}
\end{definition}
\noindent The relationship between the polytopes $\SigmaGKZ$ and $\FibPoly(\Delta^{n-1}\pito Q)$ is given in~\cite[Theorem 2.5]{BS}:
\begin{equation*}\label{eq:BS2}
	\FibPoly(\Delta^{n-1}\pito Q)  = \frac{1}{d \cdot \Voldd(Q)} \SigmaGKZ.
\end{equation*}

\begin{remark}\label{rem:GKZ}
Every fine zonotopal tiling $\Tiling$ gives rise to an $\Acal$-triangulation $\Tr:=\{\Delta_B\mid \tile_{\emptyset,B}\in\Tiling\}$, in which case we denote $\vertGKZ(\Tiling):=\vertGKZ(\Tr)$.
\end{remark}

\subsection{Fiber polytopes for projections of a hypersimplex: hypersecondary polytopes.}
Recall the definitions of $\V$, $\ZV$, $\pi$, $\Hyp_k$, and $Q_k$
from \cref{notation}.
Also recall that $\Delta_{k,n}=\cube_n\cap \{\x\in\R^n\mid x_1+\dots+x_n=k\}$. 
Note that if $k=1$, then $\Delta_{1,n} = \Delta^{n-1}$.  
We discuss the fiber polytope $\FibPoly(\Delta_{k,n}\pito Q_k)$. Such polytopes have been recently studied in~\cite{OS} under the name \emph{hypersecondary polytopes} (not to be confused with \emph{higher secondary polytopes} $\Sigma_{\A,k}$ introduced in this paper).

For integers $r$ and $d$, the  \emph{Eulerian number} $\Eul(d,r)$  is defined as the number of permutations in $S_d$ with $r$ descents, where a \emph{descent} of a permutation $w$ is a position $i$ such that $w_i>w_{i+1}$  (thus $\Eul(d,r)$ is zero if $r\notin[0,d-1]$). For example, we have $\Eul(3,0)=1$, $\Eul(3,1)=4$, $\Eul(3,2)=1$. 

\begin{lemma}\label{lemma:Eul}
Let $\Tiling$ be a fine zonotopal tiling of $\ZV \subset \R^d$.
	Then for all $r\in[d-1]$ and $\tile_{A,B}\in\Tiling$, we have
\begin{equation}\label{eq:Eul}
\Voldd(\tile_{A,B}\cap \Hyp_{|A|+r})=\frac{\Eul(d-1,r-1)}{(d-1)!} \VolB.
\end{equation}
\end{lemma}
\begin{proof}
	We have $\tile_{A,B} = \pi(\cube_{A,B})$ for $A, B$ disjoint subsets
	and $|B| = d$.  The intersection $\tile_{A,B}\cap \Hyp_{|A|+r}$ is the image of a hypersimplex $\Delta_{r,d}\subset\square_{A,B}\cong\cube_d$ under $\pi$. By~\cite{Sta77}, $\Delta_{r,d}$ can be triangulated into $\Eul(d-1,r-1)$ equal-volume simplices, 
	and the image of each of these simplices under $\pi$ has volume $\frac{\VolB}{(d-1)!}$.
\end{proof}
\begin{proposition}
The fiber polytope $\FibPoly(\Delta_{k,n}\pito Q_k)$ equals the convex hull
\begin{align}\label{eq:FibPoly_Delta_k_n}
  \FibPoly(\Delta_{k,n}\pito Q_k)&=\conv\{\vertFib_k(\Tiling)\mid \Tiling\text{ is a fine zonotopal tiling of $\Zon_\VC$}\},\quad\text{where} \\
\label{eq:vertFib_k}
  \vertFib_k(\Tiling)&:=\frac1{d!\cdot \Voldd(Q_k)} \sum_{r=1}^{d-1} \sumlevel{k-r} \VolB\cdot \Eul(d-1,r-1)\cdot  (d\cdot \e_A+r\cdot  \e_B ).
\end{align}
\end{proposition}
\begin{proof}
Let $\Tiling$ be a fine zonotopal tiling of $\ZV$. Then $\Tiling\cap \Hyp_k:=\{\tile_{A,B}\cap \Hyp_k\mid \tile_{A,B}\in\Tiling\}$ is a fine $\pi$-induced subdivision for the projection $\Delta_{k,n}\pito Q_k$. A tile $\tile_{A,B}\in\Tiling$ has a full-dimensional intersection with $\Hyp_k$ whenever $|A|+r=k$ for some $r\in[d-1]$. In this case, $\tile_{A,B}\cap \Hyp_k$ can be triangulated into $\Eul(d-1,r-1)$ simplices as in the proof of \cref{lemma:Eul}. Proceeding as in the proof of \cref{prop:vertFib}, we find that the combined contribution of these simplices to~\eqref{eq:vertFib} is precisely $\frac{\VolB}{d!\cdot \Voldd(Q_k)}(d\cdot \e_A+r\cdot  \e_B )$. Thus we have shown that $\FibPoly(\Delta_{k,n}\pito Q_k)$ contains the right hand side of~\eqref{eq:FibPoly_Delta_k_n}. 

On the other hand, by~\eqref{eq:BS_tiling_reg}, it is enough to consider only \emph{regular} fine $\pi$-induced subdivisions, and every such subdivision clearly arises as $\Tiling\cap\Hyp_k$ for some regular fine zonotopal tiling $\Tiling$. This shows that the right hand side of~\eqref{eq:FibPoly_Delta_k_n} contains $\FibPoly(\Delta_{k,n}\pito Q_k)$.
\end{proof}

\begin{example}
For $d=2$, \eqref{eq:vertFib_k} becomes 
\begin{equation}\label{eq:vertFib_k_d=2}
 \vertFib_k(\Tiling):=\frac1{\Voldd(Q_k)}  \sumlevel{k-1} \VolB\cdot   \left( \e_A+\frac12  \e_B \right).
\end{equation}
\end{example}
\begin{example}
	Substituting $k=1$ into~\eqref{eq:vertFib_k} and comparing the result with~\eqref{eq:vertGKZ}, we find
\begin{align}
\label{eq:GKZ_FibPoly}
  \vertGKZ(\Tiling)
	&=d\cdot \Voldd(Q_1)\cdot \vertFib_1(\Tiling) \quad\text{and}\quad\SigmaGKZ=d\cdot \Voldd(Q_1)\cdot \FibPoly(\Delta_{1,n}\pito Q_1),
\end{align}
	in agreement with \eqref{eq:BS2}.
\end{example}

\section{Vertices of fiber polytopes and vertices of higher secondary polytopes}
Recall the definitions of $\V \subset \R^d, \ZV, H_k$, $Q_k$,
and $\pi$ from \cref{notation}.
Also 
recall that 
$\Delta_{k,n}=\cube_n\cap \{\x\in\R^n\mid x_1+\dots+x_n=k\}$. 
In this section, we 
prove \cref{prop:HSP_vert}, which gives a duality identity, and 
expresses the vertices of fiber polytopes $\FibPoly(\cube_n\pito \Zon_\VC)$, $\SigmaGKZ$, and $\FibPoly(\Delta_{k,n}\pito Q_k)$ as linear combinations of the vectors $\vertHSP_k(\Tiling)$ defined in~\eqref{eq:HSP}. This will constitute one of the main steps in the proof of \cref{thm:sigma}, which we give
in \cref{sec:higher}.

\def\Ar{\beta}
\def\Eupoly{A}
We start by giving a refinement of the 
simple fact that for any fine zonotopal tiling $\Tiling$, the sum $\sum_{\tile_{A,B}\in\Tiling} \VolB$ equals $\Vold(\ZV)$, and therefore does not depend on $\Tiling$. 
For $k\in[n-1]$, we let
 \begin{equation}
	\Ar_k:=\Voldd(Q_k),
\end{equation} 
and we set $\Ar_k:=0$ for $k\notin[n-1]$.

\begin{proposition}
Fix a vector configuration $\V \subset \R^d$ as in \cref{notation}.
	For each $k\in[0,n-d]$, there exists a number $\gamma_k^d(\VC) = \gamma_k(\VC)\in\R_{>0}$ such that for any fine zonotopal tiling $\Tiling$, we have
\begin{equation}\label{eq:gamma_k}
	\gamma_k^d(\VC) = \gamma_k(\VC)=\sumlevel{k} \VolB.
\end{equation}
\end{proposition}
\begin{proof}
Let us temporarily denote  
\[\tilde{\gamma}_k(\Tiling, \VC) := \sumlevel{k} \VolB\]
for all $k\in\Z$. Then $\Ar_k=\Voldd(Q_k) = \Vol^{d-1}(\ZV \cap H_k) = 
	\sum_{\tile_{A,B}\in \Tiling}\Vol^{d-1}( \tile_{A,B} \cap H_k).$
	Applying~\eqref{eq:Eul}, we find that $\Ar_k=\sum_{r=1}^{d-1} \tilde{\gamma}_{k-r}(\Tiling, \VC)\cdot \frac{\Eul(d-1,r-1)}{(d-1)!}$. Since the coefficient of $\tilde{\gamma}_{k-1}$ in the right hand side is equal to $\frac1{(d-1)!}$, the numbers $\tilde{\gamma}_k(\Tiling, \VC)$ can be expressed in terms of the $\Ar_r$'s by induction for $k=0,1,\dots, n-d$. Explicitly, let $\Eupoly_{d-1}(x):=\sum_{r=0}^{d-2} \Eul(d-1,r)x^r$ be the \emph{Eulerian polynomial}, and let $c_0,c_1,\dots\in\Z$ be defined by $\frac1{\Eupoly_{d-1}(x)}=c_0+c_1x+c_2x^2+\dots$ (thus $c_0=1$). Then we have $\tilde{\gamma}_k(\Tiling, \VC)=c_0\Ar_{k+1}+c_1\Ar_k+c_2\Ar_{k-1}+\dots+c_{k+1}\Ar_0$ for all $k\in[0,n-d]$. It is clear that $\tilde{\gamma}_k(\Tiling, \VC)$ does not depend on $\Tiling$, and so we can refer to it as 
	$\gamma_k(\VC)$.
\end{proof}
\begin{example}
For $d=2,3,4$, we have respectively
	\begin{align}
\label{eq:gamma_d=2}
\gamma_k^2(\VC)&=\Ar_{k+1}, \\ 
  \gamma_k^3(\VC)&=\Ar_{k+1}-\Ar_{k}+\dots+(-1)^{k+1}\Ar_0, \\
\label{eq:Ar_d=4}
\gamma_k^4(\VC)&=\Ar_{k+1}-4\Ar_{k}+15\Ar_{k-1}-56\Ar_{k-2}+\dots, 
\end{align}
where the coefficients of~\eqref{eq:Ar_d=4} form the sequence \href{http://oeis.org/A125905}{A125905} in the OEIS~\cite{OEIS}.
\end{example}

For $i\in[n]$, let $\VC-i$ denote the vector configuration in $\R^d$ obtained from $\VC$ by omitting $\v_i$. For each $k\in[0,n-d]$, we introduce a vector $\ddelta(k,\VC)\in\R^n$ whose $i$th coordinate equals 
\begin{equation}\label{eq:delta}
  \delta_i(k,\VC):=\gamma_k(\VC)-\gamma_k(\VC-i) \quad\text{for all $i\in[n]$}.
\end{equation}
\noindent For $k\notin [0,n-d]$, we set $\gamma_k(\VC):=0\in\R$ and $\ddelta(k,\VC):=0\in\R^n$. Recall that the vectors of $\VC$ are assumed to linearly span $\R^d$. If the vectors of $\VC-i$ all belong to a lower-dimensional subspace of $\R^d$, we say that $i$ is a \emph{coloop} and set $\gamma_k(\VC-i):=0$ for all $k$.

 The following result will be useful in the proof of \cref{prop:HSP_vert}.
\begin{proposition}\label{prop:adjacent_levels}
For all $k\in[0,n-d]$, we have
\begin{equation}\label{eq:adjacent_levels}
\sumlevel{k} \VolB \cdot (\e_A+\e_B)=\sumlevel{k+1} \VolB \cdot \e_A +\ddelta(k,\VC).
\end{equation}
\end{proposition}
\begin{proof}
Fix $i\in[n]$. We first show that 
\begin{equation}\label{eq:adjacent_levels2}
\sumlevel{k+1,\ i\in A} \VolB+\sumlevel{k,\ i\notin A\sqcup B} \VolB=\gamma_k(\VC-i).
\end{equation}
Assume that $i$ is a coloop, which means that the vectors in $\VC-i$ do not linearly span $\R^d$, in which case the right hand side of~\eqref{eq:adjacent_levels2} is zero. On the other hand, for each tile $\tile_{A,B}\in\Tiling$, we must have $i\in B$, which shows that the left hand side of~\eqref{eq:adjacent_levels2} is also zero. Assume now that $i$ is not a coloop. Then each fine zonotopal tiling $\Tiling$ of $\ZV$ gives rise to a fine zonotopal tiling $\Tiling-i$ of $\Zon_{\VC-i}$ defined by 
\[\Tiling-i:=\{\tile_{A\setminus\{i\},B}\mid \tile_{A,B}\in\Tiling,\  i\in A\}\sqcup \{\tile_{A,B}\mid \tile_{A,B}\in\Tiling,\  i\notin A\sqcup B\}.\]
Using this observation, we see that~\eqref{eq:adjacent_levels2} follows from the definition~\eqref{eq:gamma_k} of $\gamma_k(\VC-i)$. For the example in \cref{fig:deletion}, for $k=1$, the left hand side of~\eqref{eq:adjacent_levels2} is equal to $3+2$ as shown in \cref{fig:deletion} (middle) while the right hand side of~\eqref{eq:adjacent_levels2} is equal to $5$ as shown in \cref{fig:deletion} (right).

\begin{figure}

\scalebox{0.68}{

\begin{tikzpicture}

\def\gridlw{1pt}
\def\gridop{0.25}
\def\lw{1.5pt}
\def\lww{1.8pt}
\def\lwww{2.5pt}

\def\setscl{0.8}
\def\vscl{1.2}
\def\VCscl{1.5}
\def\Tilingscl{1.8}
\def\Xscl{1.3}
\def\tnode#1#2#3{\node[draw,fill,circle,scale=0.4](#1) at (#2,#3) {};}
\def\tcoord#1#2#3{\coordinate(N#1) at (#2,#3);}
\def\lnode#1#2{\node[anchor=180+#2,scale=\setscl] (A#1) at (N#1.#2) {$#1$};}
\def\lnodv#1#2{\node[anchor=180+#2,scale=\vscl] (A#1) at (N#1.#2) {$\v_{#1}$};}
\def\lnod#1#2#3{\node[anchor=-#2,scale=\setscl] (A#1) at (N#1.#2) {#3};}
\def\stp{9}

\def\blueop{0.6}

\def\xmin{0}
\def\xmax{6}
\def\ymin{0}
\def\ymax{5}

\draw[line width=\gridlw,opacity=\gridop] (\xmin,\ymin) grid (\xmax,\ymax);

\tnode a30
\tnode b51
\tnode c62
\tnode d63
\tnode e54
\tnode f35
\tnode g14
\tnode h03
\tnode i02
\tnode j11
\tnode k31
\tnode l52
\tnode m42
\tnode n33
\tnode o22
\tnode p21

\draw[line width=\lw] (o)--(p)--(a)--(b)--(c)--(d)--(e)--(f)--(g)--(h)--(i)--(j)--(a)--(k);
\draw[line width=\lw] (h)--(o)--(k)--(l)--(b);
\draw[line width=\lw] (l)--(d)--(m)--(k);
\draw[line width=\lw] (m)--(n)--(o);
\draw[line width=\lw] (g)--(n)--(e);
\draw[line width=\lw] (i)--(p);

\begin{scope}[xshift=\stp cm]

\draw[line width=\gridlw,opacity=\gridop] (\xmin,\ymin) grid (\xmax,\ymax);

\tnode a30
\tnode b51
\tnode c62
\tnode d63
\tnode e54
\tnode f35
\tnode g14
\tnode h03
\tnode i02
\tnode j11
\tnode k31
\tnode l52
\tnode m42
\tnode n33
\tnode o22
\tnode p21

\draw[fill=black,opacity=0.2,draw=white] (b.center)--(l.center)--(k.center)--(o.center)--(h.center)--(g.center)--(n.center)--(m.center)--(d.center)--(c.center)--cycle;

\draw[line width=\lw]  (o)--(p)--(a)--(b);
\draw[line width=\lw] (c)--(d)--(e)--(f)--(g);
\draw[line width=\lw] (h)--(i)--(j)--(a)--(k);
\draw[line width=\lw] (h)--(o)--(k)--(l)--(b);
\draw[line width=\lw] (d)--(m);
\draw[line width=\lw] (m)--(n);
\draw[line width=\lw] (g)--(n)--(e);
\draw[line width=\lw] (i)--(p);

\draw[line width=\lww,->,>=latex,red] (b)--(c);
\draw[line width=\lww,->,>=latex,red] (l)--(d);
\draw[line width=\lww,->,>=latex,red] (k)--(m);
\draw[line width=\lww,->,>=latex,red] (o)--(n);
\draw[line width=\lww,->,>=latex,red] (h)--(g);

\draw[line width=\lwww,blue,opacity=\blueop,dotted] (i)--(o);
\draw[line width=\lwww,blue,opacity=\blueop,dotted] (n)--(d);
\node[anchor=north west,scale=\VCscl,red,inner sep=1pt](A) at (5.5,1.5) {$\v_i$};
\node[scale=\Xscl,blue] (BOT) at (-1,-0.3){$\displaystyle\sumlevel{k,\ i\notin A\sqcup B} \VolB$};
\node[scale=\Xscl,blue] (TOP) at (7,5.2){$\displaystyle\sumlevel{k+1,\ i\in A} \VolB$};
\draw[line width=\lw,->,blue] (BOT.80)to[bend right=20] (1,1.9);
\draw[line width=\lw,->,blue] (TOP)to[bend right=40] (4.5,3.1);

\end{scope}

\begin{scope}[xshift=2*\stp cm]

\draw[line width=\gridlw,opacity=\gridop] (\xmin,\ymin) grid (\xmax,\ymax);

\tnode a30
\tnode b51
\tnode e43
\tnode f24
\tnode h03
\tnode i02
\tnode j11
\tnode k31
\tnode l52
\tnode o22
\tnode p21

\draw[line width=\lw] (a)--(b)--(l)--(e)--(f)--(h)--(i)--(j)--(a)--(k)--(o)--(h);
\draw[line width=\lw] (i)--(p)--(a);
\draw[line width=\lw] (p)--(o)--(e);
\draw[line width=\lw] (k)--(l);

\draw[line width=\lwww,blue,opacity=\blueop,dotted] (i)--(l);

\node[scale=\Xscl,blue] (TOP) at (5,5.2){$\gamma_k(\VC-i)$};
\draw[line width=\lw,->,blue] (TOP)to[bend right=40] (2.7,2.1);

\end{scope}

\node[anchor=north,scale=\Tilingscl] (AA) at (3,-0.3) {$\Tiling$};
\node[anchor=north,scale=\Tilingscl] (AA) at (3+2*\stp,-0.3) {$\Tiling-i$};

\end{tikzpicture}

}

  \caption{\label{fig:deletion} Deleting $\v_i$ from $\VC$  and its effect on a tiling  $\Tiling$, see~\eqref{eq:adjacent_levels2}.}
\end{figure}
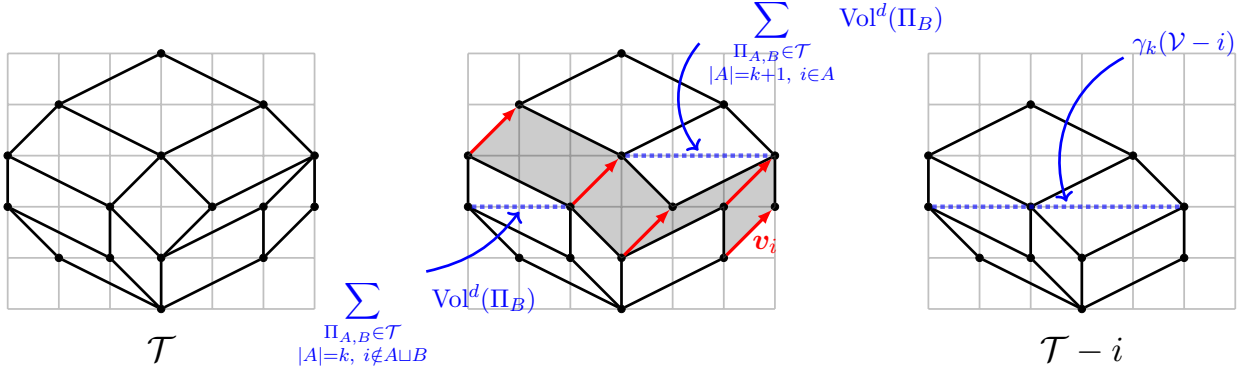 

To prove \eqref{eq:adjacent_levels}, it is enough to verify what it says
for the $i$th coordinate, which is:
\begin{equation}\label{eq:adjacent_levels3}
\sumlevel{k,\ i\in A\sqcup B} \VolB =\sumlevel{k+1,\ i\in A} \VolB+\delta_i(k,\VC).
\end{equation}
\noindent Adding $\displaystyle\sumlevel{k,\ i\notin A\sqcup B} \VolB$ to both sides of~\eqref{eq:adjacent_levels3} and applying~\eqref{eq:adjacent_levels2} gives $\gamma_k(\VC)=\gamma_k(\VC-i)+\delta_i(k,\VC)$, which is precisely the definition~\eqref{eq:delta} of $\ddelta(k,\VC)$.
\end{proof}

\begin{corollary}\label{cor:x_k_y_k}
Recall the definition of 
$\vertHSP_k(\Tiling)$
	from \eqref{eq:HSP}.
Let $K\subset \Z$ and choose some numbers $x_k,y_k\in \R$ for each $k\in K$. Then 
\begin{equation}\label{eq:cor_x_k_y_k}
\sumcond{k:=|A|\in K} \VolB\cdot (x_k\e_A+y_k\e_B)=  \sum_{k\in K} \left((x_k-y_k)\vertHSP_k(\Tiling) +y_k(\vertHSP_{k+1}(\Tiling)+\ddelta(k,\VC))\right).
\end{equation}
\end{corollary}
\begin{proof}
This follows by replacing $x_k\e_A+y_k\e_B$ on the left hand side of 
	\eqref{eq:cor_x_k_y_k} with $(x_k-y_k)\e_A+y_k(\e_A+\e_B)$, and applying \cref{prop:adjacent_levels}.
\end{proof}

\def\TilingOP{\Tiling^{\operatorname{op}}}
\def\one_#1{\e_{[#1]}}

\begin{definition}\label{def:op}
Given two disjoint sets $A,B\subset[n]$, let $C:=[n]\setminus(A\sqcup B)$, and denote $\tile_{A,B,C}:=\tile_{A,B}$. For each zonotopal tiling $\Tiling$ of $\ZV$ there exists ``the opposite'' zonotopal tiling $\TilingOP$ of $\ZV$ given by $\TilingOP:=\{\tile_{C,B,A}\mid \tile_{A,B,C}\in\Tiling\}$, see \cref{fig:TilingOP}.
\end{definition}

\begin{theorem}\label{prop:HSP_vert}
Recall the definitions of 
  $\vertFib(\Tiling)$, 
	$\vertFib_k(\Tiling)$, and $\vertGKZ(\Tiling)$
from 	\eqref{eq:vertFibzon},
 \eqref{eq:vertFib_k}, and 
\cref{rem:GKZ}.
We have 	\begin{align}
\label{eq:HSP_vertFib}
  &\vertFib(\Tiling)=\frac1{\Vold(\ZV)} \left( \sum_{k=1}^{n-d} \vertHSP_k(\Tiling) +\frac12\sum_{k=0}^{n-d}\ddelta(k,\VC) \right);\\
\label{eq:HSP_vertFib_k}
  &\vertFib_k(\Tiling)=\frac1{\Voldd(Q_k)} \left( \sum_{r=0}^{d-1} \frac{\Eul(d,r)}{d!}\vertHSP_{k-r}(\Tiling) + \sum_{r=1}^{d-1}\frac{r\cdot \Eul(d-1,r-1)}{d!}\ddelta(k-r,\VC) \right);\\
		\label{eq:GKZ}	& \vertGKZ(\Tiling)=\frac1{(d-1)!} \left( \vertHSP_{1}(\Tiling) + \ddelta(0,\VC) \right); \\
\label{eq:HSP_duality}
  &\vertHSP_{k}(\Tiling)+\vertHSP_{n-d-k+1}(\TilingOP) =\gamma_{k-1}(\VC)\cdot \one_n-\ddelta(k-1,\VC).
\end{align}
\end{theorem}
\begin{proof}
Applying \cref{cor:x_k_y_k}  to~\eqref{eq:vertFibzon} 
	with $K=[0,n-d]$, $x_k=\frac1{\Vol(\ZV)}$, and $y_k=\frac1{2\Vol(\ZV)}$ for all $k\in K$, we obtain~\eqref{eq:HSP_vertFib}. 

	Similarly, applying \cref{cor:x_k_y_k}  to~\eqref{eq:vertFib_k} with $K=[k-d+1,k-1]$,  
	$x_{k-r}=\frac{d\cdot \Eul(d-1,r-1)}{d!\Voldd(Q_k)}$, and $y_{k-r}=\frac{r\cdot \Eul(d-1,r-1)}{d!\Voldd(Q_k)}$ for all $r\in [d-1]$, we get
\begin{align*}
  \vertFib_k(\Tiling)&=\frac1{\Voldd(Q_k)} \sum_{r=1}^{d-1}  \left( \frac{(d-r) \cdot \Eul(d-1,r-1)}{d!} \vertHSP_{k-r}(\Tiling)+\frac{r \cdot \Eul(d-1,r-1)}{d!} (\vertHSP_{k-r+1}(\Tiling)+\ddelta(k-r,\VC))\right)\\
   &=\frac1{\Voldd(Q_k)}\sum_{r=0}^{d-1}  \left(  \frac{(d-r)\cdot \Eul(d-1,r-1)+(r+1)\cdot \Eul(d-1,r)}{d!} \vertHSP_{k-r}(\Tiling)+\frac{r\cdot \Eul(d-1,r-1)}{d!}\ddelta(k-r,\VC)\right).
\end{align*}

\noindent Applying the well known recurrence $(d-r)\cdot \Eul(d-1,r-1)+(r+1)\cdot \Eul(d-1,r)=\Eul(d,r)$ for Eulerian numbers yields~\eqref{eq:HSP_vertFib_k}.

	For $k=1$,  combining~\eqref{eq:HSP_vertFib_k} with~\eqref{eq:GKZ_FibPoly} yields~\eqref{eq:GKZ}.

Finally, to show~\eqref{eq:HSP_duality}, we use $|A|+|C|=n-d$ and~\eqref{eq:HSP}
	to write
\begin{align*}
\vertHSP_{n-d-k+1}(\TilingOP)&=\summa{\tile_{C,B,A}\in\TilingOP}{|C|=n-d-k+1}\VolB\cdot \e_C=\sumlevel{k-1}\VolB\cdot (\e_{[n]}-\e_A-\e_B),
\end{align*}
and by~\eqref{eq:gamma_k} and~\eqref{eq:adjacent_levels}, this is equal to
\[
\gamma_{k-1}(\VC)\cdot \e_{[n]}-\sumlevel{k-1}\VolB\cdot (\e_A+\e_B)=\gamma_{k-1}(\VC)\cdot \e_{[n]}-\ddelta(k-1,\VC)-\vertHSP_{k}(\Tiling).\qedhere
\]
\end{proof}

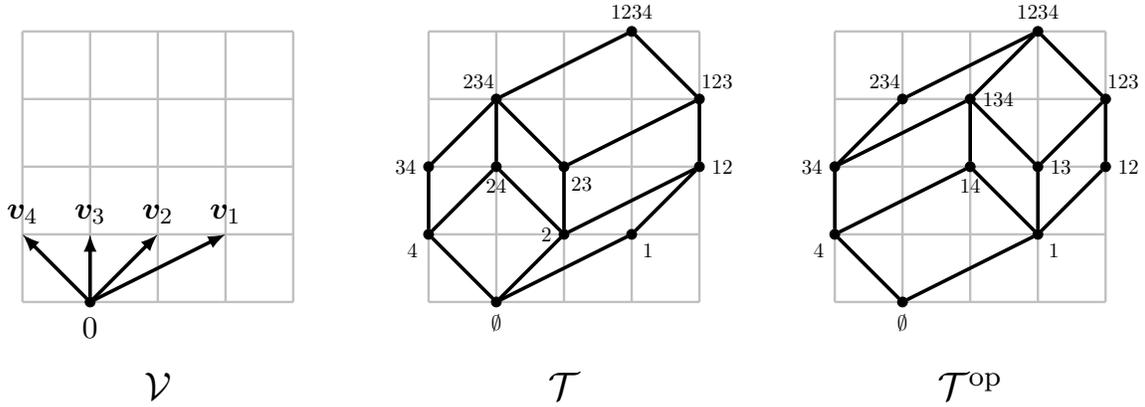
\begin{figure}
\scalebox{0.9}{

\begin{tikzpicture}

\def\gridlw{1pt}
\def\gridop{0.25}
\def\lw{1.5pt}
\def\setscl{0.8}
\def\vscl{1.2}
\def\VCscl{1.5}
\def\tnode#1#2#3{\node[draw,fill,circle,scale=0.4](N#1) at (#2,#3) {};}
\def\tcoord#1#2#3{\coordinate(N#1) at (#2,#3);}
\def\lnode#1#2{\node[anchor=180+#2,scale=\setscl] (A#1) at (N#1.#2) {$#1$};}
\def\lnodv#1#2{\node[anchor=180+#2,scale=\vscl] (A#1) at (N#1.#2) {$\v_{#1}$};}
\def\lnod#1#2#3{\node[anchor=-#2,scale=\setscl] (A#1) at (N#1.#2) {#3};}
\def\stp{6}

\def\xmin{0}
\def\xmax{4}
\def\ymin{0}
\def\ymax{4}
\def\marg{0.7}

\node(B) at (0,0){
  \begin{tikzpicture}
\draw[white] (-\marg,-\marg) rectangle (\xmax+\marg,\ymax+\marg);
\draw[line width=\gridlw,opacity=\gridop] (\xmin,\ymin) grid (\xmax,\ymax);

\tnode{0}10
\tnode{1}31
\tnode{2}21
\tnode{4}01
\tnode{12}42
\tnode{23}22
\tnode{24}12
\tnode{34}02
\tnode{123}43
\tnode{234}13
\tnode{1234}34

\draw[line width=\lw] (N0)--(N1)--(N12)--(N123)--(N1234)--(N234)--(N34)--(N4)--(N0);
\draw[line width=\lw] (N4)--(N24)--(N2)--(N12);
\draw[line width=\lw] (N234)--(N23)--(N123);
\draw[line width=\lw] (N234)--(N24);
\draw[line width=\lw] (N23)--(N2)--(N0);

\lnod{0}{270}{$\emptyset$}
\lnode{1}{-45}
\lnode{12}{0}
\lnode{123}{45}
\lnode{1234}{90}
\lnode{234}{135}
\lnode{34}{180}
\lnode{4}{225}
\lnode{2}{180}
\lnode{24}{-90}
\lnode{23}{-45}
\end{tikzpicture}

};

\node(C) at (\stp,0){
  \begin{tikzpicture}

\draw[white] (-\marg,-\marg) rectangle (\xmax+\marg,\ymax+\marg);
\draw[line width=\gridlw,opacity=\gridop] (\xmin,\ymin) grid (\xmax,\ymax);

\tnode{0}10
\tnode{1}31
\tnode{4}01
\tnode{12}42
\tnode{13}32
\tnode{14}22
\tnode{34}02
\tnode{123}43
\tnode{234}13
\tnode{134}23
\tnode{1234}34

\draw[line width=\lw] (N0)--(N1)--(N12)--(N123)--(N1234)--(N234)--(N34)--(N4)--(N0);
\draw[line width=\lw] (N4)--(N14)--(N1)--(N13);
\draw[line width=\lw] (N34)--(N134)--(N13)--(N123);
\draw[line width=\lw] (N1234)--(N134)--(N14);

\lnod{0}{270}{$\emptyset$}
\lnode{1}{-45}
\lnode{12}{0}
\lnode{123}{45}
\lnode{1234}{90}
\lnode{234}{135}
\lnode{34}{180}
\lnode{4}{225}
\lnode{14}{-90}
\lnode{13}{0}
\lnode{134}{0}
\end{tikzpicture}

};

\node(A) at (-\stp,0){
  \begin{tikzpicture}

\draw[white] (-\marg,-\marg) rectangle (\xmax+\marg,\ymax+\marg);
\draw[line width=\gridlw,opacity=\gridop] (\xmin,\ymin) grid (\xmax,\ymax);

\tnode{0}10
\tcoord{1}31
\tcoord{2}21
\tcoord{3}11
\tcoord{4}01

\draw[line width=\lw,->,>=latex] (N0)--(N1);
\draw[line width=\lw,->,>=latex] (N0)--(N2);
\draw[line width=\lw,->,>=latex] (N0)--(N3);
\draw[line width=\lw,->,>=latex] (N0)--(N4);
\node[anchor=90,scale=\vscl] (A) at (N0.-90) {$0$};
\lnodv{1}{90}
\lnodv{2}{90}
\lnodv{3}{90}
\lnodv{4}{90}

\end{tikzpicture}
};

\node[anchor=north,scale=\VCscl] (AA) at (A.south) {$\VC$};
\node[anchor=north,scale=\VCscl] (AA) at (B.south) {$\Tiling$};
\node[anchor=north,scale=\VCscl] (AA) at (C.south) {$\TilingOP$};

\end{tikzpicture}

}
  \caption{\label{fig:TilingOP} A vector configuration $\VC$, a fine zonotopal tiling $\Tiling$ of $\ZV$, and its ``opposite'' tiling $\TilingOP$ for $d=2$ and $n=4$. We label each vertex $\v_{i_1}+\dots+\v_{i_k}$ by $i_1\cdots i_k$.}
\end{figure}

\begin{example}\label{ex:d23}
	For $d=2$,~\eqref{eq:HSP_vertFib_k} becomes 
\begin{align}
\label{eq:HSP_vertFib_k_d=2}
  \vertFib_k(\Tiling)&=\frac1{2\Vol^1(Q_k)} \left(\vertHSP_{k}(\Tiling) +\vertHSP_{k-1}(\Tiling) +\ddelta(k-1,\VC) \right).
\end{align}
\end{example}

\def\vect#1#2#3#4{(#1,#2,#3,#4)}
\begin{example}\label{ex:gamma}
Consider the case $n=4$, $d=2$, and let $\VC$ be the vector configuration given in \cref{fig:TilingOP} (left), so the vectors $\v_1,\v_2,\v_3,\v_4$ of $\VC$ are the column vectors of the matrix $\begin{pmatrix}
2 & 1 & 0 &-1\\
1 & 1 & 1 & 1
\end{pmatrix}$. If $B=\{i,j\}$ for $1\leq i<j\leq 4$ then $\VolB=j-i$. We have
\[\Vold(\ZV)=10,\quad \Voldd(Q_1)=3,\quad \Voldd(Q_2)=4,\quad \Voldd(Q_3)=3,\]
	where $\Vold(\ZV)$ is the area of $\ZV$ and $\Voldd(Q_k)$ is the length of the horizontal section of $\ZV$ by the line $y_2=k$. By~\eqref{eq:gamma_d=2}, $\gamma_k(\VC)$ is equal to $\Ar_{k+1}=\Voldd(Q_{k+1})$.  Using this to compute $\gamma_k(\VC)$ (and also $\gamma_k(\VC - i)$ for $i=1,2,3,4$), we get
\begin{align*}
\gamma_0(\VC)&=3, &\gamma_1(\VC)&=4, &\gamma_2(\VC)&=3;\\
\ddelta(0,\VC)&=\vect1001,
&\ddelta(1,\VC)&=\vect2112,
&\ddelta(2,\VC)&=\vect3333. 
\end{align*}
Let $\Tiling$ and $\TilingOP$ be as in \cref{fig:TilingOP}. The corresponding vertices of the higher secondary polytopes are given by
\[\vertHSP_1(\Tiling)=\vect0301,\quad 
\vertHSP_2(\Tiling)=\vect0330,\quad 
\vertHSP_1(\TilingOP)=\vect2002,\quad 
\vertHSP_2(\TilingOP)=\vect2031. \]

We would like to verify the formulas from \cref{prop:HSP_vert}. First,~\eqref{eq:HSP_duality} clearly holds: for $k=1$ and $k=2$, we have 
\begin{align*}
\vertHSP_1(\Tiling)+\vertHSP_2(\TilingOP)&=\vect2332, &\gamma_0(\VC)\e_{[n]}-\ddelta(0,\VC)&=3\cdot \vect1111-\vect1001=\vect2332,\\
\vertHSP_2(\Tiling)+\vertHSP_1(\TilingOP)&=\vect2332, &\gamma_1(\VC)\e_{[n]}-\ddelta(1,\VC)&=4\cdot \vect1111-\vect2112=\vect2332.
\end{align*}
\def\ee#1{\e_{\{#1\}}}
Using~\eqref{eq:vertFibzon} and~\eqref{eq:vertFib_k_d=2}, we find
\begin{align*}
  \vertFib(\Tiling)=&\frac1{10} \Biggl(2 \frac{\e_{\{2,4\}}}2+\frac{\e_{\{1,2\}}}2+ \left(\e_4+\frac{\e_{\{2,3\}}}2\right)+\left(\e_2+\frac{\e_{\{3,4\}}}2\right)+2\left(\e_2+\frac{\e_{\{1,3\}}}2\right)
\Biggr.\\
 &\Biggl.+3 \left(\e_{\{2,3\}}+\frac{\e_{\{1,4\}}}2\right)\Biggr)=\frac1{10}\vect3854;\\
  \vertFib_1(\Tiling)=&\frac1{3}   \left(2\frac{\ee{2,4}}2+\frac{\ee{1,2}}2 \right)=\frac16\vect1302;\\
  \vertFib_2(\Tiling)=&\frac1{4}   \left(\left(\e_4+\frac{\e_{\{2,3\}}}2\right)+\left(\e_2+\frac{\e_{\{3,4\}}}2\right)+2\left(\e_2+\frac{\e_{\{1,3\}}}2\right)\right)=\frac18\vect2743;\\
  \vertFib_3(\Tiling)=&\frac13\cdot 3 \left(\e_{\{2,3\}}+\frac{\e_{\{1,4\}}}2\right)=\frac12\vect1221.
\end{align*}
We indeed see that~\eqref{eq:HSP_vertFib}  and~\eqref{eq:HSP_vertFib_k} (which specializes to~\eqref{eq:HSP_vertFib_k_d=2} for $d=2$) hold as well:
\begin{align*} \vertFib(\Tiling)&=\frac1{10}\vect3854=\frac1{10}\left(\vertHSP_1(\Tiling)+\vertHSP_2(\Tiling)+\frac12 \left(\ddelta(0,\VC)+\ddelta(1,\VC)+\ddelta(2,\VC)\right)\right);\\
  \vertFib_1(\Tiling)&=\frac16\vect1302=\frac1{2\cdot 3} \left(\vertHSP_1(\Tiling)+0+\ddelta(0,\VC)\right);\\
  \vertFib_2(\Tiling)&=\frac18\vect2743=\frac1{2\cdot 4} \left(\vertHSP_2(\Tiling)+\vertHSP_1(\Tiling)+\ddelta(1,\VC)\right);\\
  \vertFib_3(\Tiling)&=\frac16\vect3663=\frac1{2\cdot 3} \left(0+\vertHSP_2(\Tiling)+\ddelta(2,\VC)\right).
\end{align*}
\end{example}

\section{Flips of zonotopal tilings}\label{sec:flips-zonot-tilings}
Zonotopal tilings form a poset under refinement whose minimal elements are fine zonotopal tilings. Two fine zonotopal tilings \emph{differ by a flip} (cf. \cref{dfn:flip}) if there exists a zonotopal tiling that covers both of them in this poset.  
In this section we describe 
(see \cref{cor:flip1,cor:flip2}) 
how the vectors 
$\vertHSP_k(\Tiling)$ and $\vertHSP_k(\Tiling')$
differ when the 
 fine zonotopal tilings  $\Tiling$ and $\Tiling'$ differ by a flip.  
This will be useful in \cref{sec:regular} for describing
the $1$-skeleton of a higher secondary polytope.

\def\Bases{\Bcal(\Mcal)}
\def\Circuits{\Ccal(\Mcal)}
\def\sC^#1{\suppC^{(#1)}}
\def\setmin#1#2{#1^{(#2)}}
\def\BMC{\Bcal(\Mcal/\suppC)}

\subsection{Oriented matroids and signed circuits}

Each vector configuration $\VC=(\v_1,\dots,\v_n)$ spanning $\R^d$ defines a
rank $d$ \emph{oriented matroid} $\Mcal=\Mcal_\VC$. We refer to~\cite{RedBook} for the definition of an oriented matroid, but note that it is completely determined by its set $\Circuits$ of \emph{circuits} introduced below. We denote by $\Bases$ the collection of \emph{bases} of $\VC$, that is, 
$d$-element subsets $B\subset [n]$ such that the vectors $\{\v_i\}_{i\in B}$ form a linear basis of $\R^d$. We say that the vector configuration $\VC$ is \emph{generic} if $\Bases={[n]\choose d}:=\{B\subset [n]\mid |B|=d\}$, that is, if every $d$ vectors of $\VC$ form a basis of $\R^d$. An \emph{independent set} is a subset $I\subset [n]$ such that there is a basis $B\in\Bases$ satisfying $I\subset B$.

Let us mention a well known property of fine zonotopal tilings, see \cref{fig:TilingOP} for an example.
\begin{proposition}[{\cite[(56)]{Shephard}}]\label{prop:bases_bijection}
Let $\Tiling$ be a fine zonotopal tiling of $\ZV$. Then the map $\tile_{A,B}\mapsto B$ is a bijection between $\Tiling$ and $\Bases$. In other words, for each basis $B\in\Bases$ of $\VC$, there exists a unique set $A\subset([n]\setminus B)$ such that $\tile_{A,B}$ belongs to $\Tiling$.
\end{proposition}

\begin{definition}\label{def:signedsubset}
	A \emph{signed set} is a pair
	$X=(X^+, X^-)$ of disjoint subsets of $[n]$.  Its
	\emph{support} is $\suppX := X^+ \sqcup X^-$, and we set  $X^0:= [n] \setminus \underline{X}$, thus $[n]=X^+\sqcup X^0\sqcup X^-$.  For  each  
	$j\in [n]$ we write 
\begin{equation}\label{eq:signed_set}
	X_j = \begin{cases}
		 +1, &\text{ if }j \in X^+; \\
		 -1, &\text{ if }j \in X^-; \\
		0, &\text{ if } j \in X^0.
	\end{cases}
\end{equation}
For $j\in \suppX$, we denote $\setmin \suppX j:=\suppX\setminus\{j\}$. We also let $-X:=(X^-,X^+)$ denote the \emph{opposite} signed set.
\end{definition}

\def\alphaC{\bm{\alpha}(C)}
\def\alphaX(#1){\alpha(#1)}
\def\alphaCJ{\bm{\alpha}(C,J)}
\def\alphaXX(#1,#2){\bm{\alpha}(#1,#2)}
\def\alphaCx_#1{\alpha_{#1}(C)}
\def\alphaCJx_#1{\alpha_{#1}(C,J)}
\begin{definition}\label{dfn:alphaC}
A \emph{circuit}  of $\VC$ is a signed set $C=(C^+,C^-)$ such that  $\setmin \suppC j$ is an independent set for each $j\in\suppC$, but there exists a vector $\alphaC\in\R^n$ satisfying
\[\text{$\alphaCx_j>0$ for $j\in C^+$,}\quad \text{$\alphaCx_j<0$ for $j\in C^-$,}\quad \text{$\alphaCx_j=0$ for $j\in C^0$},\quad\text{and}\quad\text{$\sum_{j\in\suppC}\alphaCx_j \v_j=0$.}\]
Such a  vector $\alphaC$ is unique up to multiplication by a positive real number. We denote by $\Circuits$ the collection of all circuits of $\VC$.
\end{definition}

\def\Vert{\operatorname{Vert}}

Throughout, for $A\subset [n]$ and $j\in[n]$, we abbreviate $A\cup j:=A\cup\{j\}$ and $A\eltminus j:=A\setminus\{j\}$.
\subsection{Circuit orientations}\label{sec:circuit-orientations}
A convenient way to work with flips of fine zonotopal tilings is to use the language of \emph{circuit orientations}.
\begin{definition}
	A \emph{circuit orientation}  is a map $\sigma:\Circuits\to \{+1,0,-1\}$ satisfying $\sigma(-C)=-\sigma(C)$ for all $C\in\Circuits$. We say that $\sigma$ is \emph{generic} if $\sigma(C)\in\{+1,-1\}$ for all $C\in\Circuits$.
\end{definition}

 We describe a way to associate a generic circuit orientation (called \emph{colocalization} in~\cite{GaPo} because they are dual to the localizations of 
 \cite[Definition 7.1.5]{RedBook}) to each fine zonotopal tiling $\Tiling$ of $\ZV$. Let $\Tiling$ be such a tiling. Define its set of \emph{vertex labels} (cf. \cref{fig:TilingOP}) by 
\begin{equation}\label{eq:dfn_Vert}
	\Vert(\Tiling):=\{I\subset [n]\mid A\subset I\subset A\sqcup B\text{ for some $\tile_{A,B}\in\Tiling$}\}.\footnote{
 Given a fine zonotopal tiling $\Tiling$, the collection $\Vert(\Tiling)$ defined in~\eqref{eq:dfn_Vert} coincides with the collection defined in~\cite[Eq.~(2.1)]{GaPo}. The two definitions look slightly different because in~\cite{GaPo}, a tiling is a collection of faces
of all different dimensions, whereas here we identify a tiling with its collection
of top-dimensional faces.} 
\end{equation}
Given a set $S\subset [n]$ and a circuit $C\in\Circuits$, we say that $S$ \emph{orients $C$ positively} if $C^+\subset S$ and $C^-\cap S=\emptyset$. Similarly, we say that $S$ \emph{orients $C$ negatively} if $C^-\subset S$ and $C^+\cap S=\emptyset$. We say that a collection $\WS\subset 2^{[n]}$ \emph{orients $C$ positively} if some set in $\WS$ orients $C$ positively but no set in $\WS$ orients $C$ negatively. Similarly, we say that a collection $\WS\subset 2^{[n]}$ \emph{orients $C$ negatively} if some set in $\WS$ orients $C$ negatively but no set in $\WS$ orients $C$ positively.

\begin{proposition}[{\cite[Theorem~2.7 and Corollary~7.22]{GaPo}}]\label{prop:coloc_pm}
Let $\Tiling$ be a fine zonotopal tiling of $\ZV$ and let $C\in\Circuits$. Then the collection $\Vert(\Tiling)$ either orients $C$ positively or orients $C$ negatively (but not both).
\end{proposition}
\noindent Note that \cref{prop:coloc_pm} can alternatively be deduced by combining Proposition~2.2.11, Theorem~2.2.13, and Proposition~7.1.4 of~\cite{RedBook}. 
We define a generic circuit orientation $\sigma_\Tiling:\Circuits\to\{+1,-1\}$ by setting 
\begin{equation}\label{eq:sigma_tiling}
\sigma_\Tiling(C):=
  \begin{cases}
    +1, &\text{if $\Vert(\Tiling)$ orients $C$ positively,}\\
    -1, &\text{if $\Vert(\Tiling)$ orients $C$ negatively,}\\
  \end{cases}\qquad \text{for all $C\in\Circuits$.}
\end{equation}

\begin{definition}\label{dfn:flip}
Consider two fine zonotopal tilings $\Tiling,\Tiling'$ of $\ZV$, and let $\sigma:=\sigma_\Tiling$, $\sigma':=\sigma_{\Tiling'}$ be the corresponding generic circuit orientations. We say that $\Tiling$ and $\Tiling'$ \emph{differ by a flip} if there exists a circuit $C\in\Circuits$ such that $\sigma(C)=+1$, $\sigma'(C)=-1$ and $\sigma(X)=\sigma'(X)$ for all $X\in\Circuits$ such that $X\neq\pm C$. In this case, we denote this flip by $F:=(\Tiling\to\Tiling')$ and say that $F$ is a flip \emph{along $C$}.
\end{definition}
\noindent Our next goal is to describe the effect of a flip $F=(\Tiling\to\Tiling')$ on the tiles of $\Tiling$ and on $\vertHSP_k(\Tiling)$.

\subsection{Flips for generic vector configurations}
Recall that a vector configuration $\VC$ is called \emph{generic} if $\Bases={[n]\choose d}$. Before proceeding to the general case, we describe flips of zonotopal tilings and their effect on the vertices of higher secondary polytopes in the case when $\VC$ is generic. Thus in this subsection \textbf{we restrict our attention to generic vector configurations}. We postpone the proofs of all results until 
\cref{sec:flips-arbitrary}.

Recall that the vector $\alphaC$ from \cref{dfn:alphaC} is defined up to a positive real constant. We start by fixing a choice for this constant: for each $C\in\Circuits$, define $\alphaC\in\R^n$ by
\begin{equation}\label{eq:alphaC}
\alphaCx_j:=C_j\cdot \VolX(\sC^j)\quad\text{for all $j\in[n]$},   
\end{equation}
where $C_j\in\{+1,0,-1\}$ and $\sC^j\in \Bases$ are given in \cref{def:signedsubset}. As we will see  in \cref{lemma:alpha_dep},  $\alphaC$ satisfies the assumptions of \cref{dfn:alphaC}.

\begin{proposition}\label{prop:flip_generic}
Let $F=(\Tiling\to\Tiling')$ be a flip along $C\in\Circuits$. Then there exists a set $A:=A(F)\subset [n]\setminus \suppC$ such that
\[
  \Tiling\setminus \Tiling'= \left\{\tile_{A\cup j,\sC^j}\right\}_{j\in C^+}\sqcup \left\{\tile_{A,\sC^j}\right\}_{j\in C^-} \quad\text{and}\quad 
 \Tiling'\setminus \Tiling= \left\{\tile_{A,\sC^j}\right\}_{j\in C^+}\sqcup \left\{\tile_{A\cup j,\sC^j}\right\}_{j\in C^-}.\]
\end{proposition}

\begin{definition}\label{def:level}
Using the notation of 
\cref{prop:flip_generic}.
	we define $\level(F):=|A(F)|+1\in[n-d]$.
\end{definition}

\begin{corollary}\label{cor:flip1}
Let $k\in [n-d]$ and $F=(\Tiling\to\Tiling')$ be a flip along $C\in\Circuits$. Then
\[\vertHSP_k(\Tiling)-\vertHSP_k(\Tiling')=
  \begin{cases}
    \alphaC, &\text{if $\level(F)=k$,}\\
    0,&\text{otherwise.}
  \end{cases} \]
\end{corollary}

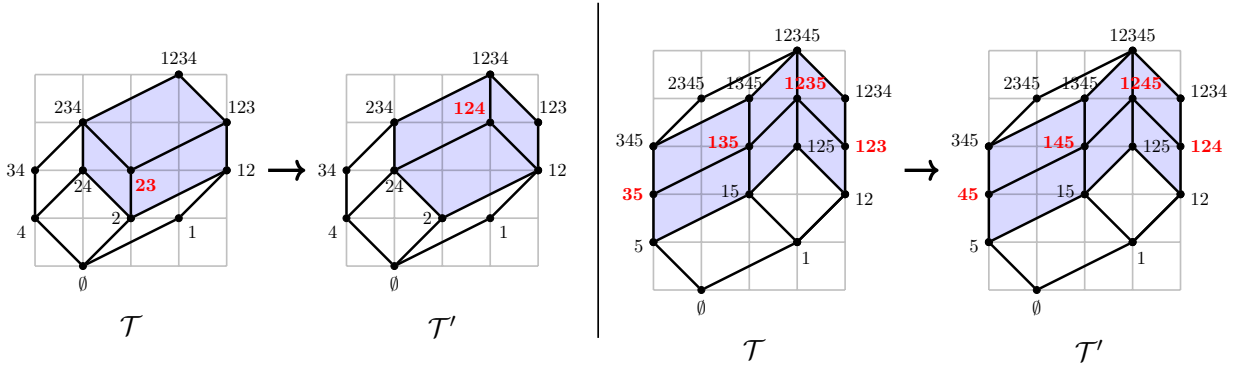
\begin{figure}
\scalebox{0.637}{

\begin{tikzpicture}

\def\gridlw{1pt}
\def\gridop{0.25}
\def\lw{1.5pt}
\def\setscl{1}
\def\vscl{1.2}
\def\VCscl{1.5}
\def\tnode#1#2#3{\node[draw,fill,circle,scale=0.4](N#1) at (#2,#3) {};}
\def\tcoord#1#2#3{\coordinate(N#1) at (#2,#3);}
\def\lnode#1#2{\node[anchor=180+#2,scale=\setscl] (A#1) at (N#1.#2) {$#1$};}
\def\lnodeRED#1#2{\node[anchor=180+#2,scale=\setscl] (A#1) at (N#1.#2) {\textcolor{red}{$\mathbf{#1}$}};}
\def\lnodv#1#2{\node[anchor=180+#2,scale=\vscl] (A#1) at (N#1.#2) {$\v_{#1}$};}
\def\lnod#1#2#3{\node[anchor=-#2,scale=\setscl] (A#1) at (N#1.#2) {#3};}
\def\lnodRED#1#2#3{\node[anchor=180+#2,scale=\setscl] (A#1) at (N#1.#2) {\textcolor{red}{$\mathbf{#3}$}};}
\def\stp{6.5}
\def\stpp{7}
\def\delt{2*\stp}

\def\arrlw{2pt}
\def\tablw{1pt}

\def\xmin{0}
\def\xmax{4}
\def\ymin{0}
\def\ymax{4}
\def\marg{0.7}

\node(B) at (0,0){
  \begin{tikzpicture}
\draw[white] (-\marg,-\marg) rectangle (\xmax+\marg,\ymax+\marg);

\tnode{2}21
\tnode{12}42
\tnode{24}12
\tnode{123}43
\tnode{234}13
\tnode{1234}34

\draw[draw=white,fill=blue!15] (N2.center)--(N12.center)--(N123.center)--(N1234.center)--(N234.center)--(N24.center)--cycle;

\draw[line width=\gridlw,opacity=\gridop] (\xmin,\ymin) grid (\xmax,\ymax);

\tnode{0}10
\tnode{1}31
\tnode{2}21
\tnode{4}01
\tnode{12}42
\tnode{23}22
\tnode{24}12
\tnode{34}02
\tnode{123}43
\tnode{234}13
\tnode{1234}34

\draw[line width=\lw] (N0)--(N1)--(N12)--(N123)--(N1234)--(N234)--(N34)--(N4)--(N0);
\draw[line width=\lw] (N4)--(N24)--(N2)--(N12);
\draw[line width=\lw] (N234)--(N23)--(N123);
\draw[line width=\lw] (N234)--(N24);
\draw[line width=\lw] (N23)--(N2)--(N0);

\lnod{0}{270}{$\emptyset$}
\lnode{1}{-45}
\lnode{12}{0}
\lnode{123}{45}
\lnode{1234}{90}
\lnode{234}{135}
\lnode{34}{180}
\lnode{4}{225}
\lnode{2}{180}
\lnode{24}{-90}
\lnodeRED{23}{-45}
\end{tikzpicture}

};

\node(C) at (\stp,0){
  \begin{tikzpicture}
\draw[white] (-\marg,-\marg) rectangle (\xmax+\marg,\ymax+\marg);

\tnode{2}21
\tnode{12}42
\tnode{24}12
\tnode{123}43
\tnode{234}13
\tnode{1234}34

\draw[draw=white,fill=blue!15] (N2.center)--(N12.center)--(N123.center)--(N1234.center)--(N234.center)--(N24.center)--cycle;

\draw[line width=\gridlw,opacity=\gridop] (\xmin,\ymin) grid (\xmax,\ymax);

\tnode{0}10
\tnode{1}31
\tnode{2}21
\tnode{4}01
\tnode{12}42
\tnode{124}33
\tnode{24}12
\tnode{34}02
\tnode{123}43
\tnode{234}13
\tnode{1234}34
\draw[line width=\lw] (N0)--(N1)--(N12)--(N123)--(N1234)--(N234)--(N34)--(N4)--(N0);
\draw[line width=\lw] (N4)--(N24)--(N2)--(N12);
\draw[line width=\lw] (N234)--(N24);
\draw[line width=\lw] (N2)--(N0);
\draw[line width=\lw] (N12)--(N124)--(N24);
\draw[line width=\lw] (N124)--(N1234);

\lnod{0}{270}{$\emptyset$}
\lnode{1}{-45}
\lnode{12}{0}
\lnode{123}{45}
\lnode{1234}{90}
\lnode{234}{135}
\lnode{34}{180}
\lnode{4}{225}
\lnode{2}{180}
\lnode{24}{-90}
\lnodeRED{124}{145}
\end{tikzpicture}

};

\def\ymax{5}

\node(D) at (\delt,0){
  \begin{tikzpicture}
\draw[white] (-\marg,-\marg) rectangle (\xmax+\marg,\ymax+\marg);

\tnode{1}31
\tnode{5}01
\tnode{12}42
\tnode{15}22
\tnode{125}33
\tnode{345}03
\tnode{1234}44
\tnode{1345}24
\tnode{125}33
\tnode{12345}35

\draw[draw=white,fill=blue!15] (N5.center)--(N15.center)--(N125.center)--(N12.center)--(N1234.center)--(N12345.center)--(N1345.center)--(N345.center)--cycle;

\draw[line width=\gridlw,opacity=\gridop] (\xmin,\ymin) grid (\xmax,\ymax);

\tnode{125}33
\tnode{1235}34
\tnode{0}10
\tnode{1}31
\tnode{5}01
\tnode{12}42
\tnode{15}22
\tnode{35}02
\tnode{123}43
\tnode{135}23
\tnode{345}03
\tnode{1234}44
\tnode{1345}24
\tnode{2345}14
\tnode{12345}35

\draw[line width=\lw](N0)--(N1)--(N12)--(N123)--(N1234)--(N12345)--(N2345)--(N345)--(N35)--(N5)--(N0);
\draw[line width=\lw] (N5)--(N15)--(N125)--(N12)--(N1)--(N15);
\draw[line width=\lw] (N35)--(N135)--(N1235)--(N123);
\draw[line width=\lw] (N345)--(N1345);
\draw[line width=\lw] (N12345)--(N1345)--(N15);
\draw[line width=\lw] (N12345)--(N125);

\lnod{0}{270}{$\emptyset$}
\lnode{1}{-60}
\lnode{12}{-15}
\lnodeRED{123}{0}
\lnode{1234}{15}
\lnode{12345}{90}
\lnode{2345}{135}
\lnode{345}{165}
\lnodeRED{35}{180}
\lnode{5}{195}
\lnode{1345}{107}
\lnodeRED{135}{170}
\lnode{15}{170}
\lnode{125}{0}
\lnodeRED{1235}{62}

\end{tikzpicture}

};

\node(E) at (\delt+\stpp,0){
  \begin{tikzpicture}
\draw[white] (-\marg,-\marg) rectangle (\xmax+\marg,\ymax+\marg);

\tnode{1}31
\tnode{5}01
\tnode{12}42
\tnode{15}22
\tnode{125}33
\tnode{345}03
\tnode{1234}44
\tnode{1345}24
\tnode{125}33
\tnode{12345}35

\draw[draw=white,fill=blue!15] (N5.center)--(N15.center)--(N125.center)--(N12.center)--(N1234.center)--(N12345.center)--(N1345.center)--(N345.center)--cycle;

\draw[line width=\gridlw,opacity=\gridop] (\xmin,\ymin) grid (\xmax,\ymax);

\tnode{125}33
\tnode{1235}34
\tnode{0}10
\tnode{1}31
\tnode{5}01
\tnode{12}42
\tnode{15}22
\tnode{35}02
\tnode{123}43
\tnode{135}23
\tnode{345}03
\tnode{1234}44
\tnode{1345}24
\tnode{2345}14
\tnode{12345}35

\draw[line width=\lw](N0)--(N1)--(N12)--(N123)--(N1234)--(N12345)--(N2345)--(N345)--(N35)--(N5)--(N0);
\draw[line width=\lw] (N5)--(N15)--(N125)--(N12)--(N1)--(N15);
\draw[line width=\lw] (N35)--(N135)--(N1235)--(N123);
\draw[line width=\lw] (N345)--(N1345);
\draw[line width=\lw] (N12345)--(N1345)--(N15);
\draw[line width=\lw] (N12345)--(N125);

\lnod{0}{270}{$\emptyset$}
\lnode{1}{-60}
\lnode{12}{-15}
\lnodRED{123}{0}{124}
\lnode{1234}{15}
\lnode{12345}{90}
\lnode{2345}{135}
\lnode{345}{165}
\lnodRED{35}{180}{45}
\lnode{5}{195}
\lnode{1345}{107}
\lnodRED{135}{170}{145}
\lnode{15}{170}
\lnode{125}{0}
\lnodRED{1235}{62}{1245}

\end{tikzpicture}

};

\node[anchor=north,scale=\VCscl] (AA) at (B.south) {$\Tiling$};
\node[anchor=north,scale=\VCscl] (AA) at (C.south) {$\Tiling'$};
\node[anchor=north,scale=\VCscl] (AA) at (D.south) {$\Tiling$};
\node[anchor=north,scale=\VCscl] (AA) at (E.south) {$\Tiling'$};

\draw[line width=\arrlw,->] (B)--(C);
\draw[line width=\arrlw,->] (D)--(E);
\draw[line width=\tablw] (1.5*\stp,-3.5)--(1.5*\stp,3.5);

\end{tikzpicture}

}
  \caption{\label{fig:flips} A flip for the case when $\VC$ is  generic (left) and non-generic (right).}
\end{figure}

\begin{example}\label{ex:flip}
Let $\VC$ and $\Tiling$ be as in  \cref{ex:gamma}. An example of a flip $F=(\Tiling\to \Tiling')$ is shown in \cref{fig:flips} (left). Here we have $C=(\{3\},\{1,4\})$ and thus $\alphaC=-\e_1+3\e_3-2\e_4=\vect{-1}03{-2}$. We also have $A(F)=\{2\}$ and $\level(F)=2$. Recall from \cref{ex:gamma} that we had $\vertHSP_1(\Tiling)=\vect0301$ and $\vertHSP_2(\Tiling)=\vect0330$. Similarly, we find $\vertHSP_1(\Tiling')=\vect0301$ and $\vertHSP_2(\Tiling')=\vect1302$. Thus $\vertHSP_1(\Tiling)-\vertHSP_1(\Tiling')=0$ and $\vertHSP_2(\Tiling)-\vertHSP_2(\Tiling')=\alphaC$, in agreement with \cref{cor:flip1}.
\end{example}

\subsection{Flips for arbitrary vector configurations}\label{sec:flips-arbitrary}
We generalize the results of the previous subsection to 
vector configurations that are
not necessarily generic.

For a circuit $C\in\Circuits$, denote by
\[\BMC:=\left\{J\subset ([n]\setminus \suppC)\;\middle|\; (J\sqcup \sC^j)\in \Bcal(\Mcal)\text{ for all $j\in\suppC$}\right\} \]
the set of bases of the \emph{contracted oriented matroid} $\Mcal/\suppC$. In other words, $\BMC$ is the set of bases of the vector configuration that is the image of $\VC$ in the quotient space $\R^d/\left< \v_j\mid j \in \suppC\right>$.

For any circuit $C\in\Circuits$ and $J\in\BMC$, define the vector $\alphaCJ\in\R^n$ by
\begin{equation}\label{eq:alphaCJ}
\alphaCJx_j:=C_j\cdot \VolX(\sC^j\sqcup J)\quad\text{for all $j\in[n]$}.
\end{equation}
We also define 
\begin{equation}\label{eq:alpha2}
\alphaC:=\sum_{J\in\BMC} \alphaCJ.
\end{equation}
When $\VC$ is generic, the set $\BMC=\{\emptyset\}$ consists of a single element, and $\alphaXX(C,\emptyset)=\alphaC$ specializes to the vector $\alphaC$ defined in~\eqref{eq:alphaC}.

\begin{lemma}  \label{lemma:alpha_dep}
Let $C\in\Circuits$ be a circuit of $\Mcal$. Then for each $J\in\BMC$, the vector $\alphaCJ$ satisfies the assumptions of \cref{dfn:alphaC}. In particular, the vectors $\{\alphaCJ\mid J\in\BMC\}$ and also
	$\alphaC$ coincide up to rescaling by a positive real number.
\end{lemma}
\begin{proof}
By~\eqref{eq:alphaCJ}, we only need to check that $\alphaCJ$ gives a linear dependence between the vectors of $\VC$, i.e., $\sum_{j\in\suppC}\alphaCJx_j\v_j=0$. Let $I:=\suppC\sqcup J=\{j_1<\dots<j_{d+1}\}$. The kernel of the $d\times (d+1)$ matrix $M$ with columns $\v_{j_1},\dots,\v_{j_{d+1}}$ is given by $\sum_{i\in[d+1]} (-1)^i \Delta_{I\setminus j_i}(M)\cdot \e_i\in\R^{d+1}$, where $\Delta_{I\setminus j_i}(M):=\det(\v_{j_i})_{i\in I\setminus j_i}$ denotes the corresponding Pl\"ucker coordinate of $M$. If $j_i\in J$ then $\Delta_{I\setminus j_i}(M)=0$. If $j_i\in \suppC$ then $|\Delta_{I\setminus j_i}(M)|=|\alphaCJx_{j_i}|$, and the sign agrees with $C_j$.
\end{proof}

We now show the following generalization of \cref{prop:flip_generic}, see \cref{fig:flips} (right) for an example. 
\begin{proposition}\label{prop:flip_arb}
Let $F=(\Tiling\to\Tiling')$ be a flip along $C\in\Circuits$. Then for each $J\in\BMC$, there exists a set $A(F,J)\subset [n]\setminus (\suppC\sqcup J)$ such that
\begin{align*}
	\Tiling\setminus \Tiling'&= \bigsqcup_{J\in\BMC} \left(\left\{\tile_{A(F,J)\cup j,\sC^j\sqcup J}\right\}_{j\in C^+}\sqcup \left\{\tile_{A(F,J),\sC^j\sqcup J}\right\}_{j\in C^-}\right), \text{ and }\\
 \Tiling'\setminus \Tiling&= \bigsqcup_{J\in\BMC} \left(\left\{\tile_{A(F,J),\sC^j\sqcup J}\right\}_{j\in C^+}\sqcup \left\{\tile_{A(F,J)\cup j,\sC^j\sqcup J}\right\}_{j\in C^-}\right).
\end{align*}
\end{proposition}

\def\ESA{\operatorname{Ext}}
\def\ExtB{\ESA_{\sigma}(B)}
\def\ExtXX_#1(#2){\ESA_{#1}(#2)}
Before proving \cref{prop:flip_arb}, we explain how to reconstruct a fine zonotopal tiling $\Tiling$ from the associated generic circuit orientation $\sigma_\Tiling$ defined in 
\eqref{eq:sigma_tiling}.
Consider a generic circuit orientation $\sigma:\Circuits\to \{+1,-1\}$ and a basis $B\in\Bases$ of $\VC$. Given $j\in[n]\setminus B$, there exists a unique circuit $C\in\Circuits$ such that $j\in C^+$ and $\suppC\subset B\sqcup\{j\}$. Following~\cite{LiPo}, we say that $j$ is \emph{externally semi-active} (with respect to $\sigma$ and $B$) if $\sigma(C)=+1$, and we denote by $\ExtB\subset([n]\setminus B)$ the set of all externally semi-active $j$. Define a collection $\Tiling_\sigma$ of tiles by
\begin{equation}\label{eq:ESA}
\Tiling_\sigma:=\left\{\tile_{A,B} \middle| B\in\Bases,\ A=\ExtB\right\}.
\end{equation}
\begin{lemma}\label{lemma:ESA}
Let $\Tiling$ be a fine zonotopal tiling of $\ZV$ and let $\sigma:=\sigma_\Tiling$ be the associated generic circuit orientation. Then $\Tiling=\Tiling_\sigma$.
\end{lemma}
\begin{proof}
Let $B\in\Bases$ be a basis of $\VC$. By \cref{prop:bases_bijection}, there exists a unique $A\subset([n]\setminus B)$ such that $\tile_{A,B}\in \Tiling$. It suffices to show that $A=\ExtB$. Let $j\in ([n]\setminus B)$ be any element, and let $C\in\Circuits$ be the unique circuit such that $\suppC\subset B\cup j$ and $j\in C^+$. We would like to show that $j\in A$ if and only if $\sigma(C)=+1$.

Suppose that $j\in A$. Then $C^+\eltminus j$ is an independent set contained in $B$ and thus $A\cup C^+=A\sqcup (C^+\eltminus j)$ belongs to $\Vert(\Tiling)$, see~\eqref{eq:dfn_Vert}. We also see that $(A\cup C^+)\cap C^-=\emptyset$, so  $A\cup C^+$ orients $C$ positively, and thus $\sigma(C)=+1$.

Conversely, suppose that $j\notin A$. Then $C^-\subset \sC^j$ is an independent set contained in $B$ and thus $A\cup C^-\in \Vert(\Tiling)$. But now $A\cup C^-$ orients $C$ negatively, and thus $\sigma(C)=-1$.
\end{proof}

\begin{corollary}\label{cor:flip_tiles}
Let $F=(\Tiling\to\Tiling')$ be a flip along $C\in\Circuits$, and let $\tile_{A,B}\in\Tiling$. Then:
\begin{itemize}
\item if $B=\sC^j\sqcup J$ for some $j\in C^+$ and $J\in\BMC$ then $j\in A$ and $\tile_{A\eltminus j,B}\in\Tiling'$;
\item if $B=\sC^j\sqcup J$ for some $j\in C^-$ and $J\in\BMC$ then $j\notin A$ and $\tile_{A\cup j,B}\in\Tiling'$;
\item otherwise, $\tile_{A,B}\in\Tiling'$.
\end{itemize}
\end{corollary}
\begin{proof}
By \cref{prop:bases_bijection}, there exists a unique set $A'$ such that $\tile_{A',B}\in\Tiling'$. By \cref{lemma:ESA}, we have $A=\ExtB$ and $A'=\ExtXX_{\sigma'}(B)$, where 
	$\sigma:=\sigma_\Tiling$ and $\sigma':=\sigma_{\Tiling'}$. 
	By \cref{dfn:flip}, the values of 
	$\sigma$ and $\sigma'$ 
	only differ on $\pm C$. By~\eqref{eq:ESA}, for each $j\in ([n]\setminus B)$ such that $\suppC\not\subset (B\cup j)$, we have $j\in A$ if and only $j\in A'$. If $\suppC\subset B\cup j$ then we have $B=\sC^j\sqcup J$ for some $J\in \BMC$, and depending on whether $j\in C^+$ or $j\in C^-$, we either get $j\in A\setminus A'$ or $j\in A'\setminus A$, respectively.
\end{proof}

\begin{proof}[Proof of \cref{prop:flip_arb}.]
Fix $J\in\BMC$ 
	and let $\sigma:=\sigma_\Tiling$.  By \cref{cor:flip_tiles}, in order to prove \cref{prop:flip_arb}, it suffices to show that 
\begin{equation}\label{eq:AFJ}
\text{for any 
	$j \in \suppC$, if we let 
	 $B:=\sC^{j}\sqcup J$, then $\ExtXX_{\sigma}(B) \eltminus j$ is 
		independent of $j$.}
\end{equation}
Indeed, in this case, the set $A(F,J):=\ExtXX_{\sigma}(B) \eltminus j$ clearly satisfies the assumptions of \cref{prop:flip_arb}. 

To prove~\eqref{eq:AFJ}, choose any
	 $j_1,j_2\in \suppC$, and 
	let $B_1:=\sC^{j_1}\sqcup J$, $B_2:=\sC^{j_2}\sqcup J$, $A_1:=\ExtXX_{\sigma}(B_1) \eltminus j_1$, $A_2:=\ExtXX_{\sigma}(B_2)\eltminus j_2$. We need to show that $A_1=A_2$.

Let $\WS:=\Vert(\Tiling)\cup \Vert(\Tiling')$. By \cref{prop:coloc_pm} and \cref{dfn:flip}, for any $X\in\Circuits$ such that $X\neq \pm C$,  $\WS$ orients $X$ either positively or negatively (but not both). Next, we have 
\begin{equation}\label{eq:flip_I}
A_1\sqcup I,A_2\sqcup I\in \WS\quad \text{for all $I\subset (\suppC\sqcup J)$.}
\end{equation}
Indeed,  by \cref{cor:flip_tiles}, we either have $A_1\sqcup (I\setminus j_1)\in\Vert(\Tiling)$ and $A_1\sqcup (I\cup j_1)\in\Vert(\Tiling')$ or vice versa, and the argument for $A_2$ is completely similar.

We would like to show $A_1\subset A_2$. Otherwise, assume that $i\in A_1\setminus A_2$. Let $X\in \Circuits$ be the unique circuit satisfying $\suppX\subset B_2\cup i$ and $i\in X^+$. Then $X\neq \pm C$ and $X^-\subset B_2$. By~\eqref{eq:flip_I}, we have $A_2\sqcup X^-\in \WS$. Since $i\notin A_2$, we have $A_2\cap X^+=\emptyset$, thus $\WS$ orients $X$ negatively.

Suppose that $j_1\notin X^+$. By~\eqref{eq:flip_I}, $A_1\cup X^+=A_1\sqcup (X^+\eltminus i)$ belongs to $\WS$, thus $\WS$ orients $X$ positively, and we get a contradiction.

Thus $j_1\in X^+$. After possibly switching the direction of the flip $F$ (which amounts to replacing $C$ with $-C$), we may assume that $j_1\in C^-$. Applying the \emph{circuit elimination axiom}~\cite[Definition~3.2.1~(C3)]{RedBook} to $X$, $C$, and $j_1$, we see that there exists $Y\in \Circuits$ satisfying 
\[Y^+\subset (X^+\cup C^+)\setminus\{j_1\},\quad Y^-\subset (X^-\cup C^-)\setminus\{j_1\}.\]
We have $Y\neq \pm C$ and $i\notin Y^-$. By~\eqref{eq:flip_I}, the sets $A_1\cup Y^+=A_1\sqcup (Y^+\eltminus i)$ and $A_2\sqcup Y^-$ both belong to $\WS$. Moreover, $A_1\cup Y^+$ orients $Y$ positively while $A_2\sqcup Y^-$ orients $Y$ negatively. We arrive at a contradiction, which shows $A_1\subset A_2$. By symmetry, we get $A_1\supset A_2$, therefore $A_1=A_2$.
\end{proof}

\begin{definition}\label{def2:level}
Using the notation of  
	\cref{prop:flip_arb},
for $J\in\BMC$, we define $\level(F,J):=|A(F,J)|+1$.
\end{definition}

\begin{corollary}\label{cor:flip2}
Let $k\in [n-d]$ and $F=(\Tiling\to\Tiling')$ be a flip along $C\in\Circuits$. Then
\[\vertHSP_k(\Tiling)-\vertHSP_k(\Tiling')=\summa{J\in\BMC}{\level(F,J)=k} \alphaCJ.\]
\end{corollary}
\begin{proof}
Recall from \cref{lemma:alpha_dep} that $\sum_{j\in\suppC}\alphaCJx_j\v_j=0$. 
Since the last coordinate of each $\v_j$ is equal to $1$, \eqref{eq:alphaCJ} implies that 
\begin{equation}\label{eq:diff_0}
  \sum_{j\in C^+} \VolX(\sC^j\sqcup J)=\sum_{j\in C^-} \VolX(\sC^j\sqcup J).
\end{equation}
Combining~\eqref{eq:HSP} with \cref{prop:flip_arb}, we see that there exists $\u\in\R^n$ such that
\begin{align*}
  \vertHSP_k(\Tiling)&=\u+\summa{J\in\BMC}{\level(F,J)=k}  \sum_{j\in C^+}\VolX(\sC^j\sqcup J) \e_{A(F,J)\cup j}+ \summa{J\in\BMC}{\level(F,J)=k+1}  \sum_{j\in C^-}\VolX(\sC^j\sqcup J) \e_{A(F,J)},\\
  \vertHSP_k(\Tiling')&=\u+\summa{J\in\BMC}{\level(F,J)=k}  \sum_{j\in C^-}\VolX(\sC^j\sqcup J) \e_{A(F,J)\cup j}+ \summa{J\in\BMC}{\level(F,J)=k+1}  \sum_{j\in C^+}\VolX(\sC^j\sqcup J) \e_{A(F,J)}.\\
\end{align*}
  By~\eqref{eq:diff_0}, the difference of the right hand sides equals to 
\[\summa{J\in\BMC}{\level(F,J)=k} \left(\sum_{j\in C^+} \VolX(\sC^j\sqcup J)\e_j-\sum_{j\in C^-} \VolX(\sC^j\sqcup J)\e_j\right)=\summa{J\in\BMC}{\level(F,J)=k}\alphaCJ.\qedhere\]
\end{proof}
\begin{example}
\def\vectt#1#2#3#4#5{(#1,#2,#3,#4,#5)}
Let $n=5$, $d=2$, and let $\VC$ consist of the column vectors of the matrix $\begin{pmatrix}
2 & 1 & 0&0 &-1\\
1 & 1 & 1&1 & 1
\end{pmatrix}$, as shown in 
\cref{fig:flips} (right). 
Thus $\v_3=\v_4$, and let $C=(\{3\},\{4\})$. We have $\BMC=\{\{1\},\{2\},\{5\}\}$.

 An example of a flip $F=(\Tiling\to \Tiling')$ along $C$ is shown in 
	\cref{fig:flips} (right). 
	Geometrically, the tiling has not changed, but some vertex labels have changed, replacing $3$ with $4$. The values of $\alphaCJ$, $A(F,J)$, $\level(F,J)$ for various $J\in\BMC$, as well as the values of $\vertHSP_k(\Tiling)$, $\vertHSP_k(\Tiling')$, $\vertHSP_k(\Tiling)-\vertHSP_k(\Tiling')$ for various $k\in[n-d]$, are given in the following tables.
	\begin{center}
\begin{tabular}{cc}
\begin{tabular}{|c|c|c|c|}\hline
$J$ & $\alphaCJ$ & $A(F,J)$ & $\level(F,J)$\\\hhline{|=|=|=|=|}
$\{1\}$ & $2(\e_3-\e_4)$ & $\{5\}$ & $2$\\\hline
$\{2\}$ & $\e_3-\e_4$ & $\{1,5\}$ &$3$\\\hline
$\{5\}$ & $\e_3-\e_4$ & $\{1,2\}$ &$3$\\\hline
\end{tabular}
&
\begin{tabular}{|c|c|c|c|}\hline
$k$ & $\vertHSP_k(\Tiling)$& $\vertHSP_k(\Tiling')$& $\vertHSP_k(\Tiling)-\vertHSP_k(\Tiling')$\\\hhline{|=|=|=|=|}
$1$ & $\vectt20002$ & $\vectt20002$ & $0$ \\\hline
$2$ & $\vectt21203$ & $\vectt21023$ & $2(\e_3-\e_4)$ \\\hline
$3$ & $\vectt21312$ & $\vectt21132$ & $2(\e_3-\e_4)$ \\\hline
\end{tabular}
\end{tabular}
\end{center}
\noindent This again agrees with \cref{cor:flip2}.
\end{example}

\section{Regular zonotopal tilings and higher secondary polytopes}\label{sec:regular}

In this section we start by introducing \emph{regular fine zonotopal tilings.}
We then define higher secondary polytopes, compute their dimension, and 
prove \cref{thm:sigma}. 

Let $\A$, $\V$, and $Q = \conv \A$ 
be as in \cref{notation}, and   
let $\h=(h_1,\dots,h_n)\in \R^n$ be a  \emph{height vector}. 
Then the upper boundary of the polyhedron $\conv\{(\a_i,h_i-t)\mid i\in[n],\ t\geq0\}\subset \R^d$ projects piecewise-linearly onto $Q$, and projections of its facets give rise to a polyhedral 
subdivision  of $Q$.   Such a subdivision is called \emph{regular,} and in particular, the 
$\Acal$-triangulations that can be obtained this way from a height vector $\h$ are called \emph{regular $\A$-triangulations.} Again, the notion of a regular $\A$-triangulation coincides with the notion of a regular fine $\pi$-induced subdivision from \cref{dfn:pi_regular}.

\def\H{{\mathcal{H_\V}}}
\def\vt{{\tilde{\bm{v}}}}
\def\tZon{{\widetilde\Zon}}
\def\tPi{{\widetilde{\tile}}}
\def\<{\langle}
\def\>{\rangle}
\def\GHV{\R^n\setminus \H}
\subsection{Regular zonotopal tilings}\label{sec:regular_zon_tilings}
Let $\VC$ be a vector configuration in $\R^d$ as above. First, we define the notion of a \emph{generic height vector} $\h\in\R^n$. 
Recall the vector $\alphaC$ from \eqref{eq:alpha2}, 
which by \cref{lemma:alpha_dep} satisfies
 the assumptions of \cref{dfn:alphaC}. 
Let $\<\cdot,\cdot\>$ denote the standard inner product on $\R^n$, and define the \emph{secondary hyperplane arrangement}
\begin{equation}\label{eq:H_sec_arr}
\H:=\{\h\in\R^n\mid \<\h,\alphaC\>=0\ \text{for some $C\in\Circuits$}\}.
\end{equation}
\begin{definition}
	We say that a height vector $\h\in\R^n$ is \emph{generic} (for $\V$) if it does not belong to $\H$. In this case, we write $\h\in\GHV$.
\end{definition}
\noindent For $\h\in\GHV$, let $\sigma_\h:\Circuits\to\{+1,-1\}$ be the generic circuit signature given by 
\begin{equation}\label{eq:circuit}
	\sigma_\h(C):=
  \begin{cases}
    +1, &\text{if $\<\h,\alphaC\>>0$,}\\
    -1, &\text{if $\<\h,\alphaC\><0$,}\\
  \end{cases}\qquad \text{for all $C\in\Circuits$.} 
\end{equation}

Recall from~\eqref{eq:sigma_tiling} that each fine zonotopal tiling $\Tiling$ gives rise to a generic circuit signature ${\sigma_\Tiling:\Circuits\to\{+1,-1\}}$.

\begin{proposition}\label{prop:circuit}
Let $\h=(h_1,\dots,h_n)\in \GHV$ be a generic height vector. Then $\Tiling:=\Tiling_\h$ from \cref{dfn:pi_regular} is the unique fine zonotopal tiling of $\ZV$ satisfying $\sigma_\Tiling=\sigma_\h$.
\end{proposition}
\begin{proof}
The uniqueness part follows from \cref{lemma:ESA}. Consider the $\pi$-induced subdivision $\Tiling:=\Tiling_\h$ from \cref{dfn:pi_regular}. Since $\h$ is generic, it follows that $\Tiling$ is a fine zonotopal tiling of $\ZV$.

It remains to show that $\sigma_\Tiling=\sigma_\h$. Otherwise, suppose that $C\in\Circuits$ is a circuit such that $\sigma_\Tiling(C)=-1$ and $\sigma_\h(C)=+1$. Then there must exist a set $S\in\Vert(\Tiling)$ that orients $C$ negatively, so $C^-\subset S$ and $C^+\cap S=\emptyset$. By \cref{dfn:pi_regular}, having $S\in\Vert(\Tiling)$ implies that $\<\e_S,\h\>\geq \<\x,\h\>$ for all $\x\in\cube_n\cap \pi^{-1}(\pi(\e_S))$. On the other hand, since 
 $\alphaC$ satisfies the assumptions of \cref{dfn:alphaC}, and $S$ orients $C$ negatively,
	it is clear that $\e_S+\epsilon \alphaC$ belongs to $\cube_n\cap \pi^{-1}(\pi(\e_S))$ for all sufficiently small $\epsilon>0$.  But now because $\sigma_\h(C)=+1$ is equivalent to $\<\alphaC,\h\>>0$, we get a contradiction. 
\end{proof}

\begin{definition}\label{dfn:regular_tiling}
A fine zonotopal tiling $\Tiling$ of $\ZV$ is called \emph{regular} if $\Tiling=\Tiling_\h$ for some $\h\in\GHV$.
\end{definition}
\noindent Thus regular fine zonotopal tilings are precisely the regular fine $\pi$-induced subdivisions for the case $\pi:\cube_n\to \ZV$. 

\def\Vt{{\widetilde{\V}}}
\def\ZVt{\Zon_{\Vt}}
\def\vt{\widetilde\v}
\begin{remark}\label{rmk:upper}
  The usual definition of $\Tiling_\h$ makes use of the zonotope $\ZVt$ associated with the vector configuration $\Vt=(\vt_1,\dots,\vt_n)$ in $\R^{d+1}$ given by $\vt_i:=(\v,h_i)$. Namely, $\Tiling_\h$ is obtained by projecting the upper boundary of $\ZVt$ down to $\ZV$ via a map that forgets the last coordinate. (Here the \emph{upper boundary} is defined as the set of all points $\x$ on the boundary of $\ZVt$ such that $\x+\epsilon \e_{d+1}\notin \ZVt$ for all $\epsilon>0$.)  It is straightforward to see that this construction gives rise to the same tiling, see~\cite[Lemma~4.2]{BS}.
\end{remark}

The following result is well known, see e.g. \cite[Corollary~4.2]{BS}. We include a proof since we will use a similar construction later in the proof of \cref{prop:regular}. 

\begin{lemma}\label{lemma:regular_flips}
Any two regular fine zonotopal tilings $\Tiling,\Tiling'$ can be connected by a sequence of flips. 
\end{lemma}

\begin{proof}
In order to construct the desired sequence of flips, we first choose generic $\h,\h'\in\GHV$ such that $\Tiling=\Tiling_\h$, $\Tiling'=\Tiling_{\h'}$, and the line segment $\h(t):=t\h+(1-t)\h'$ connecting them intersects at most one hyperplane in $\H$ at a time. (That is, for each $0\leq t\leq 1$, $\h(t)$ is orthogonal to $\alphaC$ for at most one pair $\pm C$ of opposite circuits.) Then the (finite) sequence $\Tiling_{\h(t)}$, defined for all $0\leq t\leq 1$ such that $\h(t)\in\GHV$, connects $\Tiling$ to $\Tiling'$ by flips.
\end{proof}
\noindent We also note that if $\Tiling=\Tiling_\h$ for some $\h\in\GHV$ then $\Tiling_{-\h}=\TilingOP$
(see \cref{def:op}).

\def\Dep{I}
\def\dep{i}
\def\Edep{E}

\subsection{Higher secondary polytopes}\label{sec:higher}
We use the conventions of \cref{notation}.  Recall from 
\cref{def:HSP}
that for each $k\in[n-d]$, the \emph{higher secondary polytope} $\Sigma_{\A,k}$ is defined  as the convex hull
\[\Sigma_{\A,k}:=\Conv\left\{\vertHSP_k(\Tiling)\;\middle|\;  \text{$\Tiling$ is a fine \emph{regular} zonotopal tiling of $\Zon_\VC$} \right\},\]
where the vector $\vertHSP_k(\Tiling)$ is defined in~\eqref{eq:HSP}. As mentioned in \cref{sec:main-results}, we expect that the word \emph{regular} can be omitted from the above definition.
\begin{conjecture}\label{mainconj}
The higher secondary polytope $\Sigma_{\A,k}$ is equal to
\[\Sigma_{\A,k}=\Conv\left\{\vertHSP_k(\Tiling)\;\middle|\;  \text{$\Tiling$ is a fine  zonotopal tiling of $\Zon_\VC$} \right\}.\]
That is, for each (not necessarily regular) fine zonotopal tiling $\Tiling$, the vector $\vertHSP_k(\Tiling)$ lies in $\Sigma_{\A,k}$.
\end{conjecture}
\noindent See \cref{fig:Gr_3_6} for an illustration.

We start by computing the dimension of $\Sigma_{\A,k}$.
\begin{proposition}\label{prop:HSP_dim}
The dimension of $\Sigma_{\A,k}$ is equal to $n-d$.
\end{proposition}
\begin{proof}
Let $M$ be the $d\times n$ matrix whose columns are $\v_1,\dots,\v_n$. Then the row span $U$ of $M$ is a $d$-dimensional subspace of $\R^n$. Let $W\subset \R^n$ be the $(n-d)$-dimensional subspace spanned by the vectors $\alphaC$ for all $C\in\Circuits$. It is clear that $U$ and $W$ are orthogonal subspaces and $\R^n=U\oplus W$. By \cref{cor:flip2},
\cref{lemma:alpha_dep},
	and \cref{lemma:regular_flips}, we see that all edge directions 
	of $\Sigma_{\A,k}$ belong to $W$. Thus $\dim(\Sigma_{\A,k})\leq n-d$.

	By \cref{cor:flip2}, it remains to show that for each circuit $C\in\Circuits$, there exists a flip $F=(\Tiling\to\Tiling')$ along $C$ and $J\in\BMC$ such that $\level(F,J)=k$, that is, $|A(F,J)|=k-1$. Choose any $J\in\BMC$ and any $(k-1)$-element set $S\subset ([n]\setminus (\suppC\sqcup J))$, and let $T:=[n]\setminus (\suppC\sqcup J\sqcup S)$. Choose any height vector $\h=(h_1,\dots,h_n)\in\R^n$ such that $\<\h,\alphaC\>=0$, $\<\h,\alphaX(X)\>\neq0$ for all $X\neq \pm C$, and for all $s\in S$, $b\in \suppC\sqcup J$, and $t\in T$, we have $h_s>0$, $h_t<0$,  and $|h_s|,|h_t|\gg |h_b|$. Let $\h^+,\h^-\in\GHV$ be generic height vectors given by $\h^+:=\h+\epsilon\cdot \alphaC$, $\h^-:=\h-\epsilon\cdot \alphaC$ for some small $\epsilon>0$, and let $\Tiling:=\Tiling_{\h^+}$, $\Tiling':=\Tiling_{\h^-}$. Then $F:=(\Tiling\to\Tiling')$ is a flip along $C$
	(recall  
	\cref{dfn:flip}, \eqref{eq:circuit}, and \cref{prop:circuit}), 
	 and it is easy to see from~\eqref{eq:ESA} and~\eqref{eq:AFJ} using $\sigma_{\h^+}=\sigma_\Tiling$ that $A(F,J)=S$, thus $\level(F,J)=k$.
\end{proof}

\begin{example}
For the case $d=1$ from \cref{ex:d1}, we have a circuit $C=(\{i\},\{j\})$ for all $1\leq i\neq j\leq n$. We see that for each $k\in[n-d]$, the higher secondary polytope $\Sigma_{\A,k}=\Delta_{k,n}$ contains an edge parallel to $\e_i-\e_j$ for all $i\neq j$, in agreement with the proof of \cref{prop:HSP_dim}.
\end{example}

We now proceed to proving \cref{thm:sigma}. Recall from \cref{dfn:pi_regular} that for a polytope $P\subset \R^n$ and a vector $\h\in\R^n$, $(P)^\h$ is the face of $P$ that maximizes the scalar product with $\h$.
\begin{proposition}\label{prop:regular}
Let $\h\in\GHV$ be a generic height vector, and let $\Tiling_\h$ be the corresponding regular fine zonotopal tiling of $\ZV$. 
Recall the definitions of 
  $\vertFib(\Tiling)$, 
	$\vertFib_k(\Tiling)$, and $\vertGKZ(\Tiling)$
from 	\eqref{eq:vertFibzon},
 \eqref{eq:vertFib_k}, and 
\cref{rem:GKZ}.
\begin{theoremlist}
\item\label{item:regular:GKZ} $(\SigmaGKZ)^\h=\vertGKZ(\Tiling_\h)$.  
\item\label{item:regular:cube}   $(\FibPoly(\cube_n\pito \Zon_\VC))^\h=\vertFib(\Tiling_\h)$.
\item\label{item:regular:delta}   $(\FibPoly(\Delta_{k,n}\pito Q_k))^\h=\vertFib_k(\Tiling_\h)$ for all $k\in[n-1]$.
\item\label{item:regular:sigma} $(\Sigma_{\A,k})^\h=\vertHSP_k(\Tiling_\h)$ for all $k\in[n-d]$.
\end{theoremlist}
\end{proposition}
\begin{proof}
Parts~\itemref{item:regular:GKZ}--\itemref{item:regular:delta} are well known, see~\cite[Proposition~1.2, the proof of Theorem~2.5, Corollary~4.2]{BS}, or 
\cite[the proof of Theorem~9.6]{Ziegler}.  To prove~\itemref{item:regular:sigma}, we need to show that for any regular fine zonotopal tiling $\Tiling':=\Tiling_{\h'}$ of $\ZV$ (where $\h'\in\GHV$), we have $\<\h,\vertHSP_k(\Tiling_\h)\>\geq\<\h,\vertHSP_k(\Tiling')\>$. We proceed as in the proof of \cref{lemma:regular_flips}. After slightly modifying $\h'$ without changing $\Tiling_{\h'}$, we may assume that every point of the ray $\{\h'+t\h\mid t\geq0\}$ is orthogonal to $\alphaC$ for at most one pair $\pm C$ of opposite circuits. The corresponding finite sequence of flips connects $\Tiling'$ to $\Tiling$. Suppose that for some $t>0$ and $C\in\Circuits$, we have $\<\h'+t\h,\alphaC\>=0$. Choose a small positive $\epsilon$ so that the tilings  $\Tiling_-:=\Tiling_{\h'+(t-\epsilon)\h}$ and $\Tiling_+:=\Tiling_{\h'+(t+\epsilon)\h}$ differ by a flip $F=(\Tiling_+\to\Tiling_-)$ along $C$. 
	By \cref{dfn:flip} and \cref{prop:circuit},
	$\<\h,\alphaC\>>0$. By \cref{cor:flip2}, $\vertHSP_k(\Tiling_+)-\vertHSP_k(\Tiling_-)$ is a positive scalar multiple of $\alphaC$, so $\<\h,\vertHSP_k(\Tiling_+)\>>\<\h,\vertHSP_k(\Tiling_-)\>$. Thus the dot product of $\vertHSP_k(\Tiling_{\h'+t\h})$ with $\h$ increases weakly as $t$ grows from $0$ to $\infty$, and when $t$ is sufficiently large, we obviously have $\Tiling_{\h'+t\h}=\Tiling_{\h}$.
\end{proof}

\begin{proof}[Proof of \cref{thm:sigma}.]
All four parts of \cref{thm:sigma} follow from \cref{prop:HSP_vert}, \cref{prop:regular},  and \eqref{eq:BS_tiling_reg}.  
 Explicitly,  the polytopes in question are related as follows:
\begin{equation}\label{eq:proof}
\begin{aligned}
  \SigmaGKZ=& \frac1{(d-1)!} \left(\Sigma_{\A,1}+ \ddelta(0,\VC)\right);\\
  \FibPoly(\cube_n\pito \Zon_\VC)=&\frac1{\Vold(\ZV)} \left( \Sigma_{\A,1} + \cdots + \Sigma_{\A,n-d} +\frac12\sum_{k=0}^{n-d}\ddelta(k,\VC) \right);\\
  \FibPoly(\Delta_{k,n}\pito Q_k)=&\frac1{\Voldd(Q_k)} \Biggl( p_{0,d}\Sigma_{\A,k}+p_{1,d}\Sigma_{\A,k-1}+ \dots + p_{d-1,d}\Sigma_{\A,k-d+1} \Biggr.\\
&+\Biggl.\sum_{r=1}^{d-1}\frac{r}{d} \cdot p_{r-1,d-1}\ddelta(k-r,\VC) \Biggr)\qquad \text{for all $k\in[n-1]$};\\
 \Sigma_{\A,k}=&-\Sigma_{\A,n-d-k+1}+\gamma_{k-1}(\VC)\cdot \one_n-\ddelta(k-1,\VC)\qquad \text{for all $k\in[n-d]$}.
\end{aligned}
\end{equation}
\noindent Here we set $p_{r,d}=\frac{\Eul(d,r)}{d!}$ as before.
\end{proof}

\subsection{Vertices, edges, and deformations}
\label{sec2:vert-edges-deform}
In this section, we prove \cref{prop:deform}. We state it more generally for 
point configurations
that are not necessarily generic. 
For a flip $F=(\Tiling\to\Tiling')$ along a circuit $C$ and $J\in \BMC$, recall the definition of $\level(F,J)\in[n-d]$ 
from \cref{def2:level}.
Let us write $\Level(F):=\{\level(F,J)\mid J\in \BMC\}$.

Extending the definitions of \cref{sec:vert-edges-deform}, we say that two fine zonotopal tilings $\Tiling$ and $\Tiling'$ of $\ZV$ are \emph{$k$-equivalent} if they can be connected by flips $F$ such that $k\notin \Level(F)$. Similarly, we say that two flips $F=(\Tiling_1\to\Tiling_2)$ and $F'=(\Tiling'_1\to\Tiling'_2)$ are \emph{$k$-equivalent} if $\Tiling_1$ is $k$-equivalent to $\Tiling_1'$ and $\Tiling_2$ is $k$-equivalent to $\Tiling_2'$.

\begin{proposition}
Let $\A$ be an arbitrary configuration of $n$ points in $\R^{d-1}$, and let $k\in[n-d]$.
\begin{theoremlist}
\item\label{item:deform_vert2} The vertices of the higher secondary polytope $\Sigma_{\A,k}$ are in bijection with $k$-equivalence classes of regular fine zonotopal tilings of $\ZV$.
\item\label{item:deform_edge2} The edges of $\Sigma_{\A,k}$ correspond to $k$-equivalence classes of flips $F$ such that $k\in\Level(F)$.
\item\label{item:deform_deform2} For any nonnegative real numbers $x_1,\dots,x_{n-d}$, the Minkowski sum
  \[\frac1{\Vold(\ZV)} \left(x_1 \Sigma_{\A,1}+\dots + x_{n-d} \Sigma_{\A,n-d}\right)\]
is a parallel deformation of the fiber zonotope $\FibPoly(\cube_n\pito \ZV)$, where an edge corresponding to a flip $F$ along $C\in\Circuits$ is rescaled by $\sum_{J\in\BMC}x_{\level(F,J)}$.
\end{theoremlist}  
\end{proposition}
\begin{proof}
Parts~\itemref{item:deform_vert2} and~\itemref{item:deform_edge2} follow from part~\itemref{item:deform_deform2}.  As for 
part~\itemref{item:deform_deform2}, the statement about the parallel deformation
	is an immediate consequence of 
\cref{thm:FibPoly}, together with the fact (\cite[Proposition 7.12]{Ziegler}) that the normal fan of a Minkowski sum of two polytopes 
	is the common refinement of 
	the individual normal fans.  
The statement about the edges follows from 
	\cref{prop:regular}.
\end{proof}

\def\face{F}

\def\Bpt{Black-partite\xspace}
\def\Wpt{White-partite\xspace}
\def\bpt{black-partite\xspace}
\def\wpt{white-partite\xspace}

\def\bptop{{\operatorname{bpt}}}
\def\wptop{{\operatorname{wpt}}}

\def\Gbpt{G^\bptop}
\def\Gwpt{G^\wptop}

\section{Higher associahedra and plabic graphs}\label{sec:higherassoc}

In this section, we give background on plabic graphs, and explain the relation
between plabic graphs and \emph{higher associahedra}, which are the higher secondary polytopes in the case that $d=3$ and $\A$ is the set of vertices of a convex $n$-gon in $\R^2$. We then prove 
\cref{thm:regular} and discuss several combinatorial notions arising from our construction.

\subsection{Background on plabic graphs}\label{sec:backgr-plab-graphs}

Recall the definition of a plabic graph $G$ and its bipartite version $\Gbip$ from \cref{sec:high-assoc-plab}. We always assume that plabic graphs have no interior vertices of degree $1$ or $2$. A \emph{strand} in a plabic graph $G$ is a directed path $p$ defined as follows:
\begin{itemize}
\item $p$ starts and ends at a boundary vertex of $G$;
\item at each black interior vertex of $G$, $p$ turns ``maximally right'';\footnote{Here by a \emph{maximally right (resp., left) turn} we mean that if an interior vertex $w$ of $G$ is incident to edges $e_1,\dots,e_m$ in clockwise order and $p$ passes through $e_i$ and then through $w$, it must then pass through $e_{i-1}$ (resp., $e_{i+1}$), where the indices are taken modulo $n$.}
\item at each white interior vertex of $G$, $p$ turns ``maximally left''.
\end{itemize}

\begin{figure}
\def\boundnoderadius{1.5pt}
\scalebox{0.8}{
\begin{tikzpicture}[scale=1.0]
\def\strandcolor{green!60!black}

\node[draw, ellipse, black, fill=white, scale=0.9,inner sep=1.0pt] (node123) at (8.00,7.00) {$1\,\textcolor{\strandcolor}{\bm{2}}\,3$};
\node[draw, ellipse, black, fill=white, scale=0.9,inner sep=1.0pt] (node126) at (9.00,4.00) {$1\,\textcolor{\strandcolor}{\bm{2}}\,6$};
\node[draw, ellipse, black, fill=white, scale=0.9,inner sep=1.0pt] (node136) at (7.00,4.00) {$1\,3\,6$};
\node[draw, ellipse, black, fill=white, scale=0.9,inner sep=1.0pt] (node146) at (6.00,3.00) {$1\,4\,6$};
\node[draw, ellipse, black, fill=white, scale=0.9,inner sep=1.0pt] (node156) at (6.00,1.00) {$1\,5\,6$};
\node[draw, ellipse, black, fill=white, scale=0.9,inner sep=1.0pt] (node234) at (4.00,8.00) {$\textcolor{\strandcolor}{\bm{2}}\,3\,4$};
\node[draw, ellipse, black, fill=white, scale=0.9,inner sep=1.0pt] (node236) at (6.00,6.00) {$\textcolor{\strandcolor}{\bm{2}}\,3\,6$};
\node[draw, ellipse, black, fill=white, scale=0.9,inner sep=1.0pt] (node345) at (1.00,5.00) {$3\,4\,5$};
\node[draw, ellipse, black, fill=white, scale=0.9,inner sep=1.0pt] (node346) at (3.00,5.00) {$3\,4\,6$};
\node[draw, ellipse, black, fill=white, scale=0.9,inner sep=1.0pt] (node456) at (2.00,2.00) {$4\,5\,6$};
\fill [opacity=0.2,black] (node123.center) -- (node126.center) -- (node236.center) -- cycle;
\coordinate (bnode1236n1) at (7.67,5.67);
\coordinate (wnode23n1) at (6.00,7.00);
\coordinate (wnode16n1) at (7.33,3.67);
\fill [opacity=0.2,black] (node126.center) -- (node136.center) -- (node236.center) -- cycle;
\coordinate (bnode1236n2) at (7.33,4.67);
\coordinate (wnode16n2) at (7.00,2.67);
\fill [opacity=0.2,black] (node136.center) -- (node146.center) -- (node346.center) -- cycle;
\coordinate (bnode1346n1) at (5.33,4.00);
\coordinate (wnode36n1) at (5.33,5.00);
\fill [opacity=0.2,black] (node146.center) -- (node156.center) -- (node456.center) -- cycle;
\coordinate (bnode1456n1) at (4.67,2.00);
\coordinate (wnode46n1) at (3.67,3.33);
\fill [opacity=0.2,black] (node234.center) -- (node236.center) -- (node346.center) -- cycle;
\coordinate (bnode2346n1) at (4.33,6.33);
\coordinate (wnode34n1) at (2.67,6.00);
\fill [opacity=0.2,black] (node345.center) -- (node346.center) -- (node456.center) -- cycle;
\coordinate (bnode3456n1) at (2.00,4.00);
\draw[line width=0.40mm,black] (node123) -- (node126);
\draw[line width=0.40mm,black] (node123) -- (node234);
\draw[line width=0.40mm,black] (node123) -- (node236);
\draw[line width=0.40mm,black] (node126) -- (node123);
\draw[line width=0.40mm,black] (node126) -- (node136);

\draw[line width=0.40mm,black] (node126) -- (node156);
\draw[line width=0.40mm,black] (node126) -- (node236);
\draw[line width=0.40mm,black] (node136) -- (node126);
\draw[line width=0.40mm,black] (node136) -- (node146);
\draw[line width=0.40mm,black] (node136) -- (node236);
\draw[line width=0.40mm,black] (node136) -- (node346);

\draw[line width=0.40mm,black] (node146) -- (node136);
\draw[line width=0.40mm,black] (node146) -- (node156);
\draw[line width=0.40mm,black] (node146) -- (node346);
\draw[line width=0.40mm,black] (node146) -- (node456);
\draw[line width=0.40mm,black] (node156) -- (node126);
\draw[line width=0.40mm,black] (node156) -- (node146);
\draw[line width=0.40mm,black] (node156) -- (node456);
\draw[line width=0.40mm,black] (node234) -- (node123);
\draw[line width=0.40mm,black] (node234) -- (node236);
\draw[line width=0.40mm,black] (node234) -- (node345);
\draw[line width=0.40mm,black] (node234) -- (node346);
\draw[line width=0.40mm,black] (node236) -- (node123);
\draw[line width=0.40mm,black] (node236) -- (node126);
\draw[line width=0.40mm,black] (node236) -- (node136);
\draw[line width=0.40mm,black] (node236) -- (node234);
\draw[line width=0.40mm,black] (node236) -- (node346);
\draw[line width=0.40mm,black] (node345) -- (node234);
\draw[line width=0.40mm,black] (node345) -- (node346);
\draw[line width=0.40mm,black] (node345) -- (node456);
\draw[line width=0.40mm,black] (node346) -- (node136);
\draw[line width=0.40mm,black] (node346) -- (node146);
\draw[line width=0.40mm,black] (node346) -- (node234);
\draw[line width=0.40mm,black] (node346) -- (node236);
\draw[line width=0.40mm,black] (node346) -- (node345);
\draw[line width=0.40mm,black] (node346) -- (node456);
\draw[line width=0.40mm,black] (node456) -- (node146);
\draw[line width=0.40mm,black] (node456) -- (node156);
\draw[line width=0.40mm,black] (node456) -- (node345);
\draw[line width=0.40mm,black] (node456) -- (node346);
\draw[dashed, line width=1pt,black!70] (5.00,4.50) circle (4.60);
\coordinate (bound_6_1 2_3) at (9.40,5.84);
\node[scale=1,color=black] (boundNode3) at (9.84,5.97) {$\boundLabel_{3}$};
\coordinate (bound_1_2 3_4) at (6.55,8.83);
\node[scale=1,color=black] (boundNode4) at (6.70,9.27) {$\boundLabel_{4}$};
\coordinate (bound_2_3 4_5) at (1.62,7.63);
\node[scale=1,color=black] (boundNode5) at (1.29,7.94) {$\boundLabel_{5}$};
\coordinate (bound_3_4 5_6) at (0.49,3.59);
\node[scale=1,color=black] (boundNode6) at (0.04,3.50) {$\boundLabel_{6}$};
\coordinate (bound_4_5 6_1) at (3.96,0.02);
\node[scale=1,color=black] (boundNode1) at (3.86,-0.43) {$\boundLabel_{1}$};
\coordinate (bound_5_1 6_2) at (8.76,1.84);
\node[scale=1,color={\strandcolor}] (boundNode2) at (9.13,1.58) {$\bm{\boundLabel_{2}}$};
\coordinate (node12) at (9.40,5.84);
\coordinate (node16) at (7.00,3.00);
\coordinate (node23) at (6.00,7.00);
\coordinate (node34) at (2.67,6.00);
\coordinate (node36) at (5.33,5.00);
\coordinate (node45) at (0.49,3.59);
\coordinate (node46) at (3.67,3.33);
\coordinate (node56) at (3.96,0.02);
\coordinate (node1234) at (6.55,8.83);
\coordinate (node1236) at (7.50,5.25);
\coordinate (node1256) at (8.76,1.84);
\coordinate (node1346) at (5.33,4.00);
\coordinate (node1456) at (4.67,2.00);
\coordinate (node2345) at (1.62,7.63);
\coordinate (node2346) at (4.33,6.33);
\coordinate (node3456) at (2.00,4.00);

\draw[blue, line width=\plabiclw] (bnode1236n1) to[bend right=-30] (bnode1236n2);

\draw[blue, line width=\plabiclw] (node12).. controls (8.50,5.50) .. (bnode1236n1);
\draw[blue, line width=\plabiclw] (node16).. controls (8.00,4.00) .. (bnode1236n2);
\draw[blue, line width=\plabiclw] (node23).. controls (7.00,6.50) .. (bnode1236n1);
\draw[blue, line width=\plabiclw] (node36).. controls (6.50,5.00) .. (bnode1236n2);

\draw[blue, line width=\plabiclw] (node16).. controls (6.50,3.50) .. (node1346);
\draw[blue, line width=\plabiclw] (node16).. controls (6.00,2.00) .. (node1456);
\draw[blue, line width=\plabiclw] (node16).. controls (7.50,2.50) .. (node1256);
\draw[blue, line width=\plabiclw] (node23).. controls (6.00,7.50) .. (node1234);
\draw[blue, line width=\plabiclw] (node23).. controls (5.00,7.00) .. (node2346);
\draw[blue, line width=\plabiclw] (node34).. controls (2.50,6.50) .. (node2345);
\draw[blue, line width=\plabiclw] (node34).. controls (2.00,5.00) .. (node3456);
\draw[blue, line width=\plabiclw] (node34).. controls (3.50,6.50) .. (node2346);
\draw[blue, line width=\plabiclw] (node36).. controls (4.50,5.50) .. (node2346);
\draw[blue, line width=\plabiclw] (node36).. controls (5.00,4.50) .. (node1346);
\draw[blue, line width=\plabiclw] (node45).. controls (1.50,3.50) .. (node3456);
\draw[blue, line width=\plabiclw] (node46).. controls (4.50,4.00) .. (node1346);
\draw[blue, line width=\plabiclw] (node46).. controls (2.50,3.50) .. (node3456);
\draw[blue, line width=\plabiclw] (node46).. controls (4.00,2.50) .. (node1456);
\draw[blue, line width=\plabiclw] (node56).. controls (4.00,1.50) .. (node1456);
\draw[blue, line width=\plabiclw] (node1234).. controls (6.00,7.50) .. (node23);
\draw[blue, line width=\plabiclw] (node1256).. controls (7.50,2.50) .. (node16);
\draw[blue, line width=\plabiclw] (node1346).. controls (6.50,3.50) .. (node16);
\draw[blue, line width=\plabiclw] (node1346).. controls (4.50,4.00) .. (node46);
\draw[blue, line width=\plabiclw] (node1346).. controls (5.00,4.50) .. (node36);
\draw[blue, line width=\plabiclw] (node1456).. controls (6.00,2.00) .. (node16);
\draw[blue, line width=\plabiclw] (node1456).. controls (4.00,1.50) .. (node56);
\draw[blue, line width=\plabiclw] (node1456).. controls (4.00,2.50) .. (node46);
\draw[blue, line width=\plabiclw] (node2345).. controls (2.50,6.50) .. (node34);
\draw[blue, line width=\plabiclw] (node2346).. controls (5.00,7.00) .. (node23);
\draw[blue, line width=\plabiclw] (node2346).. controls (4.50,5.50) .. (node36);
\draw[blue, line width=\plabiclw] (node2346).. controls (3.50,6.50) .. (node34);
\draw[blue, line width=\plabiclw] (node3456).. controls (2.00,5.00) .. (node34);
\draw[blue, line width=\plabiclw] (node3456).. controls (2.50,3.50) .. (node46);
\draw[blue, line width=\plabiclw] (node3456).. controls (1.50,3.50) .. (node45);
\draw[->,color={\strandcolor},line width=1.5pt] (1.15,6.99).. controls (2.50,6.50) .. (3.00,6.50);
\draw[->,color={\strandcolor},line width=1.5pt] (3.00,6.50).. controls (3.50,6.50) .. (4.00,6.00);
\draw[->,color={\strandcolor},line width=1.5pt] (4.00,6.00).. controls (4.50,5.50) .. (5.50,5.25);
\draw[->,color={\strandcolor},line width=1.5pt] (5.50,5.25).. controls (6.50,5.00) .. (7.25,4.50);
\draw[->,color={\strandcolor},line width=1.5pt] (7.25,4.50).. controls (8.00,4.00) .. (7.75,3.25);
\draw[->,color={\strandcolor},line width=1.5pt] (7.75,3.25).. controls (7.50,2.50) .. (8.24,1.26);
\draw[blue, line width=\plabiclw,fill=white] (node16) circle (3.0pt);
\draw[blue, line width=\plabiclw,fill=white] (node23) circle (3.0pt);
\draw[blue, line width=\plabiclw,fill=white] (node34) circle (3.0pt);
\draw[blue, line width=\plabiclw,fill=white] (node36) circle (3.0pt);
\draw[blue, line width=\plabiclw,fill=white] (node46) circle (3.0pt);
\draw[blue, line width=\plabiclw,fill=blue] (node1346) circle (3.0pt);
\draw[blue, line width=\plabiclw,fill=blue] (node1456) circle (3.0pt);
\draw[blue, line width=\plabiclw,fill=blue] (node2346) circle (3.0pt);
\draw[blue, line width=\plabiclw,fill=blue] (node3456) circle (3.0pt);

\draw[blue, line width=\plabiclw,fill=blue] (bnode1236n1) circle (3pt);
\draw[blue, line width=\plabiclw,fill=blue] (bnode1236n2) circle (3pt);

\boundnode(node12)
\boundnode(node45)
\boundnode(node56)
\boundnode(node1234)
\boundnode(node2345)
\boundnode(node1256)

\node[draw, ellipse, black, fill=white, scale=0.9,inner sep=1.0pt] (node123) at (8.00,7.00) {$1\,\textcolor{\strandcolor}{\bm{2}}\,3$};
\node[draw, ellipse, black, fill=white, scale=0.9,inner sep=1.0pt] (node126) at (9.00,4.00) {$1\,\textcolor{\strandcolor}{\bm{2}}\,6$};
\node[draw, ellipse, black, fill=white, scale=0.9,inner sep=1.0pt] (node136) at (7.00,4.00) {$1\,3\,6$};
\node[draw, ellipse, black, fill=white, scale=0.9,inner sep=1.0pt] (node146) at (6.00,3.00) {$1\,4\,6$};
\node[draw, ellipse, black, fill=white, scale=0.9,inner sep=1.0pt] (node156) at (6.00,1.00) {$1\,5\,6$};
\node[draw, ellipse, black, fill=white, scale=0.9,inner sep=1.0pt] (node234) at (4.00,8.00) {$\textcolor{\strandcolor}{\bm{2}}\,3\,4$};
\node[draw, ellipse, black, fill=white, scale=0.9,inner sep=1.0pt] (node236) at (6.00,6.00) {$\textcolor{\strandcolor}{\bm{2}}\,3\,6$};
\node[draw, ellipse, black, fill=white, scale=0.9,inner sep=1.0pt] (node345) at (1.00,5.00) {$3\,4\,5$};
\node[draw, ellipse, black, fill=white, scale=0.9,inner sep=1.0pt] (node346) at (3.00,5.00) {$3\,4\,6$};
\node[draw, ellipse, black, fill=white, scale=0.9,inner sep=1.0pt] (node456) at (2.00,2.00) {$4\,5\,6$};
\end{tikzpicture}}

\def\boundnoderadius{1pt}

\caption{\label{fig:plabic_tiling} A plabic tiling of a 
hexagon $Q_3$, with vertices of $Q_3$
 labeled by the cyclic intervals
of size $3$. The dual graph is a (neither trivalent nor bipartite) $(3,6)$-plabic graph. The strand from $5$ to $2$ is shown in green. A face label contains $2$ if and only if it is to the left of this strand.}
\end{figure}
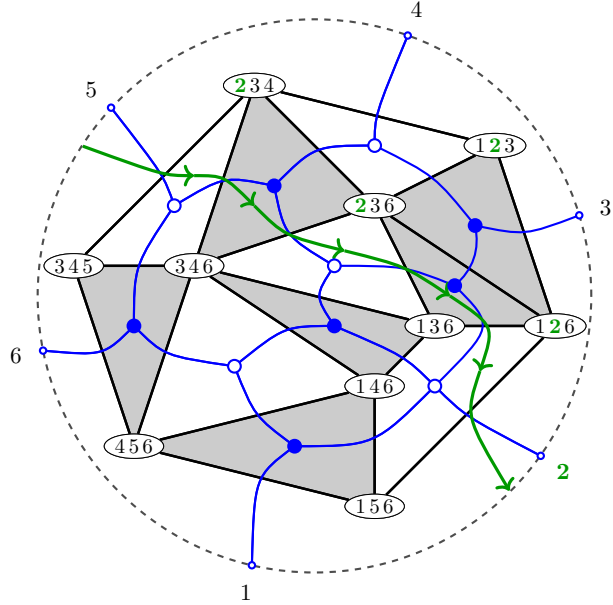

From now on, fix $n$ and $1\leq k\leq n-1$.

\begin{definition}\label{def:kn}
A \emph{$(k,n)$-plabic graph} is a plabic graph $G$ with $n$ boundary vertices such that:
	\begin{enumerate}
\item\label{item:boundary_perm} for each $i\in[n]$, the strand starting at vertex $i$ ends at vertex $i+k$ (modulo $n$);
\item\label{item:number_of_faces} $G$ has $k(n-k)+1$ faces.
\end{enumerate}
\end{definition}

Condition~\eqref{item:number_of_faces} could be replaced by describing several forbidden patterns for the way the strands in $G$ may look, see~\cite[Theorem~13.2]{Pos06}.
Note that $k(n-k)+1$ is the minimal  number of faces a plabic graph satisfying condition~\eqref{item:boundary_perm} can have. 
We label the faces of a plabic graph as follows.

\def\fl#1{S(#1)}
\def\FL#1{\mathcal{F}(#1)}

\begin{definition}\label{def:faces}
Given a $(k,n)$-plabic graph $G$, we label each face $\face$ of $G$ by a set $\fl{\face}\subset[n]$, defined by the condition that for each $i\in[n]$, $\fl{\face}$ contains $i$ if and only if $\face$ is to the left of the unique strand in $G$ that ends at vertex $i$.  
\end{definition}
\noindent It turns out~\cite{Pos06} that $\fl{\face}$ has size $k$. Let $\FL{G}:=\{\fl{\face} \ | \ \face \text{ a face of }G\} \subset {[n] \choose k}.$ 

\subsection{Plabic graphs from fine zonotopal tilings}\label{sec:plabic_graphs_from_tilings}
Throughout the rest of \cref{sec:higherassoc},
we fix $d=3$.  We also
fix a configuration $\A=(\a_1,\dots,\a_n)$ of vertices of a convex $n$-gon in $\R^2$, and 
let $\VC$, 
$\ZV$, $Q_k$, and $\pi$ be as in 
\cref{notation}. 
Recall that we have a projection 
$\Delta_{k,n}\pito Q_k$ from the 
hypersimplex to the $k$-th horizontal section of $\ZV$. 
In this section we recall how to obtain plabic graphs from fine zonotopal tilings, based on 
results of \cite{Gal} and \cite[Section 11]{PosICM}.

Given a subset $S \subset [n]$,
we let 
$$\v_S:= \sum_{i \in S} \v_i.$$
Clearly $Q_k$ is a convex $n$-gon in the affine plane $\Hyp_k = \{(y_1,y_2,y_3) \ | \ y_3=k\}$, with vertices
$\v_{[1,k]}, \v_{[2,k+1]}, \dots, \v_{[n,k-1]}$, corresponding to all consecutive cyclic intervals of size $k$ in $[n]$.
Each two-dimensional face $F$ of $\Delta_{k,n}$ is a triangle with vertices $\e_S,\e_T,\e_R$ for some $S,T,R\in{[n]\choose k}$. Moreover, we have either $|S\cap T\cap R|=k-1$ or $|S\cup T\cup R|=k+1$, in which case we say that $F$ is \emph{isomorphic} to $\Delta_{1,3}$, or $\Delta_{2,3}$, respectively. The fine $\pi$-induced subdivisions of $Q_k$ come from collections of two-dimensional faces of $\Delta_{k,n}$.   Moreover, the fine $\pi$-induced 
subdivisions are in bijection with the tilings of the $n$-gon $Q_k$ by triangles, such that:
\begin{itemize}
	\item  Each vertex has the form $\v_S$ for some $S\in {[n]\choose k}$.
	\item  Each edge has the form $[\v_S, \v_T]$ for two $k$-element subsets $S$ and $T$ such that $|S\cap T|=k-1$.
	\item  Each face is a triangle which is the projection of a two-dimensional face of $\Delta_{k,n}$ isomorphic to either 
		$\Delta_{1,3}$ or $\Delta_{2,3}$ (in which case we say that the face is white, or black, respectively).
\end{itemize}
Such a tiling of $Q_k$ is called a \emph{triangulated plabic tiling}, and its dual graph $G$ (which has white and black vertices
corresponding to the white and black faces of the tiling) is a trivalent plabic graph, 
see \cref{fig:plabic_tiling}.

In the other direction, 
given a $(k,n)$-plabic graph $G$, the corresponding \emph{plabic tiling} $\PT(G)$ 
is a polyhedral subdivision of $Q_k$ into convex polygons colored black and white: 
for each black (resp., white) vertex $w$ of $G$ that is adjacent to faces $\face_1,\dots,\face_m$ in clockwise order, $\PT(G)$ contains a black (resp., white) polygon with boundary vertices $\v_{\fl{\face_1}},\dots,\v_{\fl{\face_m}}$. By the results\footnote{The authors of~\cite{OPS} only work with bipartite $(k,n)$-plabic graphs. For general $(k,n)$-plabic graphs, one needs to ``uncontract'' some interior vertices of $G$ and add some diagonals to the corresponding faces of $\PT(G)$.} of~\cite{OPS}, $\PT(G)$ is the planar dual of $G$: the vertices/edges/faces of $\PT(G)$ correspond to the faces/edges/vertices of $G$, respectively, see \cref{fig:plabic_tiling}.

\def\Scomma{}
\def\ScommaEmpty{\emptyset}

\def\plabicxscl{0.8}
\setlength{\tabcolsep}{3pt}
\def\zoom{0.7}
\begin{figure}
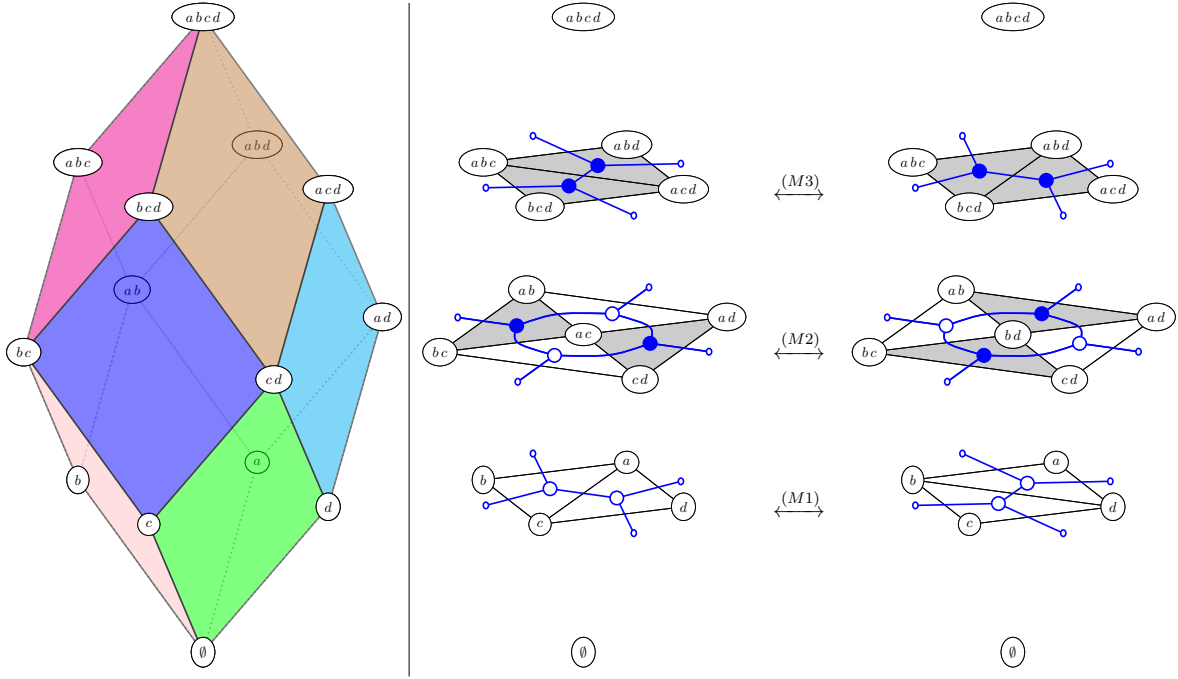

\def\boundnoderadius{2pt}
\makebox[\textwidth]{
\scalebox{0.88}{

}
}

\def\boundnoderadius{1pt}

  \caption{\label{fig:sections} The zonotope $\ZV$ associated to 
  $\V = (\v_a,\v_b,\v_c,\v_d)$ has precisely
  two fine zonotopal tilings, which differ by a flip.
  The horizontal sections give
  rise to triangulated plabic tilings and, dually, to trivalent $(k,n)$-plabic
  graphs for $n=4$ and $k=1,2,3$ (from bottom to top).
  The flip corresponds to applying the moves (M1), (M2), (M3) on plabic
  graphs, as in 
	\cref{thm:flips_moves}.}
\end{figure}

\begin{theorem}[{\cite[Theorem 1.2]{Gal}}]\label{thm:Gal}
\ 
\begin{theoremlist}
\item For each trivalent $(k,n)$-plabic graph $G$, the triangulated plabic tiling $\PT(G)$ coincides with the horizontal section $\Tiling\cap \Hyp_k$ of some fine zonotopal tiling $\Tiling$ of $\ZV$.
\item\label{item:Tiling_to_trivalent} For each fine zonotopal tiling $\Tiling$ of $\ZV$,  the intersection $\Tiling\cap \Hyp_k$ coincides with $\PT(G)$ for a unique trivalent $(k,n)$-plabic graph $G$.
\end{theoremlist}
\end{theorem}
\noindent For a fine zonotopal tiling $\Tiling$ of $\ZV$, we denote by $G_k(\Tiling)$ the trivalent $(k,n)$-plabic graph $G$ from \cref{item:Tiling_to_trivalent}, and we let $\Gbip_k(\Tiling)$ denote its bipartite version.

Recall that $(k,n)$-plabic graphs are connected by moves (M1)--(M3) from \cref{fig:3moves}. For the following result, 
  illustrated in \cref{fig:sections},
see~\cite[Section~3]{Gal}.
\begin{theorem}\label{thm:flips_moves}
  Suppose that $\face=(\Tiling\to\Tiling')$ is a flip and $\level(\face)=k$.
\begin{itemize}
\item We have $G_r(\Tiling)=G_r(\Tiling')$ for all $r\neq k,k+1,k+2$;
\item the graphs $G_k(\Tiling)$ and $G_k(\Tiling')$ are related by move (M1);
\item the graphs $G_{k+1}(\Tiling)$ and $G_{k+1}(\Tiling')$ are related by move (M2);
\item the graphs $G_{k+2}(\Tiling)$ and $G_{k+2}(\Tiling')$ are related by move (M3).
\end{itemize}
\end{theorem}

\def\Area{\operatorname{Area}}
\def\flw#1{S^\cap(#1)}
\def\ar#1{\Area(\dualface{#1})}
\def\dualface#1{#1^\ast}

\subsection{Vertices of higher associahedra}
Each fine zonotopal tiling $\Tiling$ of $\ZV$ gives rise to a point $\vertHSP_k(\Tiling)\in\R^n$ and to a bipartite $(k+1,n)$-plabic graph $\Gbip:=\Gbip_{k+1}(\Tiling)$. The definition~\eqref{eq:HSP} of $\vertHSP_k(\Tiling)$ can be expressed in a simple way in terms of $\PT(\Gbip)$, which we now explain.

Recall that $\PT(\Gbip)$ consists of black and white polygons corresponding to black and white vertices of $\Gbip$ (cf. \cref{fig:plabic_tiling}). Let $w$ be a white interior vertex of $\Gbip$, and let $\face_1,\dots,\face_m$ be the faces of $\Gbip$ adjacent to it. By the construction of face labels in \cref{sec:backgr-plab-graphs}, we see that the face labels $\fl{\face_1},\dots,\fl{\face_m}\in{[n]\choose k+1}$ have intersection $\flw{w}:=\bigcap_{i=1}^m \fl{\face_i}$ of size $k$, see 
\cref{fig:white_labels} (left).
Thus every white face 
$\dualface{w}$ of 
$\PT(\Gbip)$ is naturally labeled by a set $\flw{w}$ of size $k$. Let  $\ar{w}$ denote 
the area of this white face $\dualface{w}$ (viewed as a metric convex polygon inside $\Hyp_{k+1}\cong \R^2$), 
see \cref{fig:white_labels} (right).

\begin{figure}

\begin{tabular}{cc}
\scalebox{0.8}{
\begin{tikzpicture}[scale=1.0]
\def\whiteshape{ellipse}
\node[draw, ellipse, black, fill=white, scale=0.9,inner sep=1.0pt] (node123) at (8.00,7.00) {$1\,2\,3$};
\node[draw, ellipse, black, fill=white, scale=0.9,inner sep=1.0pt] (node126) at (9.00,4.00) {$1\,2\,6$};
\node[draw, ellipse, black, fill=white, scale=0.9,inner sep=1.0pt] (node136) at (7.00,4.00) {$1\,3\,6$};
\node[draw, ellipse, black, fill=white, scale=0.9,inner sep=1.0pt] (node146) at (6.00,3.00) {$1\,4\,6$};
\node[draw, ellipse, black, fill=white, scale=0.9,inner sep=1.0pt] (node156) at (6.00,1.00) {$1\,5\,6$};
\node[draw, ellipse, black, fill=white, scale=0.9,inner sep=1.0pt] (node234) at (4.00,8.00) {$2\,3\,4$};
\node[draw, ellipse, black, fill=white, scale=0.9,inner sep=1.0pt] (node236) at (6.00,6.00) {$2\,3\,6$};
\node[draw, ellipse, black, fill=white, scale=0.9,inner sep=1.0pt] (node345) at (1.00,5.00) {$3\,4\,5$};
\node[draw, ellipse, black, fill=white, scale=0.9,inner sep=1.0pt] (node346) at (3.00,5.00) {$3\,4\,6$};
\node[draw, ellipse, black, fill=white, scale=0.9,inner sep=1.0pt] (node456) at (2.00,2.00) {$4\,5\,6$};
\fill [opacity=0.2,black] (node123.center) -- (node126.center) -- (node236.center) -- cycle;
\coordinate (bnode1236n1) at (7.67,5.67);
\coordinate (wnode23n1) at (6.00,7.00);
\node[draw,\whiteshape,scale=1,fill=white] (Awnode23) at (6.00,7.00) {$2\,3$};
\coordinate (wnode16n1) at (7.33,3.67);
\fill [opacity=0.2,black] (node126.center) -- (node136.center) -- (node236.center) -- cycle;
\coordinate (bnode1236n2) at (7.33,4.67);
\coordinate (wnode16n2) at (7.00,2.67);
\node[draw,\whiteshape,scale=1,fill=white] (Awnode16) at (7.10,2.87) {$1\,6$};
\fill [opacity=0.2,black] (node136.center) -- (node146.center) -- (node346.center) -- cycle;
\coordinate (bnode1346n1) at (5.33,4.00);
\coordinate (wnode36n1) at (5.33,5.00);
\node[draw,\whiteshape,scale=1,fill=white] (Awnode36) at (5.33,5.00) {$3\,6$};
\fill [opacity=0.2,black] (node146.center) -- (node156.center) -- (node456.center) -- cycle;
\coordinate (bnode1456n1) at (4.67,2.00);
\coordinate (wnode46n1) at (3.67,3.33);
\node[draw,\whiteshape,scale=1,fill=white] (Awnode46) at (3.67,3.33) {$4\,6$};
\fill [opacity=0.2,black] (node234.center) -- (node236.center) -- (node346.center) -- cycle;
\coordinate (bnode2346n1) at (4.33,6.33);
\coordinate (wnode34n1) at (2.67,6.00);
\node[draw,\whiteshape,scale=1,fill=white] (Awnode34) at (2.67,6.00) {$3\,4$};
\fill [opacity=0.2,black] (node345.center) -- (node346.center) -- (node456.center) -- cycle;
\coordinate (bnode3456n1) at (2.00,4.00);
\draw[line width=0.40mm,black] (node123) -- (node126);
\draw[line width=0.40mm,black] (node123) -- (node234);
\draw[line width=0.40mm,black] (node123) -- (node236);
\draw[line width=0.40mm,black] (node126) -- (node123);
\draw[line width=0.40mm,black] (node126) -- (node136);
\draw[line width=0.40mm,black] (node126) -- (node156);
\draw[line width=0.40mm,black] (node136) -- (node126);
\draw[line width=0.40mm,black] (node136) -- (node146);
\draw[line width=0.40mm,black] (node136) -- (node236);
\draw[line width=0.40mm,black] (node136) -- (node346);
\draw[line width=0.40mm,black] (node146) -- (node136);
\draw[line width=0.40mm,black] (node146) -- (node156);
\draw[line width=0.40mm,black] (node146) -- (node346);
\draw[line width=0.40mm,black] (node146) -- (node456);
\draw[line width=0.40mm,black] (node156) -- (node126);
\draw[line width=0.40mm,black] (node156) -- (node146);
\draw[line width=0.40mm,black] (node156) -- (node456);
\draw[line width=0.40mm,black] (node234) -- (node123);
\draw[line width=0.40mm,black] (node234) -- (node236);
\draw[line width=0.40mm,black] (node234) -- (node345);
\draw[line width=0.40mm,black] (node234) -- (node346);
\draw[line width=0.40mm,black] (node236) -- (node123);
\draw[line width=0.40mm,black] (node236) -- (node136);
\draw[line width=0.40mm,black] (node236) -- (node234);
\draw[line width=0.40mm,black] (node236) -- (node346);
\draw[line width=0.40mm,black] (node345) -- (node234);
\draw[line width=0.40mm,black] (node345) -- (node346);
\draw[line width=0.40mm,black] (node345) -- (node456);
\draw[line width=0.40mm,black] (node346) -- (node136);
\draw[line width=0.40mm,black] (node346) -- (node146);
\draw[line width=0.40mm,black] (node346) -- (node234);
\draw[line width=0.40mm,black] (node346) -- (node236);
\draw[line width=0.40mm,black] (node346) -- (node345);
\draw[line width=0.40mm,black] (node346) -- (node456);
\draw[line width=0.40mm,black] (node456) -- (node146);
\draw[line width=0.40mm,black] (node456) -- (node156);
\draw[line width=0.40mm,black] (node456) -- (node345);
\draw[line width=0.40mm,black] (node456) -- (node346);
\node[draw, ellipse, black, fill=white, scale=0.9,inner sep=1.0pt] (node123) at (8.00,7.00) {$1\,2\,3$};
\node[draw, ellipse, black, fill=white, scale=0.9,inner sep=1.0pt] (node126) at (9.00,4.00) {$1\,2\,6$};
\node[draw, ellipse, black, fill=white, scale=0.9,inner sep=1.0pt] (node136) at (7.00,4.00) {$1\,3\,6$};
\node[draw, ellipse, black, fill=white, scale=0.9,inner sep=1.0pt] (node146) at (6.00,3.00) {$1\,4\,6$};
\node[draw, ellipse, black, fill=white, scale=0.9,inner sep=1.0pt] (node156) at (6.00,1.00) {$1\,5\,6$};
\node[draw, ellipse, black, fill=white, scale=0.9,inner sep=1.0pt] (node234) at (4.00,8.00) {$2\,3\,4$};
\node[draw, ellipse, black, fill=white, scale=0.9,inner sep=1.0pt] (node236) at (6.00,6.00) {$2\,3\,6$};
\node[draw, ellipse, black, fill=white, scale=0.9,inner sep=1.0pt] (node345) at (1.00,5.00) {$3\,4\,5$};
\node[draw, ellipse, black, fill=white, scale=0.9,inner sep=1.0pt] (node346) at (3.00,5.00) {$3\,4\,6$};
\node[draw, ellipse, black, fill=white, scale=0.9,inner sep=1.0pt] (node456) at (2.00,2.00) {$4\,5\,6$};
\end{tikzpicture}}
&
\scalebox{0.8}{
\begin{tikzpicture}[scale=1.0]
\draw [line width=\gridlw,opacity=\gridop] (1.00,1.00) grid (9.00,8.00); \node[draw, ellipse, black, fill=white, scale=0.9,inner sep=1.0pt] (node123) at (8.00,7.00) {$1\,2\,3$};
\node[draw, ellipse, black, fill=white, scale=0.9,inner sep=1.0pt] (node126) at (9.00,4.00) {$1\,2\,6$};
\node[draw, ellipse, black, fill=white, scale=0.9,inner sep=1.0pt] (node136) at (7.00,4.00) {$1\,3\,6$};
\node[draw, ellipse, black, fill=white, scale=0.9,inner sep=1.0pt] (node146) at (6.00,3.00) {$1\,4\,6$};
\node[draw, ellipse, black, fill=white, scale=0.9,inner sep=1.0pt] (node156) at (6.00,1.00) {$1\,5\,6$};
\node[draw, ellipse, black, fill=white, scale=0.9,inner sep=1.0pt] (node234) at (4.00,8.00) {$2\,3\,4$};
\node[draw, ellipse, black, fill=white, scale=0.9,inner sep=1.0pt] (node236) at (6.00,6.00) {$2\,3\,6$};
\node[draw, ellipse, black, fill=white, scale=0.9,inner sep=1.0pt] (node345) at (1.00,5.00) {$3\,4\,5$};
\node[draw, ellipse, black, fill=white, scale=0.9,inner sep=1.0pt] (node346) at (3.00,5.00) {$3\,4\,6$};
\node[draw, ellipse, black, fill=white, scale=0.9,inner sep=1.0pt] (node456) at (2.00,2.00) {$4\,5\,6$};
\fill [opacity=0.2,black] (node123.center) -- (node126.center) -- (node236.center) -- cycle;
\coordinate (bnode1236n1) at (7.67,5.67);
\coordinate (wnode23n1) at (6.00,7.00);
\node[draw,rectangle,scale=1,fill=white] (AREAnode23) at (6.00,7.00) {$3.0$};
\coordinate (wnode16n1) at (7.33,3.67);
\fill [opacity=0.2,black] (node126.center) -- (node136.center) -- (node236.center) -- cycle;
\coordinate (bnode1236n2) at (7.33,4.67);
\coordinate (wnode16n2) at (7.00,2.67);
\node[draw,rectangle,scale=1,fill=white] (AREAnode16) at (7.10,2.87) {$4.0$};
\fill [opacity=0.2,black] (node136.center) -- (node146.center) -- (node346.center) -- cycle;
\coordinate (bnode1346n1) at (5.33,4.00);
\coordinate (wnode36n1) at (5.33,5.00);
\node[draw,rectangle,scale=1,fill=white] (AREAnode36) at (5.33,5.00) {$3.5$};
\fill [opacity=0.2,black] (node146.center) -- (node156.center) -- (node456.center) -- cycle;
\coordinate (bnode1456n1) at (4.67,2.00);
\coordinate (wnode46n1) at (3.67,3.33);
\node[draw,rectangle,scale=1,fill=white] (AREAnode46) at (3.67,3.33) {$5.5$};
\fill [opacity=0.2,black] (node234.center) -- (node236.center) -- (node346.center) -- cycle;
\coordinate (bnode2346n1) at (4.33,6.33);
\coordinate (wnode34n1) at (2.67,6.00);
\node[draw,rectangle,scale=1,fill=white] (AREAnode34) at (2.67,6.00) {$3.0$};
\fill [opacity=0.2,black] (node345.center) -- (node346.center) -- (node456.center) -- cycle;
\coordinate (bnode3456n1) at (2.00,4.00);
\draw[line width=0.40mm,black] (node123) -- (node126);
\draw[line width=0.40mm,black] (node123) -- (node234);
\draw[line width=0.40mm,black] (node123) -- (node236);
\draw[line width=0.40mm,black] (node126) -- (node123);
\draw[line width=0.40mm,black] (node126) -- (node136);
\draw[line width=0.40mm,black] (node126) -- (node156);
\draw[line width=0.40mm,black] (node136) -- (node126);
\draw[line width=0.40mm,black] (node136) -- (node146);
\draw[line width=0.40mm,black] (node136) -- (node236);
\draw[line width=0.40mm,black] (node136) -- (node346);
\draw[line width=0.40mm,black] (node146) -- (node136);
\draw[line width=0.40mm,black] (node146) -- (node156);
\draw[line width=0.40mm,black] (node146) -- (node346);
\draw[line width=0.40mm,black] (node146) -- (node456);
\draw[line width=0.40mm,black] (node156) -- (node126);
\draw[line width=0.40mm,black] (node156) -- (node146);
\draw[line width=0.40mm,black] (node156) -- (node456);
\draw[line width=0.40mm,black] (node234) -- (node123);
\draw[line width=0.40mm,black] (node234) -- (node236);
\draw[line width=0.40mm,black] (node234) -- (node345);
\draw[line width=0.40mm,black] (node234) -- (node346);
\draw[line width=0.40mm,black] (node236) -- (node123);
\draw[line width=0.40mm,black] (node236) -- (node136);
\draw[line width=0.40mm,black] (node236) -- (node234);
\draw[line width=0.40mm,black] (node236) -- (node346);
\draw[line width=0.40mm,black] (node345) -- (node234);
\draw[line width=0.40mm,black] (node345) -- (node346);
\draw[line width=0.40mm,black] (node345) -- (node456);
\draw[line width=0.40mm,black] (node346) -- (node136);
\draw[line width=0.40mm,black] (node346) -- (node146);
\draw[line width=0.40mm,black] (node346) -- (node234);
\draw[line width=0.40mm,black] (node346) -- (node236);
\draw[line width=0.40mm,black] (node346) -- (node345);
\draw[line width=0.40mm,black] (node346) -- (node456);
\draw[line width=0.40mm,black] (node456) -- (node146);
\draw[line width=0.40mm,black] (node456) -- (node156);
\draw[line width=0.40mm,black] (node456) -- (node345);
\draw[line width=0.40mm,black] (node456) -- (node346);
\node[draw, ellipse, black, fill=white, scale=0.9,inner sep=1.0pt] (node123) at (8.00,7.00) {$1\,2\,3$};
\node[draw, ellipse, black, fill=white, scale=0.9,inner sep=1.0pt] (node126) at (9.00,4.00) {$1\,2\,6$};
\node[draw, ellipse, black, fill=white, scale=0.9,inner sep=1.0pt] (node136) at (7.00,4.00) {$1\,3\,6$};
\node[draw, ellipse, black, fill=white, scale=0.9,inner sep=1.0pt] (node146) at (6.00,3.00) {$1\,4\,6$};
\node[draw, ellipse, black, fill=white, scale=0.9,inner sep=1.0pt] (node156) at (6.00,1.00) {$1\,5\,6$};
\node[draw, ellipse, black, fill=white, scale=0.9,inner sep=1.0pt] (node234) at (4.00,8.00) {$2\,3\,4$};
\node[draw, ellipse, black, fill=white, scale=0.9,inner sep=1.0pt] (node236) at (6.00,6.00) {$2\,3\,6$};
\node[draw, ellipse, black, fill=white, scale=0.9,inner sep=1.0pt] (node345) at (1.00,5.00) {$3\,4\,5$};
\node[draw, ellipse, black, fill=white, scale=0.9,inner sep=1.0pt] (node346) at (3.00,5.00) {$3\,4\,6$};
\node[draw, ellipse, black, fill=white, scale=0.9,inner sep=1.0pt] (node456) at (2.00,2.00) {$4\,5\,6$};
\end{tikzpicture}}
\\

\end{tabular}

\caption{\label{fig:white_labels} A plabic tiling associated to a 
bipartite $(k+1,n)$-plabic graph, with $k+1=3$ and $n=6$.  The labeling
of white faces by $k$-element sets is shown at the left, while the areas of the white faces
are shown at the right.}
\end{figure}
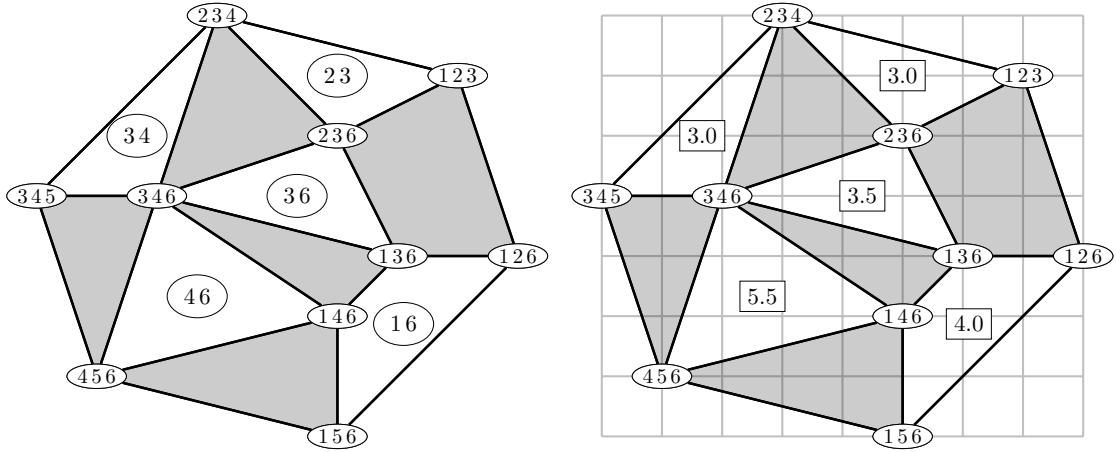

\begin{proposition}\label{prop:area}
Let $\Tiling$ be a fine zonotopal tiling of $\ZV$ and let $\Gbip:=\Gbip_{k+1}(\Tiling)$ be the corresponding bipartite $(k+1,n)$-plabic graph. Then 
\begin{equation}\label{eq:vertHSP_area}
\vertHSP_k(\Tiling)=2\sum_{w} \ar{w}\cdot \e_{\flw{w}},
\end{equation}
where the sum is taken over all \emph{white} interior vertices $w$ of $\Gbip$.
\end{proposition}

\begin{proof}
We use \eqref{eq:HSP}.  
It is not hard to see that each tile $\tile_{A,B} \in \Tiling$ 
gives rise to a white triangle 
$\dualface{w}$  in the plane $y_3 = |A|+1$ whose face label is 
	$\flw{w}=A$. Moreover every 
	white face in the plabic tilings associated to $\Tiling$
	comes from a tile of $\Tiling$.
	Therefore in \eqref{eq:HSP}, instead of summing
	over tiles $\tile_{A,B}$ with $|A|=k$, 
	we can sum over white triangles in the plane $y_3 = k+1$.
	Also note that we can relate the volume of $\tile_{A,B}$ to 
	the area of the corresponding white triangle, using  
	the normalization of our volume 
form given in the discussion 
preceding \cref{rem:notobvious}.
The result follows.
\end{proof}

\begin{example}
Applying \cref{prop:area} to the zonotopal tiling whose horizontal section
is shown in 
	\cref{fig:white_labels},
we obtain 
\begin{equation*}
\def\eij#1#2{\e_{\{#1,#2\}}}
	\vertHSP_2(\Tiling)=8 \eij 16+11\eij 46+7\eij 36+6\eij 34+6\eij 23=(8,6,19,17,0,26).
\end{equation*}
\end{example}

\def\Gh{G_{k,\h}}
\def\Ghp{G_{k,\h'}}

\subsection{Regular plabic graphs}
Recall  from \cref{sec:high-assoc-plab} that \emph{$\A$-regular} trivalent $(k,n)$-plabic graphs are by definition the horizontal sections of regular fine zonotopal tilings of $\ZV$, while $\A$-regular bipartite $(k,n)$-plabic graphs are those that are obtained from $\A$-regular trivalent ones by contracting edges. Let us give an explicit algorithm of reconstructing a trivalent (resp., bipartite) $\A$-regular $(k,n)$-plabic graph $\Gh$ (resp., $\bip \Gh$) from a given height function $\h$. In order to do so, we specialize some general constructions from \cref{sec:flips-zonot-tilings,sec:regular}.

\def\muh{\mu_\h}
If $\VC$ is a configuration of $n$ vectors in $\R^3$ such that their endpoints are vertices $\A$ of a convex $n$-gon in $\Hyp_1\cong \R^2$, then the circuits of $\VC$ are given by 
\[\Circuits=\pm\left\{\left(\{a,c\},\{b,d\}\right)\mid 1\leq a<b<c<d\leq n\right\}.\]
For each circuit $C=(\{a,c\},\{b,d\})$, we have a (unique up to rescaling by a positive real number) vector 
\begin{equation}\label{eq:circuit2}
	\alphaC=x_a\e_a-x_b\e_b+x_c\e_c-x_d\e_d
\end{equation}
	whose coordinates are the coefficients of the linear dependence $x_a \v_a-x_b \v_b+x_c\v_c-x_d\v_d=0$. (Here $x_a,x_b,x_c,x_d>0$.) Given a generic height vector $\h\in\GHV$, we define (as in \eqref{eq:circuit}) the generic circuit
	signature $\sigma_\h(C):=\pm1$ depending on whether $\muh(a,b,c,d):=x_a h_a-x_b h_b+x_ch_c-x_dh_d\in\R\setminus \{0\}$ is positive or negative (it cannot be $0$ precisely because $\h$ is generic).

\def\Compat{\mathcal{F}(\A,k,\h)}
\def\Compatk#1{\mathcal{F}(\A,#1,\h)}
\begin{definition}
We say that $I\subset[n]$ is \emph{$(\A,\h)$-compatible} if for all $1\leq a<b<c<d\leq n$, we have:
\begin{itemize}
\item if $a,c\in I$ and $b,d\notin I$ then $\muh(a,b,c,d)>0$;
\item if $a,c\notin I$ and $b,d\in I$ then $\muh(a,b,c,d)<0$.
\end{itemize}
We denote  $\Compat:=\left\{I\in{[n]\choose k}\;\middle|\; \text{$I$ is $(\A,\h)$-compatible} \right\}$.
\end{definition}

By \cref{prop:circuit}, the regular zonotopal tiling $\Tiling:=\Tiling_\h$ satisfies 
$\sigma_\Tiling=\sigma_\h$.  But now by 
\eqref{eq:sigma_tiling}, we see that the $k$-element sets in  
    $\Vert(\Tiling)$ are precisely the elements of $\Compat$.
Therefore by 
\cref{thm:Gal}, $\Compat$ is the set of labels of some triangulated
plabic tiling, and hence by 
\cite{OPS}, $\Compat$ coincides with 
$\FL{\bip \Gh}$ 
for a unique bipartite $(k,n)$-plabic graph $\bip \Gh$, and this graph $\bip \Gh$ can be explicitly reconstructed from 
$\Compat$ 
as in~\cite[Section~9]{OPS}. To find the unique trivalent $(k,n)$-plabic graph $\Gh$, we use~\cite[Proposition~4.6]{Gal}: the face labels of $\Gh$ are given by $\FL{\Gh}=\Compat$, and two faces labeled by $S,T\in{[n]\choose k}$ are connected by an edge in $\PT(\Gh)$ if and only if $S\cap T\in \Compatk{k-1}$ and $S\cup T\in \Compatk{k+1}$. This completely determines the triangulated plabic tiling $\PT(\Gh)$ from which $\Gh$ can be reconstructed as a planar dual. By 
\cref{thm:Gal},
 $\PT(\Gh)$ is the horizontal section of $\Tiling_\h$ by $\Hyp_k$, and $\PT(\bip \Gh)$ is obtained from it by removing all edges that are adjacent to two faces of the same color.

\begin{example}\label{ex:fivegon}
\begin{figure}
\setlength{\tabcolsep}{9pt}
\scalebox{0.95}{
\begin{tabular}{ccc}
\scalebox{0.7}{
\begin{tikzpicture}[scale=1.7,baseline=(Z.base)]
\coordinate (Z) at (1,1);
\draw [line width=\gridlw,opacity=\gridop] (0,0) grid (2,2);
\newcommand\nd[4]{
\node[draw,circle,fill=black,scale=0.7] (A#3) at (#1,#2) {};
\node[anchor=#4,scale=1.3] (B#3) at (A#3.center) {$\a_{#3}$};
}
\nd 001{45}
\nd 102{90}
\nd 213{180}
\nd 124{-90}
\nd 015{0}
\end{tikzpicture}}
&

\scalebox{0.7}{
\begin{tikzpicture}[scale=1.7,baseline=(Z.base)]
\coordinate (Z) at (1.5,1.5);
\draw [line width=\gridlw,opacity=\gridop] (0.00,0.00) grid (3.00,3.00); \node[draw, ellipse, black, fill=white, scale=1.0,inner sep=0.5pt] (node12) at (1.00,0.00) {$1\,2$};
\node[draw, ellipse, black, fill=white, scale=1.0,inner sep=0.5pt] (node13) at (2.00,1.00) {$1\,3$};
\node[draw, ellipse, black, fill=white, scale=1.0,inner sep=0.5pt] (node15) at (0.00,1.00) {$1\,5$};
\node[draw, ellipse, black, fill=white, scale=1.0,inner sep=0.5pt] (node23) at (3.00,1.00) {$2\,3$};
\node[draw, ellipse, black, fill=white, scale=1.0,inner sep=0.5pt] (node34) at (3.00,3.00) {$3\,4$};
\node[draw, ellipse, black, fill=white, scale=1.0,inner sep=0.5pt] (node35) at (2.00,2.00) {$3\,5$};
\node[draw, ellipse, black, fill=white, scale=1.0,inner sep=0.5pt] (node45) at (1.00,3.00) {$4\,5$};
\coordinate (wnode1n1) at (1.00,0.67);
\fill [opacity=0.2,black] (node12.center) -- (node13.center) -- (node23.center) -- cycle;
\coordinate (bnode123n1) at (2.00,0.67);
\fill [opacity=0.2,black] (node13.center) -- (node15.center) -- (node35.center) -- cycle;
\coordinate (bnode135n1) at (1.33,1.33);
\coordinate (wnode3n1) at (2.67,1.67);
\coordinate (wnode3n2) at (2.33,2.00);
\coordinate (wnode5n1) at (1.00,2.00);
\fill [opacity=0.2,black] (node34.center) -- (node35.center) -- (node45.center) -- cycle;
\coordinate (bnode345n1) at (2.00,2.67);
\draw[line width=0.40mm,black] (node12) -- (node13);
\draw[line width=0.40mm,black] (node12) -- (node15);
\draw[line width=0.40mm,black] (node12) -- (node23);
\draw[line width=0.40mm,black] (node13) -- (node12);
\draw[line width=0.40mm,black] (node13) -- (node15);
\draw[line width=0.40mm,black] (node13) -- (node23);
\draw[line width=0.40mm,black] (node13) -- (node35);
\draw[line width=0.40mm,black] (node15) -- (node12);
\draw[line width=0.40mm,black] (node15) -- (node13);
\draw[line width=0.40mm,black] (node15) -- (node35);
\draw[line width=0.40mm,black] (node15) -- (node45);
\draw[line width=0.40mm,black] (node23) -- (node12);
\draw[line width=0.40mm,black] (node23) -- (node13);
\draw[line width=0.40mm,black] (node23) -- (node34);
\draw[line width=0.40mm,black] (node34) -- (node23);
\draw[line width=0.40mm,black] (node34) -- (node35);
\draw[line width=0.40mm,black] (node34) -- (node45);
\draw[line width=0.40mm,black] (node35) -- (node13);
\draw[line width=0.40mm,black] (node35) -- (node15);
\draw[line width=0.40mm,black] (node35) -- (node34);
\draw[line width=0.40mm,black] (node35) -- (node45);
\draw[line width=0.40mm,black] (node45) -- (node15);
\draw[line width=0.40mm,black] (node45) -- (node34);
\draw[line width=0.40mm,black] (node45) -- (node35);
\draw[dashed, line width=1pt,black!70] (1.60,1.60) circle (2.23);
\coordinate (bound_5_1_2) at (0.03,0.03);
\node[scale=1,color=black] (boundNode2) at (-0.18,-0.18) {$2$};
\coordinate (bound_1_2_3) at (2.69,-0.34);
\node[scale=1,color=black] (boundNode3) at (2.83,-0.59) {$3$};
\coordinate (bound_2_3_4) at (3.72,2.27);
\node[scale=1,color=black] (boundNode4) at (4.00,2.36) {$4$};
\coordinate (bound_3_4_5) at (2.27,3.72);
\node[scale=1,color=black] (boundNode5) at (2.36,4.00) {$5$};
\coordinate (bound_4_5_1) at (-0.34,2.69);
\node[scale=1,color=black] (boundNode1) at (-0.59,2.83) {$1$};
\coordinate (node1) at (1.00,0.67);
\coordinate (node2) at (2.69,-0.34);
\coordinate (node3) at (2.50,1.75);
\coordinate (node4) at (2.27,3.72);
\coordinate (node5) at (1.00,2.00);
\coordinate (node123) at (2.00,0.67);
\coordinate (node125) at (0.03,0.03);
\coordinate (node135) at (1.33,1.33);
\coordinate (node145) at (-0.34,2.69);
\coordinate (node234) at (3.72,2.27);
\coordinate (node345) at (2.00,2.67);
\draw[blue, line width=\plabiclw] (node1).. controls (1.50,0.50) .. (node123);
\draw[blue, line width=\plabiclw] (node1).. controls (1.00,1.00) .. (node135);
\draw[blue, line width=\plabiclw] (node1).. controls (0.50,0.50) .. (node125);
\draw[blue, line width=\plabiclw] (node2).. controls (2.00,0.50) .. (node123);
\draw[blue, line width=\plabiclw] (node3).. controls (2.50,1.00) .. (node123);
\draw[blue, line width=\plabiclw] (node3).. controls (3.00,2.00) .. (node234);
\draw[blue, line width=\plabiclw] (node3).. controls (2.50,2.50) .. (node345);
\draw[blue, line width=\plabiclw] (node3).. controls (2.00,1.50) .. (node135);
\draw[blue, line width=\plabiclw] (node4).. controls (2.00,3.00) .. (node345);
\draw[blue, line width=\plabiclw] (node5).. controls (1.00,1.50) .. (node135);
\draw[blue, line width=\plabiclw] (node5).. controls (1.50,2.50) .. (node345);
\draw[blue, line width=\plabiclw] (node5).. controls (0.50,2.00) .. (node145);
\draw[blue, line width=\plabiclw] (node123).. controls (1.50,0.50) .. (node1);
\draw[blue, line width=\plabiclw] (node123).. controls (2.50,1.00) .. (node3);
\draw[blue, line width=\plabiclw] (node123).. controls (2.00,0.50) .. (node2);
\draw[blue, line width=\plabiclw] (node125).. controls (0.50,0.50) .. (node1);
\draw[blue, line width=\plabiclw] (node135).. controls (1.00,1.00) .. (node1);
\draw[blue, line width=\plabiclw] (node135).. controls (1.00,1.50) .. (node5);
\draw[blue, line width=\plabiclw] (node135).. controls (2.00,1.50) .. (node3);
\draw[blue, line width=\plabiclw] (node145).. controls (0.50,2.00) .. (node5);
\draw[blue, line width=\plabiclw] (node234).. controls (3.00,2.00) .. (node3);
\draw[blue, line width=\plabiclw] (node345).. controls (2.50,2.50) .. (node3);
\draw[blue, line width=\plabiclw] (node345).. controls (1.50,2.50) .. (node5);
\draw[blue, line width=\plabiclw] (node345).. controls (2.00,3.00) .. (node4);
\draw[blue, line width=\plabiclw,fill=white] (node1) circle (2.0pt);
\draw[blue, line width=\plabiclw,fill=white] (node3) circle (2.0pt);
\draw[blue, line width=\plabiclw,fill=white] (node5) circle (2.0pt);
\draw[blue, line width=\plabiclw,fill=blue] (node123) circle (2.0pt);
\draw[blue, line width=\plabiclw,fill=blue] (node135) circle (2.0pt);
\draw[blue, line width=\plabiclw,fill=blue] (node345) circle (2.0pt);

\boundnode(node2)
\boundnode(node4)
\boundnode(node125)
\boundnode(node145)
\boundnode(node234)

\node[draw, ellipse, black, fill=white, scale=1.0,inner sep=0.5pt] (node12) at (1.00,0.00) {$1\,2$};
\node[draw, ellipse, black, fill=white, scale=1.0,inner sep=0.5pt] (node13) at (2.00,1.00) {$1\,3$};
\node[draw, ellipse, black, fill=white, scale=1.0,inner sep=0.5pt] (node15) at (0.00,1.00) {$1\,5$};
\node[draw, ellipse, black, fill=white, scale=1.0,inner sep=0.5pt] (node23) at (3.00,1.00) {$2\,3$};
\node[draw, ellipse, black, fill=white, scale=1.0,inner sep=0.5pt] (node34) at (3.00,3.00) {$3\,4$};
\node[draw, ellipse, black, fill=white, scale=1.0,inner sep=0.5pt] (node35) at (2.00,2.00) {$3\,5$};
\node[draw, ellipse, black, fill=white, scale=1.0,inner sep=0.5pt] (node45) at (1.00,3.00) {$4\,5$};
\end{tikzpicture}}
 &

\scalebox{0.7}{
\begin{tikzpicture}[scale=1.7,baseline=(Z.base)]
\coordinate (Z) at (1.5,1.5);
\draw [line width=\gridlw,opacity=\gridop] (0.00,0.00) grid (3.00,3.00); \node[draw, ellipse, black, fill=white, scale=1.0,inner sep=0.5pt] (node12) at (1.00,0.00) {$1\,2$};
\node[draw, ellipse, black, fill=white, scale=1.0,inner sep=0.5pt] (node13) at (2.00,1.00) {$1\,3$};
\node[draw, ellipse, black, fill=white, scale=1.0,inner sep=0.5pt] (node15) at (0.00,1.00) {$1\,5$};
\node[draw, ellipse, black, fill=white, scale=1.0,inner sep=0.5pt] (node23) at (3.00,1.00) {$2\,3$};
\node[draw, ellipse, black, fill=white, scale=1.0,inner sep=0.5pt] (node34) at (3.00,3.00) {$3\,4$};
\node[draw, ellipse, black, fill=white, scale=1.0,inner sep=0.5pt] (node35) at (2.00,2.00) {$3\,5$};
\node[draw, ellipse, black, fill=white, scale=1.0,inner sep=0.5pt] (node45) at (1.00,3.00) {$4\,5$};
\coordinate (wnode1n1) at (1.00,0.67);
\fill [opacity=0.2,black] (node12.center) -- (node13.center) -- (node23.center) -- cycle;
\coordinate (bnode123n1) at (2.00,0.67);
\fill [opacity=0.2,black] (node13.center) -- (node15.center) -- (node35.center) -- cycle;
\coordinate (bnode135n1) at (1.33,1.33);
\coordinate (wnode3n1) at (2.67,1.67);
\coordinate (wnode3n2) at (2.33,2.00);
\coordinate (wnode5n1) at (1.00,2.00);
\fill [opacity=0.2,black] (node34.center) -- (node35.center) -- (node45.center) -- cycle;
\coordinate (bnode345n1) at (2.00,2.67);
\draw[line width=0.40mm,black] (node12) -- (node13);
\draw[line width=0.40mm,black] (node12) -- (node15);
\draw[line width=0.40mm,black] (node12) -- (node23);
\draw[line width=0.40mm,black] (node13) -- (node12);
\draw[line width=0.40mm,black] (node13) -- (node15);
\draw[line width=0.40mm,black] (node13) -- (node23);
\draw[line width=0.40mm,black] (node13) -- (node34);
\draw[line width=0.40mm,black] (node13) -- (node35);
\draw[line width=0.40mm,black] (node15) -- (node12);
\draw[line width=0.40mm,black] (node15) -- (node13);
\draw[line width=0.40mm,black] (node15) -- (node35);
\draw[line width=0.40mm,black] (node15) -- (node45);
\draw[line width=0.40mm,black] (node23) -- (node12);
\draw[line width=0.40mm,black] (node23) -- (node13);
\draw[line width=0.40mm,black] (node23) -- (node34);
\draw[line width=0.40mm,black] (node34) -- (node13);
\draw[line width=0.40mm,black] (node34) -- (node23);
\draw[line width=0.40mm,black] (node34) -- (node35);
\draw[line width=0.40mm,black] (node34) -- (node45);
\draw[line width=0.40mm,black] (node35) -- (node13);
\draw[line width=0.40mm,black] (node35) -- (node15);
\draw[line width=0.40mm,black] (node35) -- (node34);
\draw[line width=0.40mm,black] (node35) -- (node45);
\draw[line width=0.40mm,black] (node45) -- (node15);
\draw[line width=0.40mm,black] (node45) -- (node34);
\draw[line width=0.40mm,black] (node45) -- (node35);
\draw[dashed, line width=1pt,black!70] (1.60,1.60) circle (2.23);
\coordinate (bound_5_1_2) at (0.03,0.03);
\node[scale=1,color=black] (boundNode2) at (-0.18,-0.18) {$2$};
\coordinate (bound_1_2_3) at (2.69,-0.34);
\node[scale=1,color=black] (boundNode3) at (2.83,-0.59) {$3$};
\coordinate (bound_2_3_4) at (3.72,2.27);
\node[scale=1,color=black] (boundNode4) at (4.00,2.36) {$4$};
\coordinate (bound_3_4_5) at (2.27,3.72);
\node[scale=1,color=black] (boundNode5) at (2.36,4.00) {$5$};
\coordinate (bound_4_5_1) at (-0.34,2.69);
\node[scale=1,color=black] (boundNode1) at (-0.59,2.83) {$1$};
\coordinate (node1) at (1.00,0.67);
\coordinate (node2) at (2.69,-0.34);
\coordinate (node3) at (2.50,1.75);
\coordinate (node4) at (2.27,3.72);
\coordinate (node5) at (1.00,2.00);
\coordinate (node123) at (2.00,0.67);
\coordinate (node125) at (0.03,0.03);
\coordinate (node135) at (1.33,1.33);
\coordinate (node145) at (-0.34,2.69);
\coordinate (node234) at (3.72,2.27);
\coordinate (node345) at (2.00,2.67);
\draw[blue, line width=\plabiclw] (node1).. controls (1.50,0.50) .. (node123);
\draw[blue, line width=\plabiclw] (node1).. controls (1.00,1.00) .. (node135);
\draw[blue, line width=\plabiclw] (node1).. controls (0.50,0.50) .. (node125);
\draw[blue, line width=\plabiclw] (node2).. controls (2.00,0.50) .. (node123);

\draw[blue, line width=\plabiclw] (wnode3n1).. controls (2.50,1.00) .. (node123);
\draw[blue, line width=\plabiclw] (wnode3n1).. controls (3.00,2.00) .. (node234);
\draw[blue, line width=\plabiclw] (wnode3n2).. controls (2.50,2.50) .. (node345);
\draw[blue, line width=\plabiclw] (wnode3n2).. controls (2.00,1.50) .. (node135);
\draw[blue, line width=\plabiclw] (wnode3n1) to[bend right=20] (wnode3n2);

\draw[blue, line width=\plabiclw] (node4).. controls (2.00,3.00) .. (node345);
\draw[blue, line width=\plabiclw] (node5).. controls (1.00,1.50) .. (node135);
\draw[blue, line width=\plabiclw] (node5).. controls (1.50,2.50) .. (node345);
\draw[blue, line width=\plabiclw] (node5).. controls (0.50,2.00) .. (node145);
\draw[blue, line width=\plabiclw] (node123).. controls (1.50,0.50) .. (node1);
\draw[blue, line width=\plabiclw] (node123).. controls (2.00,0.50) .. (node2);
\draw[blue, line width=\plabiclw] (node125).. controls (0.50,0.50) .. (node1);
\draw[blue, line width=\plabiclw] (node135).. controls (1.00,1.00) .. (node1);
\draw[blue, line width=\plabiclw] (node135).. controls (1.00,1.50) .. (node5);
\draw[blue, line width=\plabiclw] (node145).. controls (0.50,2.00) .. (node5);
\draw[blue, line width=\plabiclw] (node345).. controls (1.50,2.50) .. (node5);
\draw[blue, line width=\plabiclw] (node345).. controls (2.00,3.00) .. (node4);
\draw[blue, line width=\plabiclw,fill=white] (node1) circle (2.0pt);
\draw[blue, line width=\plabiclw,fill=white] (wnode3n1) circle (2.0pt);
\draw[blue, line width=\plabiclw,fill=white] (wnode3n2) circle (2.0pt);
\draw[blue, line width=\plabiclw,fill=white] (node5) circle (2.0pt);
\draw[blue, line width=\plabiclw,fill=blue] (node123) circle (2.0pt);
\draw[blue, line width=\plabiclw,fill=blue] (node135) circle (2.0pt);
\draw[blue, line width=\plabiclw,fill=blue] (node345) circle (2.0pt);

\boundnode(node2)
\boundnode(node4)
\boundnode(node125)
\boundnode(node145)
\boundnode(node234)

\node[draw, ellipse, black, fill=white, scale=1.0,inner sep=0.5pt] (node12) at (1.00,0.00) {$1\,2$};
\node[draw, ellipse, black, fill=white, scale=1.0,inner sep=0.5pt] (node13) at (2.00,1.00) {$1\,3$};
\node[draw, ellipse, black, fill=white, scale=1.0,inner sep=0.5pt] (node15) at (0.00,1.00) {$1\,5$};
\node[draw, ellipse, black, fill=white, scale=1.0,inner sep=0.5pt] (node23) at (3.00,1.00) {$2\,3$};
\node[draw, ellipse, black, fill=white, scale=1.0,inner sep=0.5pt] (node34) at (3.00,3.00) {$3\,4$};
\node[draw, ellipse, black, fill=white, scale=1.0,inner sep=0.5pt] (node35) at (2.00,2.00) {$3\,5$};
\node[draw, ellipse, black, fill=white, scale=1.0,inner sep=0.5pt] (node45) at (1.00,3.00) {$4\,5$};
\end{tikzpicture}}
\\ && \\
$\Acal$ & $\bip \Gh$& $\Gh$
\end{tabular}
}

  \caption{\label{fig:fivegon} A point configuration $\A$, and the corresponding bipartite and trivalent plabic graphs for $k=2$ and $\h$ as in \cref{ex:fivegon}.}
\end{figure}
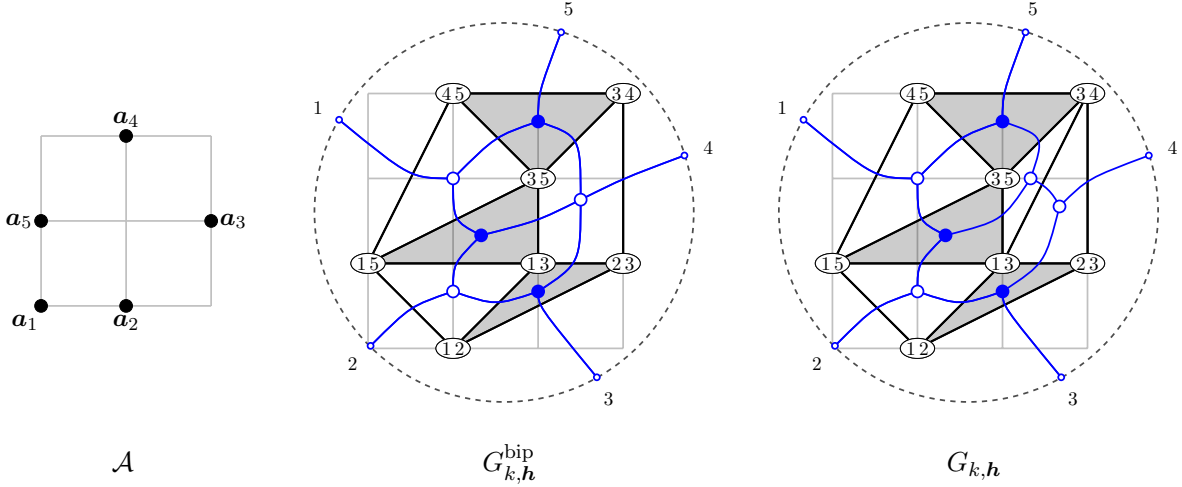
Let $\VC=(\v_1,\dots,\v_5)$ be given by the column vectors of $\begin{pmatrix}
0&1&2&1&0\\
0&0&1&2&1\\
1&1&1&1&1
\end{pmatrix}$, thus $\A$ is the point configuration shown in \cref{fig:fivegon} (left). Let $\h:=(1,0,3,0,0)\in \R^n$. For each circuit $C=(\{a,c\},\{b,d\})$ for $a<b<c<d$, the values of $\alphaC$ (computed using~\eqref{eq:alphaC}) and  $\muh(a,b,c,d)$ are given in the following table (which shows that $\h\in\GHV$ is generic).
\begin{center}
\begin{tabular}{|c|c|}\hline
$\alphaC$ & $\muh(a,b,c,d)$\\\hhline{|=|=|}
$(2,-3,2,-1,0)$ & $+8$\\\hline
$(2,-2,1,0,-1)$ & $+5$\\\hline
$(2,-1,0,1,-2)$ & $+2$\\\hline
$(2,0,-1,2,-3)$ & $-1$\\\hline
$(0,1,-1,1,-1)$ & $-3$\\\hline
\end{tabular}
\end{center}

Let $k=2$. We find that $\Compat=\{12,23,34,45,15,13,35\}$, where we abbreviate $\{a,b\}$ as $ab$. Thus the unique bipartite $(k,n)$-plabic graph $\bip \Gh$ with face labels $\FL{\bip \Gh}=\Compat$ is shown in \cref{fig:fivegon} (middle). To find the trivalent plabic graph $\Gh$, observe that $\{1,3,4\}\in \Compatk{k+1}$ while $\{3\}\in\Compatk{k-1}$, so there must be an edge connecting $\{1,3\}$ to $\{3,4\}$ in $\PT(\Gh)$. Thus $\Gh$ is the trivalent plabic graph given in \cref{fig:fivegon} (right).
\end{example}

\def\GA{\bip G_{k+1,\h}}
\def\GB{\bip G_{k+1,\h'}}

\subsection{Proof of \cref{thm:regular}}
\itemref{thm:regular_bip}: Our goal is to show that given two generic height vectors $\h,\h'\in\GHV$, we have $\vertHSP_k(\Tiling_\h)=\vertHSP_k(\Tiling_{\h'})$ if and only if $\GA=\GB$. By~\eqref{eq:vertHSP_area}, if $\GA=\GB$ then clearly $\vertHSP_k(\Tiling_\h)=\vertHSP_k(\Tiling_{\h'})$.  Conversely, assume that $\vertHSP_k(\Tiling_\h)=\vertHSP_k(\Tiling_{\h'})$. Then by \cref{item:deform_vert}, the tilings $\Tiling_\h$ and $\Tiling_{\h'}$ are $k$-equivalent. By \cref{thm:flips_moves}, we see that the trivalent graphs $G_{k+1,\h}$ and $G_{k+1,\h'}$ are related by moves (M1) and (M3), thus their bipartite versions coincide. Similarly, combining \cref{item:deform_edge} with \cref{thm:flips_moves}, we find that the edges of $\Sigma_{\A,k}$ correspond to square moves of $\A$-regular bipartite $(k+1,n)$-plabic graphs.

\itemref{thm:regular_triv}: First, note that 
	by 
\cref{thm:FibPoly_k}, 
the vertices and edges of $\Sigma_{\A, k}+\Sigma_{\A, k-1} + \Sigma_{\A, k-2}$ are in bijection with vertices and edges of $\frac1{\Voldd(Q_k)} \left(\Sigma_{\A, k}+4\Sigma_{\A, k-1} + \Sigma_{\A, k-2}\right)\shifteq  \FibPoly(\Delta_{k,n}\pito Q_k)$.  The statement that the vertices and edges of $\Sigma_{\A, k}+\Sigma_{\A, k-1} + \Sigma_{\A, k-2}$ correspond to trivalent plabic graphs and moves (M1)--(M3) connecting them follows by combining \cref{item:deform_deform} with \cref{thm:flips_moves}. \qed

\begin{figure}
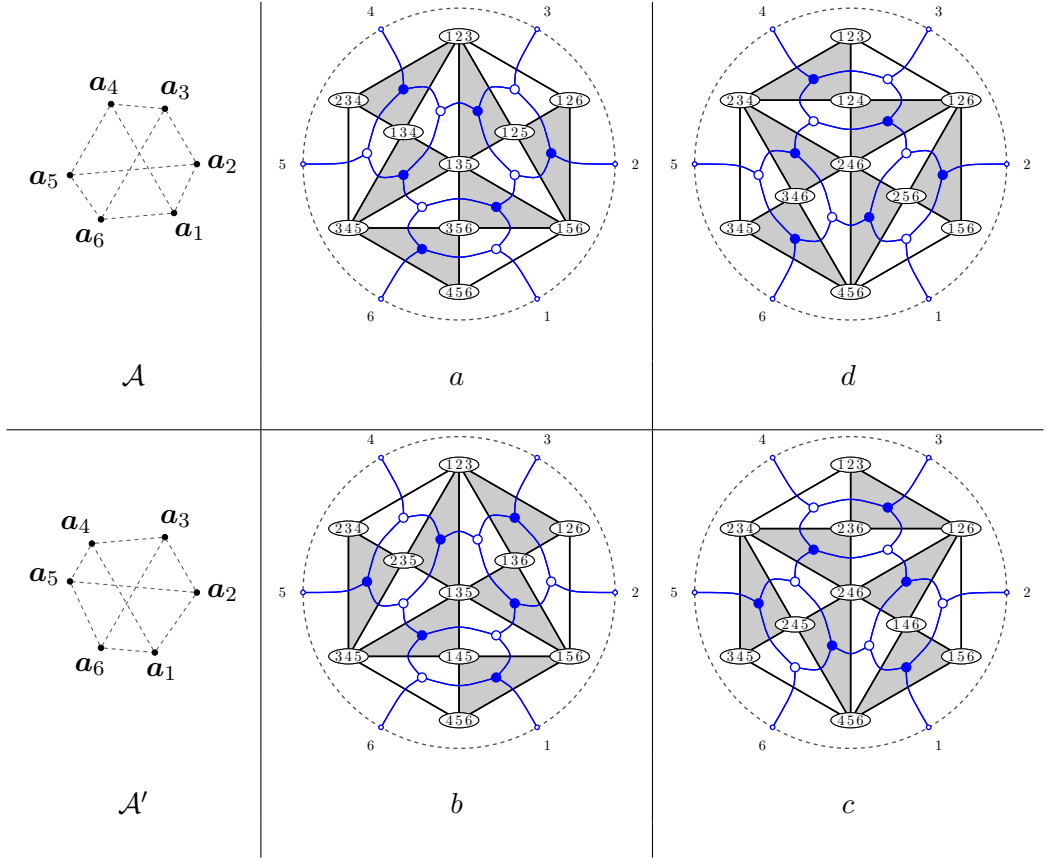




\caption{\label{fig:hexagons} The four plabic graphs, corresponding to the points in \cref{fig:Gr_3_6} labeled $a,b,c,d$. These are the only $(3,6)$-plabic graphs that are not $\A$-regular for some $\A$, see \cref{ex:hexagons}. Similar figures can be found in~\cite[Figure~18]{KK} or~\cite[Figure~1]{OS}.} 
\end{figure}

\begin{example} \label{ex:hexagons}
Let $n=6$ and $k=3$. An example of a higher associahedron $\Sigma_{\A,k}$, where $\A$ is the point configuration from \cref{fig:hexagons} (top left) is shown in \cref{fig:Gr_3_6}. The plabic graphs corresponding to the points of $\Sigma_{\A,k}$ labeled by $a,b,c,d$ are shown in \cref{fig:hexagons}. The points labeled by $b$ and $c$ belong to the interior of $\Sigma_{\A,k}$. On the other hand, for the point configuration $\A'$ from \cref{fig:hexagons} (bottom left), the points labeled $a$ and $d$ belong to the interior of $\Sigma_{\A',k}$, while the points labeled by $b$ and $c$ are among the vertices of $\Sigma_{\A',k}$. If $\A''$ is such that the three diagonals of the hexagon $Q=\conv\A''$ intersect at a single point then none of the four points $a,b,c,d$ are among the $30$ vertices of $\Sigma_{\A'',k}$. A similar computation can be found in~\cite[Theorem~4.2]{KK}.

The plabic graphs labeled by $b$ and $d$ arose in \cite[Section 8]{RW} in 
the context of mirror symmetry for Grassmannians.  If one considers
the Newton-Okounkov bodies $\Delta_G$ associated to a plabic
graph $G$ for $\Gr(3,6)$, then $32$ of the $34$ plabic graphs
give rise to integral polytopes $\Delta_G$;
$b$ and $d$ label the non-integral ones.
\end{example}

When $\A$ is the set of vertices of a convex $n$-gon, the combinatorics of the associahedron $\Sigma_{\A,1}$ does not depend on the specific choice of this $n$-gon. \Cref{ex:hexagons} shows that this is not the case for higher associahedra. Computational evidence suggests that the following result still holds. 

\begin{conjecture}
Suppose that $\A$ is the set of vertices of a \emph{generic} convex $n$-gon. Then the $f$-vector of $\Sigma_{\A,k}$ depends only  on $n$ and $k$.
\end{conjecture}
\noindent For instance, we saw in \cref{ex:hexagons} that $\Sigma_{\A,k}$ has $32$ vertices when $\A$ is generic and $n=6$. The number of vertices of $\Sigma_{\A,k}$ for generic $\A$, $n\leq 7$, and $k\in[n-3]$ is given in the following table.
\begin{center}
\begin{tabular}{c|ccccccc}
n  &   &   &   &  &   &   &  \\\hline
4  &   &   &   & 2 &   &   &  \\
5  &   &   & 5 &   & 5 &   &  \\
6  &   &14 &   & 32&   &14 &  \\
7  &42 &   &231&   &231&   &42\\
\end{tabular}
\end{center}

\def\w{w}
\def\hw{\h^{(\w)}}
\def\Compatw{\mathcal{F}(\A,k,\hw)}
\def\compl#1#2{[#1,#2]^{\operatorname{c}}}
\def\TOP#1#2{\operatorname{top}^{(\w)}_{#1}(#2)}
\def\id{\operatorname{id}}
\subsection{Large heights}
Fix a configuration $\A$ of vertices of a convex $n$-gon in $\R^2$. Let $\w=(\w_1,\dots,\w_n)\in S_n$ be a permutation of $[n]$. Choose a height vector $\hw=(h_1,\dots,h_n)\in\R^n$  satisfying
\begin{equation}\label{eq:hw}
h_{\w_1} \gg h_{\w_2} \gg \dots \gg h_{\w_n}>0.
\end{equation}
In \eqref{eq:hw}, our usage of $\gg$ means that the heights are large  
compared to the coefficients appearing
in \eqref{eq:circuit2} (for all circuits $C$), or more precisely:
for each $4$-tuple
$a<b<c<d$, we have that
$\muh(a,b,c,d) = x_a h_a + x_c h_{c} - x_b h_b - x_d h_{d}>0$
 if and only if $\max(h_a, h_c) > \max(h_b, h_d)$. 
Our goal is to explicitly describe  $\Compatw$. First we need a few definitions.

Fix $n$, and choose $s,t\in[n]$. We let $[s,t)$ be the \emph{cyclic interval} between $s$ and $t-1$:
if $s\leq t$, then $[s,t) := \{s, s+1,\dots, t-1\}$, and if $s>t$, then
$[s,t) := \{s, s+1,\dots, n, 1, 2, \dots, t-1\}$. We similarly define cyclic intervals $(s,t]$ and $[s,t]$.

For $S\subset[n]$ and $0\leq j\leq |S|$, we define $\TOP jS$ to be the $j$-element subset $T$ of $S$ such that $h_t>h_s$ (equivalently, $\w_t<\w_s$) for all $t\in T$ and $s\in S\setminus T$.

\begin{proposition}
\label{prop:WSS}
Let $w\in S_n$ and $\hw$ be as in~\eqref{eq:hw}. Then for each $1 \leq k \leq n$, we have
\begin{equation}
  \Compatw=\bigsqcup_{r=1}^k \left\{[s,t)\sqcup \TOP{k-r}{[t,s)}\;\middle|\; s,t\in[n]\text{ such that }|[s,t)|=r \right\} \sqcup\left\{\TOP{k}{[n]}\right\}.\label{eq:compatw}
\end{equation}

\end{proposition}
\begin{proof}
It is easy to see that each set in the right hand side of~\eqref{eq:compatw} is $(\A,\hw)$-compatible. Conversely, consider $I\in{[n]\choose k}$ and write $I$ as a union of cyclic intervals
        $I_1 \cup \dots \cup I_m$ with $m$ as small as possible.  For example, if
        $I = \{1, 3,4,5, 7,8,10\} \subset [10]$ then we write
        $I = [10,1] \cup [3,5] \cup [7,8]$.
        Clearly, $I$ being $(\A,\hw)$-compatible means that whenever we choose
        $i, i'\in I$ from two distinct cyclic intervals $I_a$ and $I_b$,
        either $h_i$ or $h_{i'}$ is greater than any $h_j$ for
        $j\notin I$.

        Therefore at most one of the cyclic intervals $I_1, \dots, I_m$
        can contain elements whose height is less than the height of
        any element not in $I$. Let $[s,t)$ be that cyclic interval (if it exists, otherwise we must have $I=\TOP{k}{[n]}$), and let $r:=|[s,t)|\leq k$. Since we need all remaining elements of $I$ to have greater
        heights than all elements of $[n]\setminus I$,
        we find $I=[s,t)\sqcup \TOP{k-r}{[t,s)}$.
\end{proof}

\begin{remark}
Note that by 
\cref{prop:WSS}, the set  
	$\Compatw$  explicitly constructed in \cref{prop:WSS}
	depends only on the ordering of the
largest $k$ heights.
\end{remark}

\begin{example}\label{ex:extreme}
Fix $k$ and $n$ and suppose that $\w=w_0:=(n,n-1,\dots,1)$. Then $\Compatw$
	consists of $[n-k+1,n]$  together with 
the $k$-element subsets 
        $[i, i+j) \cup (n-k+j,n]$ for
$1 \leq i \leq n-k$ and $1 \leq j \leq k$. 
Note that if we interpret $k$-element subsets of $[n]$
as Young diagrams contained in a $k \times (n-k)$ rectangle
(by identifying each Young diagram with the path consisting
of unit steps west and south
        from  $(n-k,k)$ to $(0,0)$ which cuts it out and then
reading off the positions of the vertical steps),
then $\Compatw$
	corresponds to the rectangles which fit inside the
$ k \times (n-k)$ rectangle.  This collection was called the
\emph{rectangles cluster} in \cite{RW}
	and comes from the plabic graph associated to the \emph{Le-diagram} of~\cite{Pos06}.

On the other hand, suppose that $\w=\id:=(1,2,\dots,n)$. Then $\Compatw$
 consists of $[k]$ together with the
$k$-element subsets 
$[1,i) \cup [j,j+k-i]$ for
$1 \leq i \leq k$ and $i+1 \leq j \leq n-k+i$.
If we interpret $k$-element subsets of $[n]$
as Young diagrams contained in a $k \times (n-k)$ rectangle
as before, then
$\Compatw$  corresponds to Young diagrams which are
        \emph{complements} of rectangles in the
        $k \times (n-k)$ rectangle.
\end{example}

\subsection{\Bpt and \wpt plabic graphs} 
By \cref{thm:regular}, the vertices of $\Sigma_{\A,k}$ correspond to bipartite plabic graphs, while the vertices of $\Sigma_{\A,k}+\Sigma_{\A,k-1}+\Sigma_{\A,k-2}$ correspond to trivalent plabic graphs. It is thus natural to also consider the polytope $\Sigma_{\A,k}+\Sigma_{\A,k-1}$.

\begin{definition}
A plabic graph $G$ is called \emph{\bpt} if all interior white vertices of $G$ are trivalent, and no edge of $G$ connects two black interior vertices. 
\end{definition}
\noindent We similarly define \emph{\wpt} plabic graphs by switching the roles of black and white in the above definition. For example, for each $n\geq 3$, there is only one \wpt $(1,n)$-plabic graph. As discussed in \cref{ex:Catalan}, there is a Catalan number $C_{n-2}$ of \bpt $(1,n)$-plabic graphs, and the number of \wpt $(2,n)$-plabic graphs is also equal to $C_{n-2}$. As we will show in \cref{prop:bpt_wpt} below, this is not a coincidence.

It follows from~\cite[Theorem~13.4]{Pos06} that any two \bpt $(k,n)$-plabic graphs are related by moves (M1) and (M2), and any two \wpt $(k,n)$-plabic graphs are related by moves (M2) and (M3). We deduce the following surprising bijection from the results of~\cite{Gal}.

\begin{proposition}\label{prop:bpt_wpt}
For $k<n$, \bpt $(k,n)$-plabic graphs are in bijection with \wpt $(k+1,n)$-plabic graphs.
\end{proposition}
\begin{proof}
We describe a construction that gives the desired bijection. Given a plabic graph $G$, denote by $\Gbpt$ (resp., $\Gwpt$) the \bpt (resp., \wpt) plabic graph obtained from $G$ by contracting all edges connecting two black (resp., white) interior vertices. Given a fine zonotopal tiling $\Tiling$ of $\ZV$,  denote by $\Gbpt_k(\Tiling)$ and $\Gwpt_k(\Tiling)$ the \bpt and \wpt $(k,n)$-plabic graphs obtained from the trivalent plabic graph $G_k(\Tiling)$ from \cref{sec:plabic_graphs_from_tilings}. For each $\Tiling$ and each $k<n$, we say that the plabic graphs $\Gbpt_k(\Tiling)$ and $\Gwpt_{k+1}(\Tiling)$ are \emph{linked}.
\begin{lemma}\label{lemma:bpt_wpt_bij}
Every \bpt $(k,n)$-plabic graph is linked with exactly one \wpt $(k+1,n)$-plabic graph, and every \wpt $(k+1,n)$-plabic graph is linked with exactly one \bpt $(k,n)$-plabic graph.
\end{lemma}
\begin{proof}
It follows from the results of~\cite{Gal} that every trivalent $(k,n)$-plabic graph $G$ appears as $G_k(\Tiling)$ for some fine zonotopal tiling $\Tiling$ of $\ZV$. Thus every \bpt $(k,n)$-plabic graph is equal to $\Gbpt_k(\Tiling)$ for some $\Tiling$, and is linked with the graph $\Gwpt_{k+1}(\Tiling)$. Every white vertex of $\Gbpt_k(\Tiling)$ is trivalent, thus the three faces incident to it are labeled by sets $A\cup b_1,A\cup b_2,A\cup b_3$ for some $b_1,b_2,b_3\notin A$, and the horizontal section $\Tiling\cap \Hyp_k$ contains a white triangle with vertex labels $A\cup b_1,A\cup b_2,A\cup b_3$. We find that $\Pi_{A,B}\in\Tiling$ for $B:=\{b_1,b_2,b_3\}$, but then a black triangle with vertex labels $A\cup \{b_1,b_2\}$, $A\cup \{b_1,b_3\}$, $A\cup \{b_2,b_3\}$ appears in $\Tiling\cap \Hyp_{k+1}$. Conversely, every black triangle in $\Tiling\cap \Hyp_{k+1}$ corresponds to a white triangle in $\Tiling\cap \Hyp_k$. We have shown that $\Gwpt_{k+1}(\Tiling)$ is uniquely determined by $\Gbpt_k(\Tiling)$. The proof that  $\Gbpt_{k}(\Tiling)$ is uniquely determined by $\Gwpt_{k+1}(\Tiling)$ is completely analogous.
\end{proof}
It is clear that \cref{lemma:bpt_wpt_bij} gives the desired bijection, finishing the proof of \cref{prop:bpt_wpt}.
\end{proof}

We return to the study of the polytope $\Sigma_{\A,k}+\Sigma_{\A,k-1}$. We say that a \bpt $(k,n)$-plabic graph $G$ is \emph{$\A$-regular} if it can be obtained as $\Gbpt_k(\Tiling)$ for some regular fine zonotopal tiling $\Tiling$ of $\ZV$. We similarly define $\A$-regular \wpt $(k+1,n)$-plabic graphs, and clearly the bijection of \cref{prop:bpt_wpt} restricts to such plabic graphs. Observe also that by \cref{thm:flips_moves}, applying the moves (M1) and (M2) to a \bpt $(k,n)$-plabic graph $G$ corresponds to applying the moves (M2) and (M3) to the unique $(k+1,n)$ \wpt plabic graph  linked with $G$. The proof of the following result is analogous to that of \cref{thm:regular}.

\begin{corollary} Let $d=3$ and $\A\subset\R^2$ be the configuration of vertices of a convex $n$-gon.
\begin{theoremlist}
\item The vertices of $\Sigma_{\A,k}+\Sigma_{\A,k-1}$ are in bijection with $\A$-regular \bpt $(k,n)$-plabic graphs, as well as with $\A$-regular \wpt $(k+1,n)$-plabic graphs.
\item The edges  of $\Sigma_{\A,k}+\Sigma_{\A,k-1}$ correspond to the moves (M1) and (M2) of $\A$-regular \bpt $(k,n)$-plabic graphs,  as well as to the moves (M2) and (M3) of $\A$-regular \wpt $(k+1,n)$-plabic graphs.
\end{theoremlist}
\end{corollary}

\def\X{\bm{X}}
\section{Applications to soliton graphs}

In this section we start by explaining how tropical hypersurfaces are dual to regular subdivisions
of a related zonotope, see 
	\cref{def:soliton}.  We then explain how, when $d=3$, we can  recover
 the construction of \emph{soliton graphs}---contour plots of soliton solutions of the KP equation (see \cref{cor:plabicsoliton} and \cref{def:solitongraph})---and in particular, recover the fact that 
	they are realizations of reduced plabic graphs.  
	We 
conclude with applications of our previous results to soliton graphs.

\subsection{Tropical hypersurfaces and regular zonotopal tilings}
\begin{definition}
	A \emph{tropical polynomial} is a function $F: \R^{d-1} \to \R$ that can be expressed
as the \emph{tropical sum} of a finite number of \emph{tropical monomials}. 
More precisely, if we let $\X$ denote $(X_1,\dots,X_{d-1})$, then a tropical
polynomial $F$ is the maximum
	$$F = \max_{I\in\mathcal{B}} F_I(X_1,\dots, X_{d-1}) = 
	    \max_{I\in \mathcal{B}} F_I(\X)$$ of a
        finite set $\{F_I| I \in \mathcal{B}\}$
        of linear functionals\footnote{In tropical geometry one typically
  uses integer or rational coefficients, because these coefficients come
        from valuations of power series, but in this paper everything
will make sense for real coefficients.}
	$F_I: \R^{d-1} \to\R$.
The \emph{tropical hypersurface} $V(F)$ is the set of points in
	$\R^{d-1}$ where $F$ is non-differentiable.  Equivalently, $V(F)$ is
the set of points where the maximum among the terms of $F$ is achieved
at least twice.
\end{definition}
Note that $V(F)$ is a codimension-one piecewise-linear subset of $\R^{d-1}$.  Moreover,
the complement of $V(F)$ is a collection of (top-dimensional) regions
of $\R^{d-1}$, where each region $R=R(I)$ is naturally associated to
some $I \in \mathcal{B}$; more specifically, we have that
$F_{I}(\X) > F_J(\X)$ for all points
$\X = (X_1,\dots,X_{d-1})\in R(I)$ for all $J \neq I$.

\def\FL{\mathcal{F}}
\def\FLKH{\FL(V_{k,\h})}
\def\Vkh{V_{k,\h}}

\def\Xtq{{\widetilde\X}_q}

We now look at some particularly nice examples of tropical hypersurfaces. Fix positive numbers $n, d$ and $k$, and let $\A=(\a_1,\dots,\a_n)$ be a point configuration in $\R^{d-1}$ as before.

\begin{definition}\label{def:soliton} 
Let $\h\in\R^n$. For $1\leq i \leq n$, define a linear functional
	$f_{i,\h}: \R^{d-1} \to \R$ by 
	\begin{equation}\label{functional}
          f_{i,\h}(\X):=\<\X,\a_i\>+h_i,\quad\text{equivalently,}\quad	f_{i,\h}(X_1,\dots, X_{d-1}) = a_{i,1} X_1 + \dots + a_{i,{d-1}} X_{d-1} + h_i.
	\end{equation}
 For $I \in {[n] \choose k}$, let $F_{I,\h} = \sum_{i\in I} f_{i,\h}.$

We consider the tropical polynomial
	\begin{equation}\label{eq:troppoly}
		F_{k,\h}(\X)=	
		\max_{I \in {[n] \choose k}} F_{I,\h}(\X),
	\end{equation}
and define  $\Vkh$  to be
the tropical hypersurface $V(F_{k,\h})$.  
We denote by $\FLKH\subset {[n]\choose k}$ the collection of all sets $I \in {[n] \choose k}$ that appear as a face labels of regions in the complement of $\Vkh$.
\end{definition}

Recall from \cref{notation} that for a point configuration $\A\subset \R^{d-1}$, $\ZV$ denotes the zonotope associated with the lift $\VC\subset\R^d$ of $\A$. Recall also that each generic height vector $\h\in\GHV$ determines a regular fine zonotopal tiling $\Tiling_\h$ of $\ZV$, and that its set of vertex labels is denoted by $\Vert(\Tiling_\h)\subset 2^{[n]}$, see~\eqref{eq:dfn_Vert}.

\begin{proposition}\label{prop:soliton} 
Let $\A$ and $\VC$ be as above, and let $\h\in\GHV$ be a generic height vector. Then 
\[\FLKH=\Vert(\Tiling_\h)\cap {[n]\choose k}.\]
\end{proposition}
\begin{proof}
Let $\Vt=(\vt_1,\dots,\vt_n)$ be the lift of $\V$ to $\R^{d+1}$ given by $\vt_i:=(\v,h_i)$. Let $I\in{[n]\choose k}$. By \cref{rmk:upper}, $I\in\Vert(\Tiling_\h)$ if and only if $\vt_I:=\sum_{i\in I} \vt_i$ belongs to the upper boundary of 
	$\ZVt$. Equivalently, there exists a vector $\Xtq:=(\X,q,1)\in\R^{d+1}$ (for some $\X\in\R^{d-1}$ and $q\in\R$) such that the dot product with $\Xtq$ is maximized over $\ZVt$ at $\vt_I$. Since $\ZVt=\sum_{i\in[n]} [0,\vt_i]$, we see that this happens precisely when $\<\Xtq,\vt_i\>$ is positive for $i\in I$ and negative for $i\notin I$. Note that $\<\Xtq,\vt_i\>=\<\X,\a_i\>+q+h_i=f_{i,\h}(\X)+q$. We have shown that 
	$I\in\Vert(\Tiling_\h)$ 
	if and only if there exist $\X\in\R^{d-1}$ and $q\in\R$ such that for all $i\in I$ and $j\notin I$, we have $f_{i,\h}(\X)+q>0> f_{j,\h}(\X)+q$. The latter condition can be restated as: 
	there exists $\X\in\R^{d-1}$ such that 
	for all $i\in I$ and $j\notin I$, we have
	$f_{i,\h}(\X)> f_{j,\h}(\X)$, which is equivalent to  
		$F_{I,\h}(\X) > F_{J,\h}(\X)$ for all $J \neq I$.  
	Therefore a $k$-element subset $I$ lies in $\Vert(\Tiling_\h)$ 
if and only 
	$I\in\FLKH$. 
\end{proof}

\def\FA{F_A}
\def\FM{F_{\Mcal}} 

\subsection{Soliton graphs}\label{sec:soliton}

In the case that $d=3$, we recover the \emph{soliton graphs} which were
studied in \cite{KW, KW2} in order to study soliton solutions to the KP equation.
We briefly review that construction here.

The KP equation 
\[
\frac{\partial}{\partial x}\left(-4\frac{\partial u}{\partial t}+6u\frac{\partial u}{\partial x}+\frac{\partial^3u}{\partial x^3}\right)+3\frac{\partial^2 u}{\partial y^2}=0
\]
was proposed by Kadomtsev and Petviashvili in 1970 \cite{KP70}, in order to
study the stability of the soliton solutions of the Korteweg-de Vries (KdV) equation
under the influence of weak transverse perturbations.
The KP equation can be also used to describe  two-dimensional shallow
water wave phenomena (see for example \cite{K10}).   This equation is now considered to be a prototype of
an integrable nonlinear partial differential equation.

Let $\t = (t_3, t_4,\dots, t_n)$ be a vector of ``higher times" (often one
sets $t_4 = \dots = t_n = 0$ and $t_3 = t$, but it will be convenient for us to 
use the higher times.)
There is a well known recipe
(see \cite{H04, CK1})
for using a point $A$ in the real Grassmannian $\Gr(k,n)$ together with 
$n$ real parameters $\kappa_1 < \dots < \kappa_n$ to construct a \emph{$\tau$-function} $\tau_A(x,y,\t)$,   such that a simple transformation of it
\begin{equation*}
u_A(x,y,\t)=2\frac{\partial^2}{\partial x^2}\ln\tau_A(x,y,\t)
\end{equation*}
 is a soliton solution of the KP equation.

 The $\tau$-function is defined as follows.
For $i\in[n]$, set 
\[h_i:=\kappa_i^3 t_3 + \dots + \kappa_i^n t_n\quad\text{and}\quad E_i(x,y,\t):=\exp 
(\kappa_i x + \kappa_i^2 y + h_i).\]
For $I=\{i_1<\dots<i_k\} \in {[n] \choose k}$,  set
\begin{equation}
 K_I:=\prod_{\ell<m}(\kappa_{i_m}-\kappa_{i_{\ell}})\quad\text{and}\quad E_I(x,y,\t):=K_I\cdot E_{i_1}\cdots E_{i_k}.\label{eq:K_E_kappa}
\end{equation}

For $A\in\Gr(k,n)$, we define 
 \begin{equation}\label{tau}
\tau_A(x,y,\t)=\sum_{I\in\binom{[n]}{k}}\Pl_I(A)\,E_I(x,y,\t),
\end{equation}
where $\Pl_I(A)$ is the Pl\"ucker coordinate of $A\in\Gr(k,n)$ indexed by $I$ as before.

If one is interested in the behavior of the soliton solutions when the variables
$(x,y,\t)$ are on a large scale, then, as in \cite[Section 4.2]{KW2}, it is natural to 
 rescale the variables with a small positive number $\epsilon$,
\[
x~\longrightarrow ~\frac{x}{\epsilon},\qquad y~\longrightarrow~\frac{y}{\epsilon},\qquad
\t~\longrightarrow~\frac{\t}{\epsilon}, 
\]
which leads to
\[
\tau_A^{\epsilon}(x,y,\t)
=\sum_{I\in \mathcal{M}}
\exp\left(\frac{1}{\epsilon}\,\sum_{j=1}^k(\kappa_{i_j}x+\kappa_{i_j}^2y+h_{i_j})+\ln(\Pl_I(A)K_I)\right),
\]
where $\mathcal{M} = \mathcal{M}(A) := \{I \ \mid \ \Pl_I(A) \neq 0\} \subset {[n] \choose k}$ and $I=\{i_1<\dots< i_k\}$. 
Then we define a function $\FA(x,y,\t)$ as the limit
\begin{equation}\label{eq:soliton}
\FA(x,y,\t)=\displaystyle{\lim_{\epsilon\to 0}\epsilon\ln\left(\tau^{\epsilon}_A(x,y,\t)\right)}\\[1.5ex]
	=  \underset{I\in \mathcal{M}}\max \left\{
                     \sum_{j=1}^k (\kappa_{i_j} x + \kappa^2_{i_j} y +h_{i_j})\right\}.
\end{equation}
Since the above function depends only on the collection $\mathcal{M}$,
we also denote it as $\FM(x,y,\t)$.

\begin{definition}[\cite{KW, KW2}]\label{contour}
	Fix $\t = (t_3,\dots, t_n) \in \R^{n-2}$. 
  Given a solution $u_A(x,y,\t)$ of the KP equation as above,
	we define its \emph{(asymptotic) contour plot} 
	$\mathcal{C}_{\t}(\mathcal{M})$ to be the
	set of all $(x,y)\in\mathbb{R}^2$ where $\FM(x,y,\t)$ is not linear.  
\end{definition}

The contour plot approximates the locus where the corresponding solution 
of the KP equation has its peaks, and we label each region in 
the complement of 
 $\mathcal{C}_{\t}(\mathcal{M})$ by the $k$-element subset 
 $I$ which achieves the maximum in \eqref{eq:soliton}.

\begin{remark}\label{rem:soliton}
Comparing~\eqref{eq:soliton} with \cref{def:soliton} in the case that $\mathcal{M}={[n]\choose k}$, we see that $\FM(x,y,\t)$ 
is a tropical polynomial for $d=3$, $f_i(x,y) = \kappa_i x + \kappa_i^2 y + h_i$
for $i\in[n]$, and the asymptotic contour plot
$\mathcal{C}_{\t}(\mathcal{M})$ is the tropical hypersurface
$V_k(f_1,\dots,f_n)$.
\end{remark} 

Let $d=3$ and $\mathcal{A} = \{\a_1,\dots,\a_n\}$ for
	$\a_i = (\kappa_i,\kappa_i^2)$. Consider its lift $\VC$ and the zonotope $\ZV\subset\R^3$ as in \cref{notation}. Denote $\h:=(h_1,\dots,h_n)$ where $h_i = \kappa_i^3 t_3 + \dots + \kappa_i^n t_n$, and recall that $\Tiling_\h$ is the regular zonotopal tiling of $\ZV$ induced by $\h$. Applying \cref{prop:soliton} to these contour plots, we obtain the following result.

\begin{corollary}\label{cor:plabicsoliton}
Assume that $\mathcal{M} = {[n] \choose k}$ and $I= \{i_1,\dots,i_k\}\in\Mcal$. Then there exists a point $(x,y)\in\R^2$ lying in the region of the complement of $\mathcal{C}_{\t}(\mathcal{M})$ labeled by $I$ if and only if $\v_{i_1} + \dots + \v_{i_k}$ is a vertex of $\Tiling_\h$.
\end{corollary}

Note that \cref{cor:plabicsoliton} is closely related to the discussion in 
\cite[Section 2.3]{KK}.

\begin{definition}[\cite{KW, KW2}]\label{def:solitongraph}
We associate a \emph{soliton graph} $G_{\t}(\mathcal{M})$
to each contour plot $C_{\t}(\mathcal{M})$
by marking any intersection of three line segments by either a white or black vertex,
depending on whether there is a unique line segment directed from the vertex
towards $y \to \infty$ or a unique line segment directed from the vertex 
towards $y \to -\infty$ (it is impossible for a line segment to be parallel
to the $x$-axis).  
\end{definition}

When $\mathcal{M} = {[n] \choose k}$, 
and for	generic times $\t = (t_3,\dots,t_n)$, all intersections of 
line segments are trivalent intersections, and 
by \cite[Corollary 10.9]{KW2},  
the graph $G_{\t}(\mathcal{M})$
is a $(k,n)$-plabic graph, 
see \cref{fig:soliton}.
\cref{cor:plabicsoliton} then says the following (for $\A\subset\R^2$ as above).
\begin{corollary}\label{cor2:plabicsoliton}
Each soliton graph 
$G_{\t}(\mathcal{M})$
associated to $\mathcal{M} = {[n] \choose k}$ is a trivalent  
$\mathcal{A}$-regular $(k,n)$-plabic graph.
\end{corollary}

\cref{fig:soliton} shows the contour plot associated to the positive Grassmannian
$\Grtp(2,6)$; each region is labeled by an element $I=\{i_1,i_2\} \in {[6] \choose 2}$
which indicates that in that region, $F_I(x,y) = f_{i_1}(x,y)+f_{i_2}(x,y) > F_J(x,y)$ 
for all other $J\in {[6] \choose 2}$.  The trivalent intersections of line
segments are marked by white or black vertices as in \cref{def:solitongraph}.

It is natural to ask how the soliton graph (plabic graph) changes when the higher times 
$\t = (t_3,\dots,t_n)$ evolve. In \cite{KW}, the authors speculated (cf. \cref{fig:mutation}) 
 that the face labels of the soliton graph should change via \emph{cluster transformations}, or in 
other words, via moves (M1)--(M3) of the plabic graph from \cref{fig:3moves}. This is now a consequence of \cref{thm:regular}.

\begin{corollary}
Fix $\mathcal{A}$ and $\mathcal{M}$ as in \cref{cor:plabicsoliton}, and consider the 
associated soliton graphs 
$G_{\t}(\mathcal{M})$.  Then as the higher times $\t=(t_3,\dots,t_n)$ evolve,  
$G_{\t}(\mathcal{M})$ changes via the moves 
from \cref{fig:3moves}.  In particular the face labels change via square moves.
\end{corollary}
\begin{proof}
Changing the higher times continuously corresponds to changing the heights continuously, which by 
\cref{thm:regular} 
corresponds to walking around the normal fan of 
$\Sigma_{\A, k}+\Sigma_{\A, k-1} + \Sigma_{\A, k-2}$.
\end{proof}

\begin{figure} 
 \begin{center}
   \includegraphics[height=1.5in]{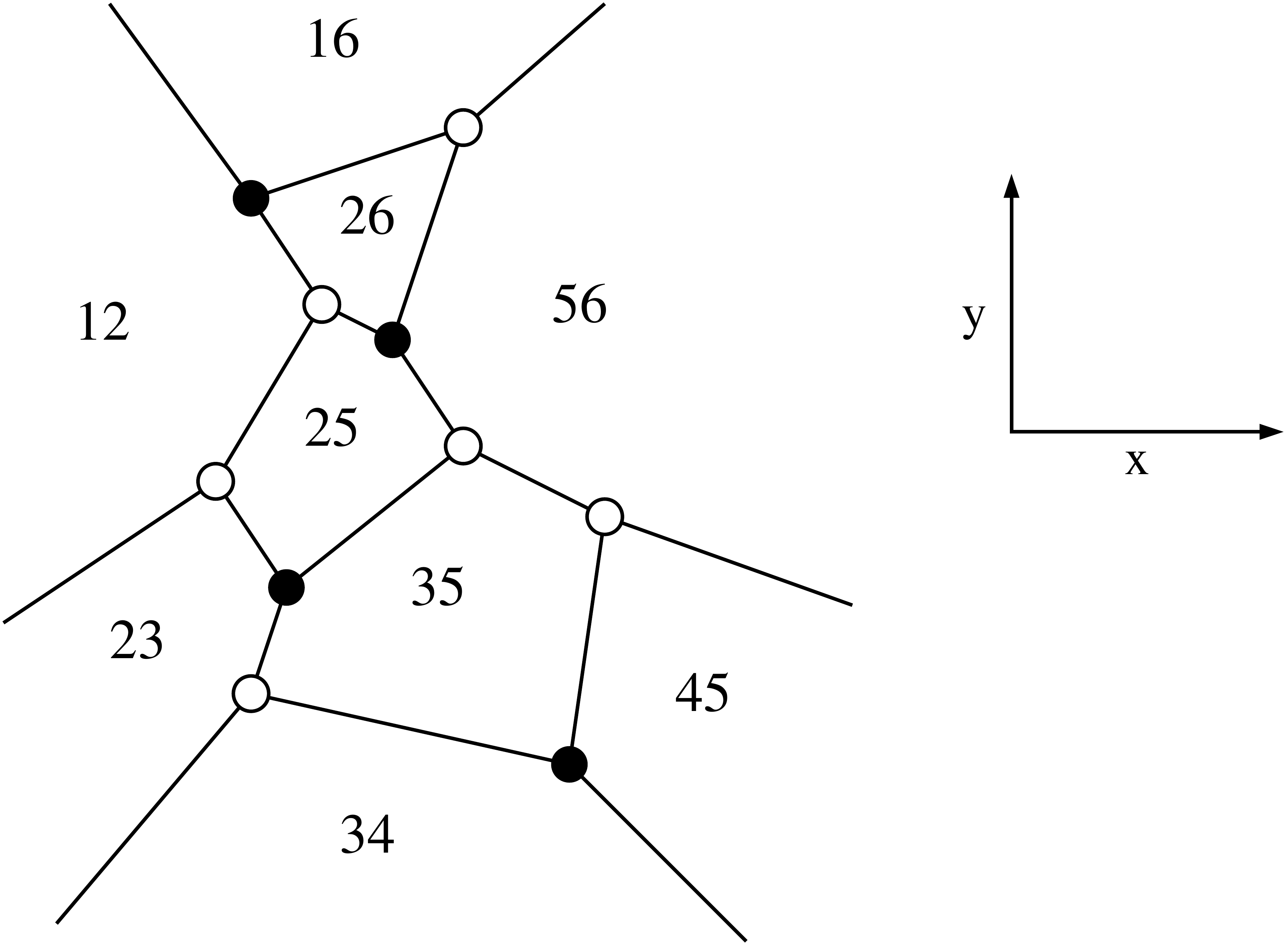}
	 \caption{
	A {soliton graph} $G_{\t}(\mathcal{M})$  coming from
	 $\Gr(2,6)$.}
	 \label{fig:soliton}
\end{center}
\end{figure}

In \cite[Theorem 8.5 and Theorem 8.9]{KW2}, the authors classified the 
contour plots $\mathcal{C}_{\t}(\mathcal{M})$ obtained when 
$\t = (t_3,0,\dots,0)$ and $t_3 \to \pm \infty$.  We can now give a generalization of their
results (cf. \cref{cor:time_positive}) in the case that $\mathcal{M} = {[n] \choose k}$ and the $\kappa_i$'s are positive. Let us write $\t\geq0$ if $t_i\geq0$ for $i=3,\dots,n$.

\def\hij{{h_i^{(j)}}}
\def\hj{{\h^{(j)}}}

\begin{proposition}\label{prop:time_positive}
	Assume that $\mathcal{M} = {[n] \choose k}$, the numbers $\kappa_1<\dots<\kappa_n$ are positive, and that the vector $\t\geq0$ is nonzero. Then $\mathcal{C}_{\t}(\mathcal{M})$ can be identified with the plabic graph associated to the Le-diagram, and its regions are labeled by the elements of $\Compatw$ for $w=w_0$  as in \cref{ex:extreme}. Similarly, the regions of $\mathcal{C}_{-\t}(\mathcal{M})$ are labeled by the elements of $\Compatw$ for $w=\id$.
\end{proposition}
\begin{proof} Recall that $\v_i = (\kappa_i, \kappa_i^2, 1)$ and $h_i=\kappa_i^3t_3+\dots+\kappa_i^n t_n$. Our goal is to show that $h_n\gg h_{n-1} \gg\dots \gg h_1$ in the sense of~\eqref{eq:hw}. In other words, we need to show that $\muh(a,b,c,d)<0$ for all $1\leq a<b<c<d\leq n$. For $3\leq j\leq n$, let $\hj\in\R^n$ be given by $\hij:=\kappa_i^j t_j$, thus $\h=\sum_{j=3}^n\hj$. It suffices to show $\mu_{\hj}(a,b,c,d)<0$. It follows from~\eqref{eq:circuit2} and~\eqref{eq:alphaC} that 
  \[\mu_{\hj}(a,b,c,d)=-\det \left(\begin{smallmatrix}
1&1&1&1\\
\kappa_a&\kappa_b&\kappa_c&\kappa_d\\
\kappa_a^2&\kappa_b^2&\kappa_c^2&\kappa_d^2\\
t_j\kappa_a^j&t_j\kappa_b^j&t_j\kappa_c^j&t_j\kappa_d^j\\
        \end{smallmatrix}\right)=-t_j\cdot K_{\{a,b,c,d\}}\cdot s_{(j-3)}(\kappa_a, \kappa_b,\kappa_c,\kappa_d),\] 
where $K_{\{a,b,c,d\}}$ was defined in~\eqref{eq:K_E_kappa} and $s_\lambda$ is the Schur polynomial associated with a partition $\lambda=(\lambda_1,\dots,\lambda_m)$, see~\cite[\S7.15]{EC2}. Thus $s_{(j-3)}=h_{j-3}$ is the \emph{complete homogeneous symmetric polynomial}~\cite[\S7.5]{EC2}. Since $\kappa_1<\dots<\kappa_n$, we find $K_{\{a,b,c,d\}}>0$. Since we have also assumed that $\kappa_1,\dots,\kappa_n>0$, we find $s_{(j-3)}(\kappa_a, \kappa_b,\kappa_c,\kappa_d)>0$. We have shown $\mu_{\hj}(a,b,c,d)<0$ for all $j$ such that $t_j>0$, which implies $\muh(a,b,c,d)<0$. For the case of $\mathcal{C}_{-\t}(\mathcal{M})$, the same argument shows  $\muh(a,b,c,d)>0$.
\end{proof}

In \cref{prop:time_positive}, we required the $\kappa$-parameters to be positive. For the case $\t = (t_3,0,\dots,0)$ studied in~\cite{KW2}, this assumption can be lifted.
\begin{corollary}\label{cor:time_positive}
\Cref{prop:time_positive} still holds when the numbers $\kappa_1<\dots<\kappa_n$ are not necessarily positive, provided that $\t = (t_3,0,\dots,0)$ with $t_3>0$.
\end{corollary}
\begin{proof}
Indeed, in this case the polynomial $s_{(j-3)}=s_{(0)}$ from the proof of \cref{prop:time_positive} is equal to $1$, thus we have $\muh(a,b,c,d)<0$ regardless of the sign of the $\kappa$-parameters.
\end{proof}

Since the generic soliton graphs 
 $G_{\t}(\mathcal{M})$
for $\mathcal{M} = {[n] \choose k}$ are trivalent $(k,n)$-plabic graphs, it is natural to ask which $(k,n)$-plabic graphs
are realizable as soliton graphs. Similarly to \cref{sec:high-assoc-plab}, let us say that a bipartite $(k,n)$-plabic graph is \emph{realizable} if it can be obtained from some $G_{\t}(\mathcal{M})$ by contracting unicolored edges. Thus every realizable $(k,n)$-plabic graph is also $\A$-regular for some $\A$. (It is not clear to us whether the converse is true.) In \cite{KW, KW2}, the authors showed that all bipartite $(2,n)$-plabic graphs are realizable.  In \cite{KK}, building on work of \cite{Huang}, Karpman and Kodama showed that for $k=3$ and $n=6,7,8$, every bipartite $(k,n)$-plabic graph  is realizable for \emph{some} choice of $\kappa$- and $\t$-parameters (see however \cref{ex:hexagons} and~\cite[Theorem~4.2]{KK}).

\newcommand{\arxiv}[1]{\href{https://arxiv.org/abs/#1}{\textup{\texttt{arXiv:#1}}}}

\bibliographystyle{alpha}
\bibliography{biblio}

\end{document}